\documentclass{amsart}
\usepackage{amssymb}
\usepackage{amsfonts}
\usepackage{amssymb}
\usepackage{amsmath}
\usepackage{amsthm}
\usepackage{enumerate}
\usepackage{tabularx}
\usepackage{centernot}
\usepackage{mathtools}
\usepackage{stmaryrd}
\usepackage{amsthm,amssymb}
\usepackage{etoolbox,color}
\usepackage{tikz}
\usepackage{amssymb}
\usetikzlibrary{matrix}
\usepackage{tikz-cd}
\usepackage{tikz}
\usepackage{marginnote}
\definecolor{mygray}{gray}{0.85}
\usepackage[linecolor=black, backgroundcolor=mygray,colorinlistoftodos,prependcaption,textsize=small]{todonotes}
\usepackage{xargs}  
\usepackage[colorlinks,citecolor=blue,urlcolor=blue, linkcolor=blue]{hyperref}

\renewcommand{\leq}{\leqslant}
\renewcommand{\geq}{\geqslant}

\renewcommand{\trianglelefteq}{\trianglelefteqslant}
\newcommand{\mrm}[1]{\mathrm{#1}}

\usepackage[backgroundcolor=mygray,colorinlistoftodos,prependcaption,textsize=small]{todonotes}
\usepackage{xcolor}

\makeatletter
\def\subsection{\@startsection{subsection}{3}%
  \z@{.5\linespacing\@plus.7\linespacing}{.3\linespacing}%
  {\bfseries\centering}}
\makeatother

\makeatletter
\def\subsubsection{\@startsection{subsubsection}{3}%
  \z@{.5\linespacing\@plus.7\linespacing}{.3\linespacing}%
  {\centering}}
\makeatother

\makeatletter
\def\myfnt{\ifx\protect\@typeset@protect\expandafter\footnote\else\expandafter\@gobble\fi}
\makeatother

\newtheorem{theorem}{Theorem}[section]

\newtheorem{convention}[theorem]{Conventions}
\newtheorem{corollary}[theorem]{Corollary}

\newtheorem{lemma}[theorem]{Lemma}

\theoremstyle{plain}

\theoremstyle{definition}

\newtheorem{fact}[theorem]{Fact}

\newtheorem{definition}[theorem]{Definition}
\newtheorem{remark}[theorem]{Remark}

\newtheorem{notation}[theorem]{Notation}

\newtheorem*{theorem1.1}{Theorem 1.1}
\newtheorem*{theorem1.2}{Theorem 1.2}
\newcounter{claimcounter}
\newenvironment{claim}{\stepcounter{claimcounter}{\noindent {\bf Claim \theclaimcounter.}}}{}
\newenvironment{claimproof}[1]{\noindent{{\em Proof.}}\space#1}{\hfill $\rule{0.40em}{0.40em}$}

\usepackage{booktabs} 
\usepackage{hyperref} 

\newcommand{\pointw}{{P(W)}}
\newcommand{\pointwone}{{P(W')}}
\newcommand{\tranw}{{T(W)}}
\newcommand{\tranwone}{{T(W')}}
\newcommand{\lattw}{{L(W)}}
\newcommand{\lattwone}{{L(W')}}

\newcommand{\N}{{\mathbb{N}}}
\newcommand{\Z}{{\mathbb{Z}}}
\newcommand{\Q}{{\mathbb{Q}}}
\newcommand{\R}{{\mathbb{R}}}

\newcommand{\iso}{\mathrm{Iso}}
\newcommand{\aff}{\mathrm{Aff}}
\newcommand{\gl}{\mathrm{GL}}

\newcommand{\coalpha}{\alpha^{\vee}}

\newcommand{\tshort}{{\mathbb{Z}^n\overset{t}{\rightarrow}W\overset{p}{\rightarrow}W_0}}

\newcommand{\cspace}{~}
\newcommand{\say}[1]{``#1''}

\usetikzlibrary{arrows.meta}
\newcommand{\nodecircle}[2]{\node[draw, circle, fill=white, thick, minimum size=5pt, inner sep=0pt, label=below:{\scriptsize #1}] at #2 {};}
\newcommand{\nodercircle}[2]{\node[draw, circle, fill=white, thick, minimum size=5pt, inner sep=0pt, label=right:{\scriptsize #1}] at #2 {};}

\newcommand{\verteq}{\rotatebox{90}{$\,=$}}

\begin{document}

\begin{abstract}
We prove that every finite direct product of crystallographic groups arising from an irreducible root system (in the sense of Lie theory) is profinitely rigid (equiv. first-order rigid). This is a generalization of recent proofs of profinite rigidity of affine Coxeter \mbox{groups \cite{paolini_simon_homog, corson-hughes-moller-varghese, paolini-sklinos}. Our proof uses model theory.}
\end{abstract}

\title[Profinite rigidity of crystallographic groups from Lie theory]{Profinite rigidity of crystallographic groups arising from Lie theory}
\thanks{Research of the second named author was  supported by project PRIN 2022 ``Models, sets and classifications", prot. 2022TECZJA, and by INdAM Project 2024 (Consolidator grant) ``Groups, Crystals and Classifications''.}

\author{Davide Carolillo}

\address{Department of Mathematics ``Giuseppe Peano'', University of Torino, Via Carlo Alberto 10, 10123, Italy.}
\email{davide.carolillo@unito.it}

\author{Gianluca Paolini}

\address{Department of Mathematics ``Giuseppe Peano'', University of Torino, Via Carlo Alberto 10, 10123, Italy.}
\email{gianluca.paolini@unito.it}

\date{\today}
\maketitle

\tableofcontents



\section{Introduction}

A finitely generated residually finite group $G$ is said to be profinitely rigid if for any finitely generated residually finite group $H$ we have that $\widehat{G} \cong \widehat{H}$ implies that $G \cong H$, where $\widehat{G}$ denotes profinite completion.
In the current literature, the problem of profinite rigidity of finitely generated groups has become central to group theory, motivated by the following major open question posed by Remeslennikov in \cite[Question~5.48]{MK10}: is a non-abelian free group profinitely rigid? The problem remains open but much progress has been made on profinite rigidity in recent years.

\smallskip
Motivated by these developments, the problem of profinite rigidity of Coxeter groups has been considered in \cite{profinite_Coxeter, varghese}, with \cite{varghese} focusing specifically on affine Coxeter groups and posing the question of their profinite rigidity. Now, by classical works of Oger \cite{Oger88}, the problem of profinite rigidity for affine Coxeter groups is equivalent to a model-theoretic question, i.e., that of first-order rigidity, which asks whether such groups are, up to isomorphism, the only finitely generated models of their first-order theory.
This led to a model-theoretic solution \cite{paolini-sklinos} to the problem posed in \cite{varghese} due to the second named author of this paper and R. Sklinos.
 A purely group theoretic proof of this result was given in \cite{corson-hughes-moller-varghese}. Yet another proof (also model-theoretic) of profinite rigidity of affine Coxeter groups appears in \cite{paolini_simon_homog}.
	
\smallskip
The present paper extends and fully leverages the technology introduced in \cite{paolini-sklinos} toward a proof of profinite rigidity of affine Coxeter groups, broadening its scope of application, with a particular focus on applications to Lie theory and root systems. 

	
\smallskip	
We now introduce the main object of interest to this paper, namely {\em crystallographic groups}. A crystallographic group is a group $G$ fitting into a short exact sequence: $1 \rightarrow T \rightarrow W \rightarrow W_0 \rightarrow 1$, with $T \cong \mathbb{Z}^n$ and $W_0$ finite. The group $T$ is called the {\em translation subgroup} of $W$, and $W_0$ is called the {\em point group} of $W$. By Bieberbach's First Theorem (and its strengtehing due to Zassenhaus cf. \cite[Theorem\cspace3.2 and 3.3]{hiller}) these groups correspond exactly to the discrete cocompact subgroups of the isometry group of the Euclidean space $\mathbb{E}^n$ containing $n$ linearly independent translations. These groups have been studied since at least the 19th century, in fact they also appear in Hilbert's 18th problem. This problem specifically asked whether there are only finitely many ``essentially different'' crystallographic groups in $n$-dimensional Euclidean spaces. Here, ``essentially different'' can be defined as isomorphism of abstract groups, or equivalently, by Bieberbach's Second Theorem (cf. \cite[Theorem\cspace3.4]{hiller}), up to conjugation by affine motions of $\mathbb{E}^n$.
	
\smallskip 	This connects with affine Coxeter groups, as irreducible affine Coxeter groups are crystallographic groups of a certain kind. In fact, they have several additional properties of interest, namely:
	\begin{enumerate}[$\bullet$]
	\item they are split (that is, the sequence $1 \rightarrow T \rightarrow W \rightarrow W_0 \rightarrow 1$ splits);
	\item their associated integral representations are absolutely irreducible;
	\item they arise from a root system (in the sense of Lie theory).
	\end{enumerate}
	
	In this paper, we will see that the methods used in \cite{paolini-sklinos} toward proving the profinite rigidity of affine Coxeter groups have a much greater scope of applicability. In particular, we prove two major profinite rigidity results.
	 The first one is:

	\begin{theorem}\label{main_th1} Finite direct products of absolutely irreducible split crystallographic groups are profinitely rigid (equiv. first-order rigid).
	\end{theorem}
	
	Regarding our second main result, following \cite{martinais}, we say that a crystallographic group \emph{arises from an irreducible root system} if it admits an affine realization as a group of isometries of an Euclidean space (cf. Definition\cspace\ref{def - crystallographic group arising from a root system}) such that the associated point group is essential and it is the \mbox{Weyl group of an irreducible root system.}
	
	\begin{theorem}\label{main_th2} Finite direct products of crystallographic groups arising from an irreducible root system are profinitely rigid (equiv. first-order rigid).
	\end{theorem}
	
	Notice that these two theorems generalize the results in \cite{corson-hughes-moller-varghese, paolini-sklinos} in various different directions. On the one hand, Theorem~\ref{main_th1} simply assumes absolute irreducibility of the integral representation associated to the crystallographic group $W$, without asking for specific properties of any of its affine realizations. On the other hand, Theorem~\ref{main_th2} does ask that the group can be realized as the group of symmetries of a root system (as in the case of affine Coxeter groups), but it generalizes \cite{corson-hughes-moller-varghese, paolini-sklinos} twofold, firstly, in considering other root lattices (not only the ones associated to the affine Coxeter groups), and, most importantly, in considering {\em any} group extension associated to any such lattice, not only the split ones. Our proof crucially relies on a combination of integral representation methods from \cite{plesken-holt} and the ``Crystallography of Coxeter Groups'', as developed by Maxwell and Martinais in \cite{martinais, maxwell}.
	
\smallskip 
Notice that {\em some} assumptions on the given crystallographic group are necessary in order to conclude profinite rigidity, as e.g. for every integer $n$ such that the class number of the cyclotomic field $\mathbb{Q}(\zeta_n)$ is strictly greater than $1$ (this is true for every $n\geq 85$), there exist split crystallographic groups $G_1,G_2$ of dimension $\phi(n)$ such that $G_1\not\cong G_2$ but $\widehat{G}_1 \cong \widehat{G}_2$ (see \cite{brigham1971}), where $\phi(n)$ is Euler's function. In particular, for any $p\geq 23$, there are such groups of the form $\mathbb{Z}^{p-1} \rtimes\mathbb{Z}/p\mathbb{Z}$ (see \cite[p.\cspace204-205]{finken-neubuser-plesken}). Similarly, there are known examples of non-isomorphic non-split crystallographic groups of the same genus (i.e., elementarily equivalent) which have isomorphic translation lattices (cf. \cite[p.\cspace205]{finken-neubuser-plesken}). On the other hand, notice that it is know that all crystallographic groups of dimension $\leq 4$ are profinitely rigid \cite{crystallo_profinite}.

What we find particularly interesting about our approach is that our proof uses model theory, i.e., in both Theorem~\ref{main_th1} and Theorem~\ref{main_th2} we actually prove that the groups under consideration are first-order rigid, and then deduce, via the already mentioned fundamental work of Oger \cite{Oger88}, that they are profinitely rigid.
\section{Preliminaries on crystallographic groups}\label{prel_sec}

In this section, we introduce the basics of crystallographic groups.



\begin{convention}\label{semi_convention}
	If $A,B,C$ are groups, then $A = B \rtimes_\alpha C$ denotes the external semidirect product of $B$ and $C$. In particular, the action of $C$ on $B$ is given by the image of the homomorphism $\alpha: C \rightarrow \mrm{Aut}(B)$. On the other hand, the notation $A = B \rtimes C$ is used for the internal semidirect product perspective. In the latter case, it is implicitly intended that $C$ acts by conjugation on $B$.
\end{convention}
Let $V$ be an $n$-dimensional real vector space, and let $(\,\cdot\,,\,\cdot\,)$ be a positive definite symmetric bilinear form on $V$ (that is, a non-degenerate inner product).
Then, the associated Euclidean vector space $E\vcentcolon=(V,(\,\cdot\,,\,\cdot\,))$ carries a natural notion of length given by the norm $\|\,\cdot\,\|\vcentcolon=(\,\cdot\,,\,\cdot\,)^{1/2}$. An (affine) \emph{isometry} of $E$ is a map $f:V\rightarrow V$ preserving this norm, i.e., a function such that $\|f(x)-f(y)\|=\|x-y\|$ for all $x,y\in V$. 
%
%
As well-known, the group $\iso(E)$ of the isometries of $E$ has a natural decomposition as:
$$	\iso(E)=V\rtimes O(E),$$
where $V$ corresponds to the translations $t_v:V\rightarrow V$, $x\mapsto x+v$, for a given vector $v\in V$, and $O(E)$ denotes the subgroup of orthogonal maps of $E$, i.e., the automorphisms of $V$ preserving the inner product.

\smallskip
We consider $E$ also as a topological space, with a topology compatible with the norm.
\begin{definition}{\cite[Definition\cspace3.1]{hiller}}\label{def - affine crysp group}
	Let $E$ be a finite-dimensional real Euclidean vector space. Then, an \emph{affine crystallographic group} $W$ is a subgroup of $\iso(E)$ whose action on $E$ is:
	\begin{enumerate}[(1)]
		\item \emph{discrete},
		i.e., for each $x\in E$, the orbit $W.x\subseteq E$ has no accumulation points;
		\item \emph{cocompact}, i.e., $E/W$ is compact with respect to the quotient topology.
	\end{enumerate} 
\end{definition}
Observe that condition (2) from \ref{def - affine crysp group} is equivalent to saying that $W$ has a fundamental domain with compact topological closure in $E$ (cf. \cite[Section\cspace3]{hiller}).

By the identification $\iso(E)=V\rtimes O(E)$, any affine crystallographic group $W\leq\iso(E)$ canonically determines two associated objects, namely:
\begin{enumerate}[$\bullet$]
	\item $T(W)\vcentcolon=W\cap (V\times\{1\})$, the \emph{translation subgroup} of $W$;
	\item $P(W)\vcentcolon=\{s\in O(E):\exists v\in V((v,s)\in W)\}$, the \emph{point group} of $W$.
\end{enumerate}

Definition~\ref{def - affine crysp group} entails that $T(W)$ is a \emph{lattice} in the vector space $V\times \{1\}$ of translations in $\iso(E)$, i.e., it is the free abelian group generated by some basis of $V\times\{1\}$. Thus, the rank of $T(W)$ as a free abelian group coincides with the dimension of $V$, and it is called the \emph{dimension} of $W$. Since $T(W)$ is a normal subgroup of $W$, it is stable under the action of $P(W)$, which forms a discrete subgroup of the compact group $O(E)$, and therefore must be finite. It follows that any affine crystallographic group $W$ fits into a short-exact sequence of groups:
\begin{equation*}
	1 \rightarrow T(W) \rightarrow W \rightarrow P(W) \rightarrow 1
\end{equation*}
\noindent
such that $T(W)$ is free abelian of finite rank (i.e., $T(W)\cong \Z^n$, for some $n\in\N$), $P(W)$ is finite, and it acts faithfully on $T(W)$. This observation is part of \say{Bieberbach's First Theorem} (see e.g. \cite[Theorem~3.2]{hiller}).

 By a fundamental result of Zassenhaus (cf. \cite{zassenhaus}, or \cite[Theorem\cspace3.3]{hiller}), this description actually suffices to provide an abstract characterization of crystallographic groups.
\begin{theorem}[Zassenhaus]\label{thm - zassenhaus}
	An abstract group $W$ is isomorphic to an affine crystallographic  group of dimension $n$ if and only if $W$ contains a free abelian subgroup $T$ of rank $n$ that is normal, of finite index, and maximal abelian.
\end{theorem}
Theorem~\ref{thm - zassenhaus} allows to redefine crystallographic groups as follows:

	

\begin{definition}\label{def - abstract crysp group}
	A group $W$ is said to be an \emph{(abstract) crystallographic group} if it admits a short-exact sequence of groups:
	\begin{equation*}
		1 \rightarrow \tshort \rightarrow 1
	\end{equation*}
	such that $W_0$ is finite and it acts faithfully on $T=t(\Z^n)$ via the map $\alpha:W_0\rightarrow\mathrm{Aut}(T)$, $w\mapsto (\,\cdot\,)^u$, for some (or, equivalently, any) $u\in p^{-1}(w)$.
\end{definition}

In light of Theorem\cspace\ref{thm - zassenhaus}, we can freely pass from one perspective (the affine one) to the other (the abstract one).
In this spirit, we refer to $T$ and $W_0$ as the \emph{translation subgroup} and \emph{point group} of $W$, respectively, as in the affine case.

By definition, associated to each crystallographic group there is a faithful action $\alpha: W_0 \rightarrow\mathrm{Aut}(T)$ of its point group $W_0$ on the translation subgroup $T$. Since $T \cong \mathbb{Z}^n$, this action is actually an integral representation of $W_0$, i.e., $\alpha: W_0 \rightarrowtail \mrm{GL}_n(\mathbb{Z})$. Accordingly, $T$ acquires a natural structure of a $\Z[W_0]$-lattice, which we denote by $L(W)$ and refer to as the \emph{translation lattice} of $W$. If $W_0=\{w_i:i<k\}$ is an enumeration without repetition of $W_0$, the scalar multiplication on $L(W)$ is defined as follows:
\begin{equation*}
	\sum_{i<k}n_iw_i.x\vcentcolon=\sum_{i < k}n_i\alpha(w_i)(x)=\sum_{i < k}n_i x ^{u_i},
\end{equation*}
for all $n_0,\ldots,n_{k-1}\in \Z$, $x\in T$ and some (or, equivalently, any) $u_i\in p^{-1}(w_i)$, with $i<k$ and $p:W_0\rightarrow\mathrm{Aut}(T)$ as in \ref{def - abstract crysp group}. In what follows, we often pass from $W$ to the associated lattice $L(W)$ without explicit mention, and will therefore use the terminology of representation theory when referring to $W$.
\begin{notation}
	The term \say{translation lattice} is sometimes used in the literature to denote the group structure on $T$. For clarity, in this paper we always refer to the group $T$ as the \emph{translation subgroup} of $W$, and to the associated $\mathbb{Z}[W_0]$-module $\lattw$ described above as the \emph{translation lattice} of $W$. Moreover, if $W$ is an abstract crystallographic group as in \ref{def - abstract crysp group}, we write $\tranw$ and $\pointw$ in place of $T$ and $W_0$ when we wish to emphasize the dependence on $W$; the maps in the associated short-exact sequence will always be clear from the context.
\end{notation}
An abstract crystallographic group $W$ can have many affine realizations inside a given Euclidean vector space $E$. Thus, it is natural to ask whether these realizations are equivalent, in the affine sense. The answer to this question is provided by the following result, knowns as \emph{Bieberbach's Second Theorem} (see \cite[Theorem\cspace3.4]{hiller}).


\begin{theorem}[Bieberbach\cspace II]\label{thm - bieberbach II}
	Two abstract crystallographic groups are isomorphic if and only if there is an Euclidean space $E$ such that their affine realizations in $\aff(E)$ are conjugated by some affine motion in $E$.
\end{theorem}

\section{Model theory of crystallographic groups}
\begin{definition}
	We say that two groups $W$ and $H$ are \emph{elementarily equivalent}, and we write $W\equiv H$, if they satisfy the same first-order sentences in the usual language of groups $L_{gp}=\{\,\cdot\,,(\,\cdot\,)^{-1},1\}$.
\end{definition}

Given a group $G$, we denote by $\widehat{G}$ its \emph{profinite completion}, i.e., the inverse limit of its finite quotients. As is well known, if $G$ is \emph{residually finite}, then the canonical map from $G$ to $\widehat{G}$ is an embedding. In particular, any finitely generated abelian-by-finite group is residually finite. The following result from \cite{Oger88} will play a crucial role in the remainder of the paper.


\begin{fact}[Oger]\label{Oger}
	Let $G,H$ be finitely generated abelian-by-finite groups. Then, $G\equiv H$ if and only if $\widehat{G}\cong \widehat{H}$. Furthermore, the canonical embedding of $G$ into $\widehat{G}$ is an elementary embedding.
\end{fact}
\begin{lemma}\label{lemma_the_prop_def_transla_crysp} 
	Let $W$ be a crystallographic group with translation group $T$ and point group $W_0$. Then $T$ is $\emptyset$-definable in $W$. More precisely, if $W_0$ has order $k$, then $T$ is definable by the first-order formula $\phi(x)\equiv \forall y ([x,y^k]=e)$.
\end{lemma}
\begin{proof}
	Let $p: W \rightarrow W_0$ be the canonical projection as in \ref{def - abstract crysp group}, and let $\phi(x)$ be as in the statement of the lemma. Denote by $\phi(W)$ the set of realizations of $\phi(x)$ in $W$. We claim that $T=\phi(W)$. The inclusion $[\subseteq]$ follows directly from the fact that $T$ is abelian, and $W/T\cong W_0$ has finite order $k$. Indeed, for each $w\in W$, in $W/T$ we have $w^kT=(wT)^k=T$, showing that $w^k \in T$.
	
	\smallskip\noindent
	Conversely, for the inclusion $[\supseteq]$, we argue by contradiction. Suppose there exists some $w \in W\setminus T$ such that $W\models \phi(w)$. Then, for each $t\in T$, we have:
	\begin{equation*}
		w(kt)w^{-1}(kt)^{-1}=[w,kt]=e.
	\end{equation*}
	\noindent
	It follows that:
	\begin{equation*}
		(kt)^w = kt.
	\end{equation*}
	Since $T$ is torsion-free, this implies that $t^w=t$, for all $t\in T$. In particular, by the definition of the action $\alpha:W_0\rightarrow\mathrm{Aut}(T)$, we have:
	\begin{equation*}
		\alpha(p(w))(t)=t^w=t,
	\end{equation*}
	\noindent
	for all $t\in T$. However, by the faithfulness of $\alpha$, $\alpha(p(w))$ is the identity on $T$ if and only if $p(w)= 1_{W_0}$, i.e., if $w\in T$. This contradicts the assumption $w\in W\setminus T$, completing the proof.
\end{proof}

We briefly recall a few notions and facts concerning $\mathbb{Z}[W_0]$-lattices, which we will need below. For further details, see \cite{plesken-holt}.

\begin{definition}[{\cite[Definition\cspace4.1.6, p.\cspace94]{plesken-holt}}]\label{same_genus}
	\begin{enumerate}[(i)]
		\item Let $W$ and $W'$ be crystallographic groups with translations subgroups $T(W)$, and $T(W')$, respectively. Then, we say that $W$ and $W'$ belong to the \emph{same genus} if $W/mT(W)\cong W'/mT(W')$ for all $m\in\N$.
		\item Let $W_0$ be a finite group and $L, L'$ be two $\mathbb{Z}[W_0]$-lattices. We say that $L$ and $L'$ belong to the \emph{same genus} (as $\mathbb{Z}[W_0]$-modules) if for every prime number $p$ and $k \in \mathbb{N}$ we have that $L/p^kL \cong L'/p^kL'$ as $\mathbb{Z}[W_0]$-modules.
	\end{enumerate}
\end{definition}
\begin{fact}[{\cite[Exercise\cspace1, p.\cspace96]{plesken-holt}}]\label{Exercise1}
	Let $W$ and $W'$ be crystallographic groups belonging to the same genus. Suppose that $P(W), P(W')$ are the point groups, and $L(W), L(W')$ are the translation lattices, of $W$ and $W'$, respectively. Then, after identifying their (isomorphic) point groups with a group $W_0$, $L(W)$ and $L(W')$ lie in the same genus as $\Z[W_0]$-lattices.     
\end{fact}
\begin{fact}[{\cite[Theorem\cspace4.1.8]{plesken-holt}}]\label{plesken_fact} Let $W_0$ be a finite group and $L, L'$ be two $\Z[W_0]$-lattices belonging to the same genus. Then, for every $m \in \mathbb{N}$ there is an injective $\mathbb{Z}[W_0]$-module homomorphism $\sigma: L' \rightarrow L$ such that the index $[L : \sigma\,( L')]$ is coprime~to~$m$.
\end{fact}
\begin{lemma}\label{lemma - every f.g. group el. equivalent to a crysp gp is crysp}
	Let $W$ be a crystallographic group, and $W'$ be a finitely generated group elementarily equivalent to $W$. Then, $W'$ is a crystallographic group with translation subgroup $\tranwone$ and point group $\pointwone$ such that the following hold:
	\begin{enumerate}[(1)]
	\item $\tranw\cong\tranwone$;
	\item $\pointw\cong\pointwone$;
	\item after identifying the (isomorphic) point groups with a group $W_0$, $L(W)$ and $L(W')$ lie in the same genus as $\Z[W_0]$-lattices.
\end{enumerate}	
\end{lemma}
\begin{proof}
	By Lemma\cspace\ref{lemma_the_prop_def_transla_crysp}, $\tranw$ is $\emptyset$-definable in $W$ by a first-order formula $\phi(x)$. Since $W\equiv W'$, it follows that $\phi(W')$ is normal in $W'$, and $W'/\phi(W') \cong P(W)$. Moreover, $\phi(W')$ is finitely generated, as it is normal of finite index in $W'$, which is finitely generated by assumption. Finally, $\phi(W) $ and $\phi(W')$ are elementarily equivalent, and hence $\phi(W) \cong \phi(W')$, since finitely generated free abelian groups are well known to be first-order rigid.
	\noindent
	
	\smallskip\noindent
	We show that $W'/\phi(W')$ acts faithfully on $\phi(W')$ via the map:
	\begin{equation*}
		\alpha':W'/\tranwone\rightarrow \mathrm{Aut}(\tranwone)\quad\text{such that}\quad w'\tranwone\mapsto(\,\cdot\,)^{u'},
	\end{equation*}
	\noindent
	for some (or, equivalently, any) $u'\in w'\tranw$. This map is a well-defined group homomorphism, since $\phi(W')$ is abelian and normal in $W'$.
	
	\smallskip\noindent
	Now, consider an element $w\in W$. By the faithfulness of the action of $\pointw$ on $\tranw$, the conjugation $(\,\cdot\,)^w$ restricts to the identity on $\tranw$ (i.e., $w$ commutes with every element of $\tranw$) if and only if $w\in\tranw$. Hence, we have the following:
	\begin{equation*}
		W \models \forall x\left(\forall y(\phi(y)\rightarrow[x,y]=1)\leftrightarrow\phi(x)\right).
	\end{equation*} 
	\noindent
	It follows that $\alpha':W'/\tranwone\rightarrow \mathrm{Aut}(\tranwone)$ is injective, since $W'$ is elementarily equivalent to $W$ by assumption. Therefore, by Definition\cspace\ref{def - abstract crysp group}, $W'$ is crystallographic with translation subgroup $\tranwone=\phi(W')$ and point group $\pointwone=W'/\tranwone\cong \pointw$. This completes the proof of (1) and (2).
	
	\smallskip\noindent
	Concerning (3), let $T,L$ and $T',L'$ denote the translation subgroups and translation lattices of $W$ and $W'$, respectively. Identifying their isomorphic point groups with a finite group $W_0$, $L$ and $L'$ are naturally $\Z[W_0]$-lattices (cf. item (2) of this lemma). We claim that $L$ and $L'$ belong to the same genus as $\Z[W_0]$-lattices. Indeed, by Lemma~\ref{lemma_the_prop_def_transla_crysp}, there exists a $\emptyset$-definable formula whose solution set in $W$ (respectively $W'$) is the translation lattice $L$ (respectively $L'$). Thus, the conditions $W/mL\cong W'/mL'$, for $m\in\N$, are first-order expressible. Since $W'\equiv W$, it follows that $W$ and $W'$ belong to the same genus. Therefore, by Fact \ref{Exercise1}, $L, L'$ belong to the same genus as $\mathbb{Z}[W_0]$-lattices.
\end{proof}
\begin{theorem}\label{thm - finitely many centerings -> split profinite rigidity}
	Let $W$ be a crystallographic group with translation subgroup $\tranw$ and point group $\pointw$, and let $W'$ be a finitely generated group elementarily equivalent to $W$. Suppose that there are finitely many subgroups $N_0,\ldots,N_{k-1}$ {of $\tranw$ such that:}
	\begin{enumerate}[(a)]
		\item $N_i\trianglelefteq W$, and $[\tranw:N_i]<\infty$ for all $i<k$;
		\item every subgroup $N$ of $\tranw$ that is normal in $W$ is a multiple $\ell N_i$ of $N_i$, for some $i<k$ and $\ell\in\N$.
	\end{enumerate}
	\noindent
	Then, the following holds:
	\begin{enumerate}[(1)]
		\item identifying $\pointw$ and $\pointw'$ with $W_0$, $\lattw\cong\lattwone$ as $\Z[W_0]$-lattices;
		\item if in addition $W$ is split, then $W \cong W'$ (so $W$ is profinitely rigid).
	\end{enumerate}
\end{theorem}
\begin{proof}
	Item (2) follows from (1) and a direct argument, see e.g. \cite[Section~3]{paolini-sklinos}. We prove item (1): the argument is essentially as in \cite[Theorem~3.17]{paolini-sklinos}. By Lemma\cspace\ref{lemma - every f.g. group el. equivalent to a crysp gp is crysp}, $W'$ is also crystallographic, with $\tranw\cong\tranwone$ and $\pointw\cong\pointwone$. For ease of notation, let $T,L$ and $T',L'$ the translation subgroups and translation lattices of $W$ and $W'$, respectively. Identifying $P(W)$ and $P(W')$ with a common finite group $W_0$, again by \ref{lemma - every f.g. group el. equivalent to a crysp gp is crysp}, we have that $L, L'$ belong to the same genus \mbox{as $\mathbb{Z}[W_0]$-lattices.}
	
	\smallskip\noindent
	We now show that, under the additional assumptions of the present theorem, $L$ and $L'$ are isomorphic as $\mathbb{Z}[W_0]$-lattices. By assumption ($a$), the order $m_i\vcentcolon=[T:N_i]$ is finite, for all $i<k$. Hence, there exists a non-negative integer $m$ such that:
	\begin{equation*}
		m=\prod_{1<k}m_i.
	\end{equation*}
	\noindent
	Since $L$ and $L'$ belong to the same genus as $\Z[W_0]$-lattices, it follows from Fact\cspace\ref{plesken_fact} that there exists an injective $\Z[W_0]$-module homomorphism $\sigma: L' \rightarrow L$ such that the index $[T: \sigma (T')]$ is coprime to $m$. In particular, $\sigma(L')$ is a submodule of $L$, i.e., $\sigma(T')$ is stable under the action of the point group of $W$. Therefore, $\sigma(T')$ is a normal subgroup of $W$. By assumption ($b$), this implies that there are some $\ell\in\N$ and $j<k$ such that $\sigma(T') = \ell N_j$. We distinguish two cases.
	\begin{itemize}
		\item[\underline{Case 1}:] If $N_j=T$, then $\sigma(T')=\ell T$ and $\sigma:L'\rightarrow \ell L$ is an isomorphism of $\Z[W_0]$-lattices. Observe that if $\ell=0$ the thesis is immediate, as both $L$ and $L'$ are trivial. Otherwise, for $\ell\neq 0$, the function:
		\begin{equation*}
			\tau_\ell:L\rightarrow \ell L, \quad x\mapsto \ell x,
		\end{equation*}
		\noindent
		is naturally an isomorphism of $\Z[W_0]$-lattices, since $T$ is supposed to be torsion-free and the action of $W_0$ on $T$ is linear. Therefore, the composition $\sigma^{-1}\circ\tau_\ell$ witnesses that $L$ and $L'$ are isomorphic $\Z[W_0]$-lattices. 
		
		\item[\underline{Case 2}:] If $N_j\neq L$, then $[T:N_j] = m_j>1$. However, since:
		\begin{equation*}
			[T:\sigma(T')]=[T:\ell N_j]=[T:N_j]\cdot [N_j:\ell N_j]=m_j\ell^n,
		\end{equation*}
		\noindent
		this leads to a contradiction, as $[T:\sigma(T')]$ was supposed to be \mbox{coprime to $m$.}
	\end{itemize}
\end{proof}

\section{Centering conditions}\label{sec4}
In this section, we follow \cite{plesken-pohst} and show how to adapt it to our context.
\begin{definition}\label{def - centering of a Z[G]-repr module}
	Let $W_0$ be a finite group and $L$ be a $\Z[W_0]$-lattice. Then, a \emph{centering} of $L$ is a submodule $C$ of $L$ of finite index.
\end{definition}
A subgroup $C$ of a free abelian group $T$ of finite rank $n$ is itself free abelian of rank $m\leq n$, with the equality $m=n$ realized if and only if $C$ has finite index in $T$. Therefore, for any finite group $W_0$, we can equivalently characterize the centerings of a $\Z[W_0]$-lattice $L$ as the sublattices of maximal rank in $L$.

\smallskip 
Let $W_0$ be a finite group. There is a canonical correspondence assigning to each $\Z[W_0]$-lattice $L$ an integral representation $\alpha_L : W_0 \to \mathrm{GL}_n(\Z)$, where $n = \mathrm{rank}_\mathbb{Z}(L)$. Under this correspondence, two $\Z[W_0]$-lattices $L$ and $L'$ are isomorphic if and only if the associated representations $\alpha_L$ and $\alpha_{L'}$ are equivalent, i.e., if they are conjugate in $\mrm{GL}_n(\mathbb{Z})$. Moreover, $W_0$ acts faithfully on $L$ if and only if the corresponding representation $\alpha_L$ is faithful. We will often pass from one perspective to the other.

\smallskip 
Any integral representation $\alpha_L: W_0 \to \mathrm{GL}_n(\Z)$ naturally extends to a rational representation via the inclusion of $\mathrm{GL}_n(\Z)$ in $\mathrm{GL}_n(\Q)$. It is thus natural to associate to the $\Z[W_0]$-lattice $L$ a $\Q[W_0]$-module $\Q L$ that extends $L$ by allowing $\Q$-scalar multiplications. The rational representation associated to $L$ can then be viewed as arising from the action of $W_0$ on $\Q L$. Formally, we define:
\begin{definition}{\cite[\S\phantom{.}73, Ch.\phantom{.}XI]{representation_theory_book}}\label{def.-QL}
	Let $W_0$ be a finite group and $L$ be a $\Z[W_0]$-lattice. Suppose $W_0=\{w_i:i<k\}$ is an enumeration without repetition of $W_0$. Then, $\Q L$ is the $\Q[W_0]$-module with domain $\Q\otimes_\Z L$, and $\Q[W_0]$-scalar multiplication:
	\begin{equation*}
		r.(h\otimes x)=\sum_{i < k}q_ih\otimes w_i.x,
	\end{equation*}
	\noindent
	for all $r\in \Q [W_0]$, $x\in L$, and $h,q_0,\ldots,q_{k-1}\in\Q$ such that $r=\sum_{i<k}q_iw_i$.
\end{definition}
\begin{fact}{\cite[\S\phantom{.}73, Ch.\phantom{.}XI]{representation_theory_book}}\label{remark - QL}
	In the context of Definition\cspace\ref{def.-QL}, the $\Q[W_0]$-module $\Q L$ naturally inherits the structure of a $\Q$-vector space, where the $\Q$-scalar multiplication is defined by:
	\begin{equation*}
		q.(h\otimes x)=q1.(h\otimes x)=q h\otimes x,
	\end{equation*}
	\noindent
	for all $q,h\in\Q$, and $x\in L$. If $\{x_i:i<n\}$ is a free abelian basis of $L$, then $\{1\otimes x_i:i<n\}$ forms a basis of $\Q L$ as a $\Q$-vector space. In particular, the dimension of $\Q L$ over $\Q$ equals the rank of $L$ as a free abelian group.
\end{fact}
\begin{definition}{\cite[Section\phantom{.}2]{plesken-pohst}}\label{def. Z/Q-equivalence}
	Let $W_0$ be a finite group, and $L,L'$ be $\Z[W_0]$-lattices. Then, we say that $L$ and $L'$ are:
	\begin{enumerate}[(1)]
		\item $\Z$\emph{-equivalent}, and we write $L\sim_\Z L'$, if\, $L\cong L'$ (as $\Z[W_0]$-lattices);
		\item $\Q$\emph{-equivalent}, and we write $L\sim_\Q L'$, if\, $\Q L \cong \Q L'$ (as $\Q[W_0]$-modules).
	\end{enumerate}
\end{definition}
As mentioned above, in the language of representation theory, item\cspace$(1)$ in \ref{def. Z/Q-equivalence} corresponds to saying that the associated integral representations $\alpha_L$ and $\alpha_{L'}$ are equivalent, i.e., that there exists an invertible \mbox{matrix $A\in\mathrm{GL}_n(\Z)$ such that:}
\begin{equation*}
	\alpha_L(w)=A\alpha_{L'}(w)A^{-1},
\end{equation*}
\noindent
for all $w\in W_0$ (cf.\cspace\cite[Definition\cspace73.2, Ch.\phantom{.}XI]{representation_theory_book}). Similarly, item\cspace$(2)$ in \ref{def. Z/Q-equivalence} is equivalent to the representations $\alpha_L$ and $\alpha_{L'}$ being conjugate in $\mathrm{GL}_n(\Q)$ (cf. \cite[Definition\cspace73.1, Ch.\phantom{.}XI]{representation_theory_book}).

\smallskip
Clearly, $\Z$-equivalence implies $\Q$-equivalence, but the converse does not generally hold (cf. \cite[\S\phantom{.}73, Ch.\phantom{.}XI]{representation_theory_book}).
For example, consider the case where $W_0=\Z_2 = \mathbb{Z}/2\mathbb{Z}$, and let $L,L'$ be the $\Z[W_0]$-lattices on $\Z^2$ induced by the representations:
\begin{equation*}
	\alpha:\bar{1}\mapsto\begin{pmatrix}
		1&0\\
		0&-1
	\end{pmatrix}\quad\text{and}\quad \alpha':\bar{1}\mapsto \begin{pmatrix}
		1 & 1\\
		0 & -1
	\end{pmatrix},
\end{equation*}
\noindent
respectively. Then $L\sim_\Q L'$, but $L\not\sim_\Z L'$, as the matrices $\alpha(\bar{1})$ and $\alpha'(\bar{1})$ are conjugate in $\mathrm{GL}_2(\Q)$, yet not in $\mathrm{GL}_2(\Z)$.
\begin{lemma}{\cite[Section\phantom{.}2]{plesken-pohst}}\label{lemma every centering is Q-equivalent to the ambient module}
	Let $W_0$ be a finite group and $L$ be a $\Z[W_0]$-lattice. Then, every centering $C$ of $L$ is $\Q$-equivalent to $L$.
\end{lemma}
\begin{proof}
	Since $C$ is a submodule of $L$, the inclusion $C \subseteq L$ induces a natural embedding of $\Q[W_0]$-modules $\delta: \Q C \rightarrowtail \Q L$. We claim that this embedding is in fact an isomorphism.
	
	\smallskip\noindent
	By assumption, $C$ has finite index in $L$, and hence it is a free abelian group of the same rank as $L$. Consequently, by Remark\cspace\ref{remark - QL}, the $\Q$-vector spaces on $\Q C$ and $\Q L$ have the same finite dimension. Since the map $\delta$ is clearly $\Q$-linear and injective, it follows that $\delta$ is also surjective. We conclude that $\Q C \cong \Q L$, and therefore 
	$C \sim_\Q L$, as required.
\end{proof}
By the preceding discussion, for a fixed finite group $W_0$, each $\Q$-equivalence class of $\Z[W_0]$-lattices decomposes into a disjoint union of $\Z$-equivalence classes. Lemma\cspace\ref{lemma-there are only finitely many Z-equivalence classes}, which is an instantiation of the Jordan-Zassenhaus Theorem (cf.\cspace\cite[Theorem~79.1, Ch.\phantom{.}XI]{representation_theory_book}) in the context of $\Z[W_0]$-lattices, guarantees that this decomposition yields only finitely many distinct $\Z$-equivalence classes. We will need the following basic fact from \cite{plesken-pohst}.
\begin{fact}{\cite[Section\phantom{.}2]{plesken-pohst}}\label{fact-every_Q-equivalent_lattice_is_Z-equivalent_to_a_center}
	Let $W_0$ be a finite group, and $L,L'$ be $\Z[W_0]$-lattices. If $L\sim_\Q L'$, then there exists a centering $C$ of $L$ such that $C\sim_\Z L'$.
\end{fact}
\begin{lemma}\label{lemma-there are only finitely many Z-equivalence classes}
	Let $W_0$ be a finite group, and $L$ be a $\Z[W_0]$-lattice. Then, the $\Q$-equivalence class of $L$ splits into finitely many disjoint $\Z$-equivalence classes.
\end{lemma}
\begin{proof}
	By the Jordan-Zassenhaus Theorem (see~\cite[\S\phantom{.}79,\phantom{.}Ch.\phantom{.}XI]{representation_theory_book}), \mbox{the set of all} $\Z[W_0]$-lattices $L'$ such that $L' \subseteq \Q L$ and $L' \sim_\Q L$ decomposes into finitely many $\Z$-equivalence classes. Thus, it suffices to observe that, by Fact\cspace\ref{fact-every_Q-equivalent_lattice_is_Z-equivalent_to_a_center}, every $\Z[W_0]$-lattice $\Q$-equivalent to $L$ is isomorphic to some centering $C$ of $L$. Since, by Lemma\cspace\ref{lemma every centering is Q-equivalent to the ambient module}, every such centering is itself $\Q$-equivalent to $L$, the claim follows immediately.
\end{proof}

A $\Z[W_0]$-lattice $L$ always admits infinitely many centerings. Indeed, if $C$ is a centering of $L$, then the infinite descending chain
\begin{equation*}
	C\geq 2C\geq 3C\geq\cdots
\end{equation*}
\noindent
consists entirely of centerings of $L$, since, for every $k\in\N^+$, the map $x\mapsto kx$ defines an isomorphism $C\rightarrow kC$. This observation implicitly suggests the following notion, already considered in \cite[Section\phantom{.}2]{plesken-pohst}.
\begin{definition}\label{def. - order centerings}
	Let $W_0$ be a finite group, and let $L$ be a $\Z[W_0]$-lattice. We define:
	\begin{enumerate}[(1)]
		\item $\mathcal{C}(L)$ to be the set of all centerings of $L$;
		\item a partial order $\prec$ on $\mathcal{C}(L)$, given by
		\begin{equation*}
			C\prec C' \quad\text{if and only if}\quad \exists \lambda\in\Z\quad\text{such that}\quad C=\lambda C',
		\end{equation*}
		\noindent
		for all $C,C'\in\mathcal{C}(L)$.
	\end{enumerate}
\end{definition} 
\begin{fact}{\cite[Section\phantom{.}2]{plesken-pohst}}\label{lemma - properties order centerings}
	Let $W_0$ be a finite group and $L$ be a $\Z[W_0]$-lattice. Then, the following hold:
	\begin{enumerate}[(1)]
		\item for all $C,C'\in\mathcal{C}(L)$, $C\sim_\Z C'$ whenever $C\prec C'$;
		\item for each $C\in\mathcal{C}(L)$, there exists a unique $\prec$-maximal centering $\overline{C}\in\mathcal{C}(L)$ such that $C\prec \overline{C}$.
	\end{enumerate}
\end{fact}

	It may happen that two distinct $\prec$-maximal centerings $C$ and $C'$ in $\mathcal{C}(L)$ are $\Z$-equivalent while remaining $\prec$-incomparable. Consequently, the number of $\Z$-equivalence classes (which is finite by \ref{lemma-there are only finitely many Z-equivalence classes}) does not necessarily bound the number of $\prec$-maximal centerings of a given $\Z[W_0]$-lattice. We will see that this phenomenon does not occur when the $\Z[W_0]$-lattice $L$ is in addition an \emph{absolutely irreducible} module. 
\smallskip\noindent
First, we fix some module-theoretic terminology.
\begin{definition}
	Let $Q$ be a ring (with unity). \mbox{Then a $Q$-module $M$ is said to be:}
	\begin{enumerate}[(1)]
		\item \emph{irreducible} $($or \emph{simple}$)$, if $M$ admits only trivial submodules, i.e., for every submodule $N$ of $M$ either $N=\{0\}$ or $N=M$;
		\item \emph{decomposable} if there exist two non-trivial submodules $N_0$ and $N_1$ of $M$ such that $M=N_0\oplus N_1$;
		\item \emph{indecomposable} if it is not decomposable.
	\end{enumerate}
\end{definition}

Clearly, if a $Q$-module is irreducible, then it is also indecomposable. In the case of $Q$ being the group algebra $K[W_0]$ of a finite group $W_0$ over a field $K$ of characteristic $0$, the converse is also true, as a consequence of a classic result by Maschke (cf. \cite[Proposition\phantom{.}2, Ch.\phantom{.}V, Annexe]{bourbaki} or \cite[Theorem\phantom{.}10.8, Ch.\phantom{.}II]{representation_theory_book}).

\begin{definition}{\cite[\S\phantom{.}29, Ch.\phantom{.}IV]{representation_theory_book}}\label{def - absolutely irreducible module}
	Let $W_0$ be a finite group and $M$ be a $\Q[W_0]$-module. Suppose that $W_0=\{w_i:i<k\}$ is an enumeration {without repetition of $W_0$.} 
	\begin{enumerate}[(1)]
		\item For every field extension $K$ of $\Q$, we define the $K[W_0]$-module $K M$ as follows\footnote{Various notations are used in the literature: this $K[W_0]$-module is denoted $M^K$ in \cite[\S\phantom{.}29, Ch.\phantom{.}IV]{representation_theory_book}, and $M_{(K)}$ in \cite[\S\phantom{.}8.1, Ch.\phantom{.}II]{bourbaki algèbre}. Here,
		we adopt $K M$ to preserve consistency with Definition\cspace\ref{def.-QL}.}:
		\begin{enumerate}[(i)]
			\item the underlying abelian group is $K\otimes_\Q M$, where the tensor product is computed with respect to the natural $\Q$-vector space structure on $M$;
			\item the $K[W_0]$-scalar multiplication is given by
			\begin{equation*}
				r.\left(\sum_{j<m}b_j\otimes x_j\right)=\sum_{\overset{i<k}{j<m}}a_ib_j\otimes w_i.x_j,
			\end{equation*}
			\noindent
			for all $r\in K [W_0]$, $x_0,\ldots,x_{m-1}\in M$, and $a_0,\ldots,a_{k-1},b_0,\ldots,b_{m-1}\in K$ such that $r=\sum_{i<n}a_iw_i$.
		\end{enumerate}
		\item The module $M$ is said to be \emph{absolutely irreducible} if $K M$ is irreducible for every field extension $K$ of $\Q$.
	\end{enumerate} 
\end{definition}
\begin{definition}\label{def - absolutely irreducible lattice}
	Let $W_0$ be a finite group. A $\Z[W_0]$-lattice $L$ is said to be \emph{absolutely irreducible} if $\Q L$ is absolutely irreducible as a $\Q[W_0]$-module. In particular, for every field extension $K$ of $\Q$, the tensor product defining $K L$ is computed with respect to the $\Q$-vector space structure on $\Q L$ from Fact\cspace\ref{remark - QL}. A crystallographic group $W$ with point group $W_0$ and translation lattice $L$ is said to be \emph{absolutely irreducible} if $L$ is absolutely irreducible (as a $\Z[W_0]$-lattice).
\end{definition}
\noindent 
\begin{lemma}\label{thm. - absolutely irreducible -> there are only finitely many maximal centerings}
	Let $W_0$ be a finite group and $L$ be an absolutely irreducible $\Z[W_0]$-lattice. Then there are only finitely many maximal centerings of $L$.
\end{lemma}
\begin{proof}
	 By \cite[Theorem\cspace2.1]{plesken-pohst}, the $\prec$-maximal centerings of $L$ form a complete set of representatives for the $\Z$-equivalence classes contained in the $\Q$-equivalence class of $L$. Moreover, by Lemma\cspace\ref{lemma-there are only finitely many Z-equivalence classes}, there are only finitely many such $\Z$-equivalence classes. Therefore, it suffices to show that any two distinct $\prec$-maximal centerings $\overline{C}$ and $\overline{C}'$ of $L$ cannot be $\Z$-equivalent. The case $L=\{0\}$ is trivial. If $L\neq\{0\}$, suppose for contradiction that there exists an isomorphism of $\Z[W_0]$-modules $\sigma:\overline{C}\rightarrow\overline{C}'$. Identifying $\Q\overline{C}$ and $\Q\overline{C}'$ with $\Q L$ via the isomorphisms induced by the inclusions $\overline{C},\overline{C}'\leq L$, this map linearly extends to a $\Q[W_0]$-module endomorphism $\bar{\sigma}$ of $\Q L$.
 	
 	\smallskip\noindent
 	Since $\Q L$ is absolutely irreducible, it follows from \cite[Theorem\cspace29.13, Ch.\phantom{.}IV]{representation_theory_book} that $\bar{\sigma}$ is a scalar multiplication by some rational number. In particular, this means that there exist $\alpha\in\Z$ and $\beta\in\Z^\times$ such that $\beta\sigma(x)=\alpha x$ for all $x\in \overline{C}$. Hence:
	\begin{equation*}
		\alpha \overline{C}=\beta\sigma(\overline{C})=\beta\overline{C}'.
	\end{equation*}
	\noindent
	Let $C\vcentcolon=\beta\overline{C}'$. Clearly, $C$ is a centering satisfying $C\prec\overline{C}'$, by definition. Since $C$ has maximal rank, we must have $\alpha\neq0$, as otherwise $L$ would be trivial. Therefore, the identity $C=\alpha\overline{C}$ witnesses that $C\prec\overline{C}$. However, this contradicts Fact\cspace\ref{lemma - properties order centerings}{(2)}, as $C\prec \overline{C}'$ and $\overline{C},\overline{C}'$ are assumed to be distinct.
\end{proof}
We can now restate Theorem\cspace\ref{thm - finitely many centerings -> split profinite rigidity} in this terminology as follows:
\begin{theorem}\label{theorem}
	Let $W$ be a crystallographic group with translation subgroup $\tranw$ and point group $\pointw$. Suppose that $W'$ is a finitely generated group elementarily equivalent to $W$. If $\lattw$ admits only finitely many $\prec$-maximal centerings, then:
	\begin{enumerate}[(1)]
		\item after identifying $\pointw$ and $\pointw'$ with a common finite group $W_0$, $\lattw\cong\lattwone$ as $\Z[W_0]$-lattices;
		\item if, in addition, $W$ is split, then $W \cong W'$. In particular, $W$ is profinitely rigid.
	\end{enumerate}
\end{theorem}

    \begin{corollary}\label{cor_asb_irr_rigid}  Absolutely irreducible split crystallographic groups are profinitely rigid (equiv. first-order rigid).
\end{corollary}

    \begin{proof} This is by \ref{thm. - absolutely irreducible -> there are only finitely many maximal centerings} and \ref{theorem}.
    \end{proof}
\begin{fact}{\cite[Theorem\cspace2.2]{oger-lasserre}}\label{fact - lasserre}
	Let $G$ and $H$ be elementarily equivalent polycyclic-by-finite groups. If $G$ admits a decomposition of the form $G\cong \prod_{i\in[1,k]}G_{i}$, then there exists a decomposition $H\cong \prod_{i\in[1,k]}H_i$ of $H$ \mbox{such that $H_i\equiv G_i$ for all $i\in[1,k]$.}
\end{fact}
\begin{lemma}\label{lemma_direct_products}
	Suppose that $G\cong \prod_{i\in[1,k]}G_i$ is a finitely generated abelian-by-finite group. If, for every $i\in[1,k]$, $G_i$ is first-order rigid, then $G$ is first-order rigid.
\end{lemma}
\begin{proof}
	Let $H$ be a finitely generated group elementarily equivalent to $G$. Then $H$ is also abelian-by-finite. In particular, both $G$ and $H$ are polycyclic-by-finite. By Fact\cspace\ref{fact - lasserre}, there exists a decomposition $H\cong\prod_{i\in[1,k]}H_i$, with $H_i\equiv G_i$ for all $i\in[1,k]$. Each factor $H_i$ is finitely generated, and hence $H_i\cong G_i$ for all $i\in[1,k]$. Consequently, $G\cong H$.
\end{proof}
\begin{theorem1.1}
	\emph{Finite direct products of absolutely irreducible split crystallographic groups are profinitely rigid (equiv. first-order rigid).}
\end{theorem1.1}
\begin{proof}
	This is immediate by \ref{cor_asb_irr_rigid} and \ref{lemma_direct_products}.
	
	
\end{proof}

\section{Root systems and Weyl groups}
In this section, we present some basic notions on root systems and their symmetries that we will need in the following. Detailed introductions to the subject can be found in \cite{bourbaki, humphreys_1}.

In Section\cspace\ref{sec6}, we will employ some computations from \cite{martinais}, where Bourbaki's framework from \cite{bourbaki} is adopted. So, for the reader's convenience, we will maintain consistency with \cite{martinais} and
primarily follow \cite{bourbaki}.

\smallskip
Throughout this section, we adopt the following notational conventions.
\begin{notation}\label{notation - root systems}
	Let $K$ be a field extending $\Q$, and $V$ be a finite-dimensional vector space over $K$. Then:
	\begin{enumerate}[(1)]
		\item $V^*$ denotes the dual space of $V$, i.e., the $K$-vector space consisting of all $K$-linear maps from $V$ to $K$;
		\item given a subring $Q$ of $K$ and a subset $X\subseteq V$, we write $\langle X\rangle_Q$ for the $Q$-span of $X$ in $V$, i.e., the smallest $Q$-submodule of $V$ containing $X$.
	\end{enumerate}
\end{notation}
\begin{definition}\label{def - reflection}
	Let $V$ be a vector space of dimension $l$ over a field $K$ extending $\Q$. Then, a \emph{reflection} in $V$ is an automorphism $s\in\mathrm{GL}(V)$ which is both:
	\begin{enumerate}[(1)]
		\item a \emph{pseudo-reflection}, i.e., the subspace $\mathrm{ker}(1-s)$ has dimension $l-1$;
		\item an \emph{involution}, i.e., $s^2=1$.
	\end{enumerate}
We sometimes say {\em affine reflection}, to stress that we are considering $s \in \mathrm{GL}(V)$.
\end{definition}

Notice that any choice of a non-zero vector $a\in V$ and a non-trivial $K$-linear form $\alpha^*\in V^*$ yields a pseudo-reflection $s_{a,\alpha^*}$ of the following form:
\begin{equation*}
	s_{a,\alpha^*}(x)\vcentcolon=x-\alpha^*(x).a,
\end{equation*}
\noindent
for all $x\in V$ (cf. \cite[\S\phantom{.}2.1, Ch.\phantom{.}V]{bourbaki}). Moreover, if $\alpha^*(a)=2$, then $s_{a,\alpha^*}$ is a reflection (in the sense of \ref{def - reflection}) and it satisfies $s_{a,\alpha^*}(a)=-a$.
\begin{fact}\label{fact - every involution induces a decomposition of V in ker(1-s)+ker(1+s)}
	Let $V$ be a finitely generated real vector space. Then, every involution $s\in \gl(V)$ induces a decomposition of $V$ as a direct sum of subspaces of the form:
	\begin{equation}\tag{$\star_1$}\label{eq.1 - decomposition V induced by involutions}
		V=\mathrm{ker}(1-s)\oplus\mathrm{ker}(1+s).
	\end{equation}
	\noindent
	When $V$ is Euclidean with inner product $(\,\cdot\,,\,\cdot\,)$ and $s\in O(V)$ is a reflection of the form $s=s_{a,\alpha^*}$, for some $a\in V$ and $\alpha^*\in V^*$, then $s$ admits explicit description as:
	\begin{equation*}
		s(x)=x-2\frac{(x,a)}{(a,a)}a,
	\end{equation*}
	for all $x\in V$ (cf. \cite[\S\phantom{.}2.3, Ch.\phantom{.}V]{bourbaki}, \cite[\S\phantom{.}9.1, Ch.\phantom{.}III]{humphreys_1}), and the decomposition \eqref{eq.1 - decomposition V induced by involutions} simply reduces to:
	\begin{equation*}
		V=H_a\oplus\langle a\rangle_\R,
	\end{equation*}
	\noindent
	with $H_a$ being the hyperplane orthogonal to $a$ in $V$.	
\end{fact}
\begin{definition}{\cite[D\'efinition\cspace1, \S\phantom{.}1.1, Ch.\phantom{.}VI]{bourbaki}}\label{def - root system}
	Let $V$ be a vector space over a field $K$ extending $\Q$. Then, a subset $R\subseteq V$ is said to be a \emph{root system in $V$} if the following conditions are satisfied:
	\begin{enumerate}[(1)]
		\item $R$ is finite, does not contain $0$, and $\langle R\rangle_K=V$;
		\item for all $\alpha\in R$, there is a $K$-linear form $\coalpha\in V^*$ such that:
		\begin{enumerate}[(a)]
			\item $\coalpha(\alpha)=2$;
			\item $R$ is stable under the reflection $s_{\alpha,\coalpha}$, i.e., $s_{\alpha,\coalpha}(R)\subseteq R$;
			\item $\coalpha(R)\subseteq \Z$.
		\end{enumerate}
	\end{enumerate}
	\noindent
	The elements $\alpha\in R$ are called the \emph{roots} of $R$, and the corresponding $K$-linear maps $\coalpha\in V^*$ are called the \emph{coroots} of $R$.
\end{definition}
\begin{notation}\label{s_a_notation}
	By \cite[Lemme\cspace1, \S\phantom{.}1.1, Ch.\phantom{.}VI]{bourbaki}, each root $\alpha\in R$ uniquely determines $\coalpha\in V^*$ and $s_{\alpha,\coalpha}\in \mathrm{GL}(V)$ through the axioms (2)(a),(b) above. Hence, the reflection $s_{\alpha,\coalpha}$ is simply denoted by $s_\alpha$ (cf. \cite[\S\phantom{.}1.1, Ch.\phantom{.}VI]{bourbaki}).
\end{notation}
In this paper, we are primarily concerned with root systems $R$ in real vector spaces. However, in Theorem\cspace\ref{thm. - R irreducible iff modules are (abs) irreducible} it will be necessary to regard $R$ as a root system in the $\Q$-vector space generated by its roots. The following notion provides a precise mean of translating between these two settings.
\begin{definition}{\cite[Proposition\cspace1, \S\phantom{.}8.1, Ch.\phantom{.}II]{bourbaki algèbre}}\label{prop. - characterization Q-structures}
	Let $V$ be a real vector space, and $U$ be a subspace of the $\Q$-vector space structure on $V$. Then, $U$ is a $\Q$\emph{-structure on }$V$ if it satisfies one of the following equivalent conditions:
	\begin{enumerate}[(1)]
		\item there exists a basis of $U$ (over $\Q$) which is also a basis of $V$ over $\R$;
		\item every basis of $U$ (over $\Q$) is also a basis of $V$ over $\R$;
		\item $V=\langle\, U\rangle_\R$ and every $\Q$-linearly independent subset of $U$ is $\R$-linearly independent;
		\item for every $\R$-vector space $V$, each $\Q$-linear function $f:U\rightarrow V'$ uniquely extends to an $\R$-linear map $\bar{f}:V\rightarrow V'$.
	\end{enumerate}
\end{definition}
The following fact will be crucial in Section\cspace\ref{sec6}.
\begin{fact}\label{fact - Q(R) is always a Q-structure in V}
	Every root system $R$ in a real vector space $V$ yields a $\Q$-structure $\langle R\rangle_\Q$ on $V$ (see \cite[Proposition\cspace1, \S\phantom{.}1.1, Ch.\phantom{.}VI]{bourbaki}, or \ref{fact - basis of a root system}).
\end{fact}
As a consequence of Fact\cspace\ref{fact - Q(R) is always a Q-structure in V}, a set $R$ of generators of a finite-dimensional real vector space $V$ is a root system in $V$ if and only if it is a root system in the $\Q$-vector space $\langle R\rangle_\Q$. Indeed, by Definition\cspace\ref{def - root system}(2)(c), each coroot $\coalpha$ of $R$ in $V$ only takes rational values on $\langle R\rangle_\Q$, and thus restricts canonically to a coroot in the rational span of $R$. Conversely, by Definition\cspace\ref{prop. - characterization Q-structures}(4), each coroot of $R$ in $\langle R\rangle_\Q$ extends uniquely to a coroot in $V$.
\begin{fact}\label{fact - Q(R^vee) is a Q-structure}
	Let $R$ be a root system in a real vector space $V$. Then, the following conditions hold:
	\begin{enumerate}[(1)]
		\item the set of coroots $R^\vee\vcentcolon=\{\coalpha:\alpha\in R\}$ is a root system in $V^*$ (cf. \cite[Proposition\cspace2, \S\phantom{.}1.1, Ch.\phantom{.}VI]{bourbaki});
		\item $\langle R^\vee\rangle_\Q$ is a $\Q$-structure on $V^*$ (cf. \cite[Proposition\cspace1, \S\phantom{.}1.1]{bourbaki}, or Fact\cspace\ref{fact - Q(R) is always a Q-structure in V}).
	\end{enumerate}
\end{fact}
\begin{definition}{\cite[\S\phantom{.}1.2, Ch.\phantom{.}VI]{bourbaki}}\label{def. - irreducible root system}
	Let $R$ be root system in a vector space $V$ over a field $K$ extending $\Q$. Then, $R$ is said to be \emph{reducible} if there are two non-empty subsets $R_0,R_1\subseteq R$ such that:
	\begin{enumerate}[(i)]
		\item $R=R_0\sqcup R_1$ (where $\sqcup$ denotes disjoint union);
		\item $V$ is the direct sum of the $K$-vector spaces $\langle R_0\rangle_K$ and $\langle R_1\rangle_K$;
		\item $R_i$ is a root system in $\langle R_i\rangle_K$, for $i = 0, 1$.
	\end{enumerate}
	In this case, we also say that $R$ is the \emph{direct sum} of the root systems $R_0$ and $R_1$.
	Finally, we say that $R$ is \emph{irreducible} if $R$ is not reducible and $R\neq\emptyset$.
\end{definition}
Any root system $R$ in $V$ admits a decomposition as a direct sum of the form:
\begin{equation*}
	\quad R=\bigsqcup_{i<m} R_i,
\end{equation*}
where $m\in\N$, and each $R_i$ is an irreducible root system in its span $\langle R_i\rangle_K$. Moreover, this decomposition is unique up to permutation of the factors (see \cite[Proposition\cspace6, \S\phantom{.}1.2, Ch.\phantom{.}VI]{bourbaki}).
\begin{definition}{\cite[\S\phantom{.}1.1, Ch.\phantom{.}VI]{bourbaki}}\label{def. - Weyl group of a root system}
	Let $R$ be a root system in a vector space $V$ over a field $K$ extending $\Q$. Then, the \emph{Weyl group} of $R$ is the subgroup $W_0(R)$ of $\mathrm{GL}(V)$ generated by the reflections $s_\alpha$, for all $\alpha\in R$ (cf. \ref{s_a_notation}).
\end{definition}
Notice that a Weyl group $W_0(R)$ is always finite, since by {(2)(b)} of Definition\cspace\ref{def - root system} we have an identification of $W_0(R)$ with a subgroup of the symmetric group on $R$.

\begin{fact}{\cite[Proposition, \S\cspace1.14, Ch.{\cspace}I]{humphreys}}\label{fact - every reflection in a Weyl group is induced by a root}
	Let $R$ be a root system in a real Euclidean vector space $V$, and $W_0(R)\leq O(V)$ its Weyl group. Then, every reflection $s\in W_0(R)$ (in the sense of \ref{def - reflection}) is of the form $s=s_\alpha$, for some $\alpha\in R$.
	
\end{fact}
Several structural properties of a Weyl group $W_0(R)$ can be deduced from the combinatorics of the root system $R$. For instance, if $R$ admits a direct sum decomposition $R=\bigsqcup_{i<m}R_i$, where each $R_i$ is irreducible with respect to its linear span $\langle R_i\rangle_K$, then, by \cite[\S\phantom{.}1.2, Ch.VI]{bourbaki}), we have:
\begin{equation*}
	W_0(R)\cong\prod_{i<m}W_0(R_i),
\end{equation*}
\noindent
where $W_0(R_i)$ denotes the Weyl group associated with $R_i$ in $\langle R_i\rangle_K$, for all $i<m$.
\begin{definition}[Coxeter groups]\label{def_Coxeter_groups} Let $S$ be a set. A matrix $m: S \times S \rightarrow \{1, 2, . . . , \infty \}$ is called a {\em Coxeter matrix} if it satisfies:
	\begin{enumerate}[(1)]
		\item $m(s, s') = m(s' , s)$;
		\item $m(s, s') = 1 \Leftrightarrow s = s'$.
	\end{enumerate}
	For such a matrix, let $S^2_{*} = \{(s, s') \in S^2 : m(s, s' ) < \infty \}$. A Coxeter matrix $m$ determines
	a group $W_0$ with presentation:
	$$
	\begin{cases} \text{Generators:} \; \;  S \\
		\text{Relations:} \; \;   (ss')^{m(s,s')} = e, \text{ for all } (s, s' ) \in S^2_{*}.
	\end{cases} $$
	A group with a presentations as above is called a \emph{Coxeter group}, and the pair $(W_0, S)$ is a called a \emph{Coxeter system}. The \emph{rank} of the Coxeter system $(W_0, S)$ is $|S|$.
\end{definition}
\begin{definition}\label{def_Coxeter_graph}
	In the context of Definition~\ref{def_Coxeter_groups}, the Coxeter matrix $m$ can be equivalently represented by a labeled graph $\Delta$, known as the \emph{Coxeter diagram} of the system $(W_0,S)$. The vertex set of $\Delta$ is $S$, and its edge set $E_\Delta$ consists of unordered pairs $\{s, s' \}\subseteq S$ such that $3\leq m(s, s') < \infty$. Each edge $\{s,s'\}$ is labeled by the integer $m(s,s')$. By convention, labels corresponding to $m(s,s')=3$ are typically omitted.
\end{definition}
An automorphism $\sigma$ of a Coxeter diagram $\Delta$ is a bijection of the set of nodes of $\Delta$ such that two nodes $s,s'\in\Delta$ are connected by an edge labeled $m(s,s')$ if and only if $\sigma(s),\sigma(s')$ are connected by an edge with the same label. Each such automorphism $\sigma$ of the Coxeter diagram $\Delta$ of a Coxeter group $W_0$ extends uniquely to an automorphism $f_\sigma$ of $W_0$.

\smallskip
If $R$ is a root system in an real vector space $V$, then the only values $\delta\in\R$ for which the the scalar multiple $\delta.\alpha$ of a root $\alpha\in R$ still belongs to $R$ are:
\begin{equation*}
	\delta=\;\pm\frac{1}{2},\;\pm1,\;\pm2.
\end{equation*}
(cf. \cite[Proposition\cspace8(i), \S\phantom{.}1.3, Ch.\phantom{.}VI]{bourbaki}). A root $\alpha$ for which $\frac{1}{2}\alpha\notin R$ is said \emph{indivisible}, and a root system entirely consisting of indivisible roots is called \emph{reduced}.
\begin{definition}
	\label{fact - basis of a root system}
	Let $R$ be a root system in a real vector space $V$. Then, a subset $B\subseteq R$ is a \emph{basis} (or \emph{simple system}) of $R$ if the following conditions are satisfied:
	\begin{enumerate}[(1)]
		\item $B$ is a basis of $V$;
		\item $B$ consists entirely of indivisible roots;
		\item each $\alpha\in R$ can be written as a $\Z$-linear combination of elements from $B$ with coefficients all of the same sign (either all non-negative, or all non-positive).
	\end{enumerate}
	The elements $\alpha$ belonging to some basis of $R$ are called \emph{simple roots} of $R$, and the corresponding reflections $s_\alpha$ are the \emph{simple reflections} of $R$.
\end{definition}
\begin{remark}\label{remark - humphreys}
	In \cite{bourbaki}, bases of root systems are introduced through a constructive method that also ensures the existence of a basis $B$ for any root system $R$ (see \cite[Théorème\cspace2, \S\phantom{.}1.5, Ch.\phantom{.}VI]{bourbaki}). Definition\cspace\ref{fact - basis of a root system} provides an equivalent characterization, as stated in \cite[Corollaire\cspace3, \S\phantom{.}1.7, Ch.\phantom{.}VI]{bourbaki}. This formulation is also adopted as the definition of a simple system in \cite{humphreys_1, humphreys}, where only reduced systems are considered (cf. \cite[\S\phantom{.}9.2,\phantom{.}10.1, Ch.\phantom{.}III]{humphreys_1}).
\end{remark}
\begin{fact}{\cite[Théorème\cspace2(vii), \S{\phantom{.}}1.5, Ch.VI]{bourbaki}}\label{fact - simple reflections form a Coxeter system of the Weyl gp}
	Let $R$ be a root system in a real vector space $V$, and $W_0(R)$ its Weyl group. If $B$ is a basis of $R$ and $S\vcentcolon=\{s_\alpha\in W_0(R):\alpha\in B\}$ is the set of simple reflections of $B$, then $(W_0(R), S)$ is a (finite) Coxeter system. 
\end{fact}
\begin{fact}\label{fact - the Weyl group acts simply transitively on the set of bases of a root system}
	Let $R$ be a root system in a real vector space $V$, and $W_0(R)$ its Weyl group. If $B$ is a basis of $R$, then, for each $w\in W_0(R)$, $w(B)$ is also a basis of $R$ (see \cite[\S\phantom{.}1.5, Ch.\phantom{.}VI]{bourbaki} and \cite[\S\phantom{.}10.1, Ch.\phantom{.}III]{humphreys_1}). The resulting action of $W_0(R)$ on the sets of bases of $R$ is simply transitive (cf. \cite[Remarques(4), \S\phantom{.}1.5, Ch.\phantom{.}VI]{bourbaki}).
\end{fact}
\begin{fact}{\cite[Proposition\cspace15, \S\phantom{.}1.5]{bourbaki}}\label{fact - every root in a reduced root system is the image of a root of the basis through an element of the Weyl group}
	Let $R$ be a reduced root system in a real vector space $V$, and $W_0(R)$ be the Weyl group of $R$. If $B$ is a basis of $R$ (cf. \ref{fact - basis of a root system}), then, for every root $\alpha'\in R$, there exist some $w\in W_0(R)$ and $\alpha\in B$ such that $\alpha'=w(\alpha)$.
\end{fact}
\begin{fact}{\cite[Corollaire, \S\phantom{.}1.6, Ch.\phantom{.}VI]{bourbaki}}\label{fact - every vector in Q(R) l.d. from an indivisible root alpha, is Z-proposrtional to alpha}
	Let $R$ be a root system in a real vector space, and $\alpha$ be an indivisible root in $R$. Then, for every $x\in \langle R\rangle_\Z\cap\langle\alpha\rangle_\R$, $x$ belongs to $\langle\alpha\rangle_\Z$.
\end{fact}
\begin{fact}
	The reduced irreducible root systems have been classified (see \cite[Planche{\cspace}I to IX]{bourbaki}). A list of the Coxeter diagrams of the corresponding Weyl groups can be found in Table~\ref{table - Coxeter diagrams of irreducible root systems} from the Appendix (i.e., Section~\ref{app_sec}).
\end{fact}
The following result is folklore, we will need it in the \mbox{next section (cf. \ref{remark - we can consider reduced root systems}).}
\begin{fact}\label{fact - R reduced root system is irreducible iff the coxeter diagram is connected}
	Let $R$ be a reduced root system in a real vector space $V$, and $W_0(R)$ be its Weyl group. Then, $R$ is irreducible if and only if the Coxeter diagram of $W_0(R)$ is connected.
\end{fact}
\begin{fact}\label{fact - w s_alpha w^-1 = s_w(alpha)}
	Let $R$ be a root system in a real Euclidean vector space $V$, and $W_0(R)\leq O(V)$ be its Weyl group. Then,
	\begin{enumerate}[(1)]
		\item for every basis $B$ of $R$, and every $w\in W_0(R)$, $w(B)$ is a basis of $R$ (cf. \cite[\S\phantom{.}1.4, Ch.\phantom{.}1]{humphreys});
		\item for every $\alpha\in R$ and every $w\in W_0(R)$, $(s_\alpha)^w=s_{w(\alpha)}$ (cf. \cite[Proposition, \S\phantom{.}1.2, Ch.\phantom{.}1]{humphreys}).
	\end{enumerate}
	In particular, inner automorphisms of $W_0(R)$ map Coxeter systems of simple reflections into Coxeter systems of simple reflections.
\end{fact}

\begin{definition}{\cite[Definition\cspace1.36, \S\phantom{.}1.4, Ch.{\cspace}1]{franzsen}}\label{def - inner-by-graph automorphism}
	Let $R$ be a root system in a real vector space $V$ with Weyl group $W_0(R)$. Suppose that $B$ is a basis of $R$ and $S=\{s_\alpha:\alpha\in B\}$ is a system of simple reflections. Then, an automorphism $f\in\mathrm{Aut}(W_0(R))$ is said to be \emph{inner-by-graph} if it belongs to the subgroup of $\mathrm{Aut}(W_0(R))$ generated by the subgroup of inner automorphisms and the subgroup of graph automorphisms of $(W_0(R), S)$, i.e., the permutations of $S$ inducing an automorphism of the Coxeter diagram $\Delta$ of $(W_0(R),S)$ (cf. \ref{def_Coxeter_graph} and \ref{fact - simple reflections form a Coxeter system of the Weyl gp}).
\end{definition}
Since the subgroup of inner automorphisms of $W_0(R)$ is normal, any inner-by-graph automorphism $f$ of $W_0(R)$ is a composition of the form:
\begin{equation*}
	f=(\,\cdot\,)^w\circ f_\sigma,
\end{equation*}
\noindent
for some graph automorphism $\sigma$ and some $w\in W_0(R)$ (cf. \cite[\S\phantom{.}1.4, Ch.\phantom{.}1]{franzsen}).
\begin{fact}{\cite[Proposition\cspace1.44, \S\phantom{.}1.4, Ch\phantom{.}1]{franzsen}}\label{fact - every automorphism of a Weyl group of an irreducible root system is inner-by-grapph} Let $R$ be an irreducible reduced root system in a real vector space $V$ with Weyl group $W_0(R)$. If $B$ is a basis of $R$, then every automorphism $f\in\mathrm{Aut}(W_0(R))$ preserving the set $S^{W_0(R)}=\{ws_\alpha w^{-1}:\alpha\in B\}$ is inner-by-graph.
\end{fact}
In Section\cspace\ref{sec6} we will see that, for any choice of a basis $B$ for a root system $R$ in $V$, the set $S^{W_0(R)}=\{ws_\alpha w^{-1}:\alpha\in B\}$ actually consists of all the affine reflections in $W_0(R)$. This will be crucial for the proof of Theorem\cspace\ref{main_th2}.

\begin{definition}{\cite[$\S\phantom{.}1.9$, Ch.\phantom{.}VI]{bourbaki}}\label{def - Q(R) and P(R)}
	Let $V$ be a finite-dimensional real vector space. If $R$ is a root system in $V$ and $W_0(R)$ is the Weyl group of $R$, then we define:
	\begin{enumerate}[(1)]
		\item $P(R)\vcentcolon=\{x\in V:\forall \alpha\in R(\coalpha(x)\in\Z)\}$, the group of \emph{weights} of $R$.
		\item $Q(R)\vcentcolon=\langle R\rangle_\Z$, the group of \emph{radical weights} of $R$;
	\end{enumerate}
\end{definition}
Clearly, by Definition\cspace\ref{def - root system}(2)(c), every radical weight is also a weight, and hence the inclusion $Q(R)\leq P(R)$ holds for any root system $R$. Both $Q(R)$ and $P(R)$ are known to be free abelian of the same rank (see \cite[Proposition\cspace26, \S\phantom{.}1.9, Ch.\phantom{.}VI]{bourbaki}), and the finite index $[P(R):Q(R)]$ is called the \emph{index of connection} of $R$ (cf.\cspace\cite[$\S${\phantom{.}}1.9, Ch.{\phantom{.}}VI]{bourbaki}). Since $P(R)$ and $Q(R)$ are stable under the action of the Weyl group $W_0$ of $R$, they are also endowed with a natural structure of $\Z[W_0]$-lattice (cf. \cite{martinais, maxwell}).

\smallskip
If $V$ is a finite-dimensional real vector space, any choice of a basis $B=\{v_j:1\leq j\leq l\}$ of $V$ determines a \emph{dual basis} $B^*=\{\delta_i:1\leq i\leq l\}$ in $V^*$ (cf. \ref{notation - root systems}) whose elements are the $\R$-linear maps $\delta_i:V\rightarrow \R$ such that $\delta_i(v_j)=1$ if $i=j$, and $0$ otherwise. The correspondence $v_i\mapsto\delta_i$, for $1\leq i\leq l$, naturally yields and isomorphism of $V$ and $V^*$.

If $R$ is a root system in $V$, then the identity $(\coalpha)^\vee=\alpha$ holds for each root $\alpha\in R$ (cf. \cite[Proposition\cspace2, \S\phantom{.}1.1, Ch.\phantom{.}VI]{bourbaki}). As a result, one has that the assignment $\alpha\mapsto\coalpha$ is a 1-1 correspondence, called \emph{canonical bijection}. This map usually does not extend to an isomorphism of $\langle R\rangle_\Z$ and $\langle R^\vee\rangle_\Z$, as it does not preserve sums of roots (see \cite[\S\phantom{.}1.1, Ch.\phantom{.}VI]{bourbaki}). Nevertheless, if the system $R$ is reduced, then every basis $B$ of $R$ yields a basis of $R^\vee$ of the form $B^\vee\vcentcolon=\{\coalpha:\alpha\in B\}$  (cf. \cite[Remarques(5), \S\phantom{.}1.5, Ch.\phantom{.}VI]{bourbaki}).

The weight group $P(R)$, as defined in  Definition\cspace\ref{def - Q(R) and P(R)}, is the free abelian group generated by the dual basis $(B^\vee)^*=\{\overline{\omega}_i:1\leq i\leq l\}$ of $B^\vee\subset V^*$ in $V$. The generators $\overline{\omega}_i$ are known as the \emph{fundamental weights} of the root system $R$ (cf. \cite[\S\phantom{.}1.10, Ch.\phantom{.}VI]{bourbaki}).  By Fact\cspace\ref{fact - Q(R^vee) is a Q-structure}, the natural isomorphism $V\cong V^*$ induced by the identification of $B^\vee$ and $(B^\vee)^*$ witnesses the following fundamental property.
\begin{fact}\label{fact - P(R) is a Q-structure on V}
	Let $R$ be a reduced root system in a real vector space $V$, and $P(R)$ be the group of weights of $R$. Then, $\langle P(R)\rangle_\Q$ is a $\Q$-structure on $V$.
\end{fact}
\begin{theorem}\label{thm. - R irreducible iff modules are (abs) irreducible}
	Let $R$ be a root system in a real vector space $V$ with Weyl group $W_0$. Then, the following conditions are equivalent:
	\begin{enumerate}[(1)]
		\item $R$ is irreducible in $V$;
		\item $R$ is irreducible as a root system in $\langle R\rangle_\Q$;
		\item $V$ is irreducible $($or indecomposable$)$ as an $\R[W_0]$-module;
		\item $V$ is absolutely irreducible as an $\R[W_0]$-module;
		\item $\langle R\rangle_\Q$ is irreducible $($or indecomposable$)$ as a $\Q[W_0]$-module;
		\item $\langle R\rangle_\Q$ is absolutely irreducible as a $\Q[W_0]$-module;
		\item $Q(R)$ is absolutely irreducible as a $\Z[W_0]$-lattice.
	\end{enumerate}
\end{theorem}
\begin{proof}
	Equivalence $[6\Leftrightarrow 7]$ follows directly from Definition\cspace\ref{def - absolutely irreducible lattice} and the fact that $\langle R\rangle_\Q$ and $\Q Q(R)$ are isomorphic as $\Q[W_0]$-modules. Equivalences $[1\Leftrightarrow3\Leftrightarrow4]$ and $[2\Leftrightarrow5\Leftrightarrow6]$ are established in \cite[Corollaire, \S\phantom{.}1.2, Ch.\phantom{.}VI]{bourbaki}. It remains to prove the equivalence $[1\Leftrightarrow2]$. The implication $[2\Rightarrow1]$ is immediate, as any decomposition of $R=R_0\sqcup R_1$ satisfying
	\begin{enumerate}[(a)]
		\item $R_0,R_1\neq\emptyset$;
		\item $V=\langle R_0\rangle_\R\oplus\langle R_1\rangle_\R$;
		\item each $R_i$ is a root system in $\langle R_i\rangle_\R$;
	\end{enumerate}
	\noindent
	naturally induces a decomposition of $\langle R\rangle_\Q$ as a direct sum of $\Q$-vector spaces $\langle R\rangle_\Q=\langle R_0\rangle_\Q\oplus\langle R_1\rangle_\Q$, in which each $R_i$ remains a root system in $\langle R\rangle_\Q$, by Definition\cspace\ref{def - root system}(2)(c).
	
	\smallskip\noindent
	For implication $[1\Rightarrow2]$, assume, towards a contradiction, that there exist two non-empty subsets $R'_0,R_1'\subseteq R$ such that:
	\begin{enumerate}[(a')]
		\item $R=R'_0\sqcup R'_1$;
		\item $\langle R\rangle_\Q=\langle R'_0\rangle_\Q\oplus\langle R'_1\rangle_\Q$;
		\item each $R'_i$ is a root system in the $\Q$-vector space $\langle R'_i\rangle_\Q$.
	\end{enumerate}
	\noindent
	Then, by (c'), $R_0$ and $R_1$ are also root systems in $\langle R_0\rangle_\R$ and $\langle R_1\rangle_\R$, respectively (cf. Fact\cspace\ref{fact - Q(R) is always a Q-structure in V} and the related discussion). Since the rational spans of $R_0,R_1$ are $\Q$-structures in $\langle R_0\rangle_\R$ and $\langle R_1\rangle_\R$, Definition\cspace\ref{prop. - characterization Q-structures}(2) implies that each basis $B_i$ of $\langle R_i\rangle_\Q$ is also a basis of $\langle R_i\rangle_\R$ as a real vector space. Moreover, by (b'), $B_0\sqcup B_1$ is a basis of $\langle R\rangle_\Q$, which is a $\Q$-structure on $V$, since $R\subseteq V$ is a root system. Hence, applying Definition\cspace\ref{prop. - characterization Q-structures}(2) once more, we conclude that $B_0\sqcup B_1$ is a basis of $V$ yielding the decomposition $V=\langle R_0\rangle_\R\oplus\langle R_1\rangle_\R$. This contradicts the irreducibility of $R$ in $V$, completing the proof.
\end{proof}
\section{Crystallography of Coxeter groups}\label{sec6}

In this section, we consider crystallographic groups admitting an affine realization as groups of isometries of a finite-dimensional real Euclidean vector space $V$, whose point group is generated by reflections and essential. A subgroup $W_0\leq O(V)$ is called \emph{essential} if its fixed-point subspace is trivial; that is, if we have:
\begin{equation*}
    V^{W_0}\vcentcolon=\{x\in V:w(x)=x\quad\text{for all}\quad w\in W_0\}=\{0\}.
\end{equation*}
\begin{fact}[{\cite[Proposition\cspace9, \S\phantom{.}2.5]{bourbaki}}]\label{fact - property of finite groups of isometries essential and generated by reflections} If $V$ is a real Euclidean vector space of finite dimension $l$, and $W_0$ is a finite subgroup of $O(V)$ that is essential and generated by reflections, then the following conditions are equivalent:
\begin{enumerate}[(1)]
	\item there exists a lattice of rank $l$ in $V$ stable under $W_0$;
	\item there exists a root system $R$ in $V$ whose Weyl group is $W_0$.
\end{enumerate}
\end{fact}
In \cite{martinais, maxwell}, crystallographic groups with point groups satisfying the above conditions were analyzed. Building on \cite{martinais}, we consider the following notion.
\begin{definition}\label{def - crystallographic group arising from a root system}
	We say that an (abstract) crystallographic group $W$ \emph{arises from a root system} if it admits an affine realization as a discrete subgroup of the isometry group of an $l$-dimensional Euclidean vector space $V$, with $l$ being the rank of the translation subgroup, such that the associated point group $W_0\leq O(V)$ is essential and generated by reflections $($in the sense of the ambient space $V)$. By Fact\cspace\ref{fact - property of finite groups of isometries essential and generated by reflections}, $W_0$ is the Weyl group of some root system $R$ in $V$. If, in addition, $R$ can be chosen to be irreducible, we say that $W$ \emph{arises from an irreducible root system}.
\end{definition}
%
\begin{fact}{\cite[Theorem\cspace1.10]{martinais}}\label{fact - existence reduced root system and sandwich condition}
	If $R$ is a root system in a real Euclidean vector space $V$ whose Weyl group $W_0(R)\leq O(V)$ is essential and generated by reflections, then for every lattice $L$ in $V$ invariant under $W_0(R)$ there exists a reduced root system $R'$ in $V$ such that the following conditions are satisfied:
	\begin{enumerate}[(1)]
		\item $W_0(R)$ is the Weyl group of $R'$ (i.e., $W_0(R)=W_0(R')$);
		\item $Q(R')\leq L\leq P(R')$.
	\end{enumerate}
\end{fact}
\begin{remark}\label{remark - we can consider reduced root systems}
	Let $W'$ be a crystallographic group arising from an irreducible root system, and let $V$ be a finite-dimensional real Euclidean vector space. Suppose that ${W}$ is an affine realization of $W'$ in $V$, whose point group $W_0\leq O(V)$ is the Weyl group of an irreducible root system $R$ in $V$. Then, in light of Fact\cspace\ref{fact - existence reduced root system and sandwich condition}, it is always possible to find a reduced root system $R'$ in $V$ with Weyl group $W_0(R')=W_0$ such that the translation lattice $L$ of $W$ satisfies the inclusions:
	\begin{equation*}
		Q(R')\leq L\leq P(R').
	\end{equation*}
	\noindent
	Since $R$ is irreducible, the Coxeter diagram of $W_0$ is connected, and hence $R'$ is also irreducible, by Fact\cspace\ref{fact - R reduced root system is irreducible iff the coxeter diagram is connected}.
	For this reason, in the remainder of the paper, we will adopt a standard convention (cf. \cite{martinais, maxwell}) and simply say \emph{irreducible} root system in place of \emph{irreducible reduced} root system. In fact, since all our results only depend on the point group $W_0(R)$ and the translation lattice $L$ (considered as a $\mathbb{Z}[W_0(R)]$-module), they remain valid in the general non-reduced context. 
\end{remark}

In \cite[Table I]{maxwell} (also reproduced in Table\cspace\ref{table all inequivalent lattices} below), Maxwell classified the isomorphism classes of lattices in a finite-dimensional real Euclidean vector space $V$ that are invariant under the action of the Weyl group $W_0(R)$ associated with an irreducible root system $R$ in $V$. These classes are represented by lattices $L$ satisfying:
\begin{equation*}
	Q(R)\leq L\leq P(R)
\end{equation*}
\noindent
(see also \cite[Tables I, II]{martinais}). Building on the techniques developed in Section~\ref{sec4}, we establish in the following lemma that all these lattices are absolutely irreducible.

\begin{lemma}\label{lemma-every_lattice_L_s.t._Q(R)<L<P(R)_is_absolutely_irreducible}
	Let $R$ be a root system in a real vector space $V$. If $W_0(R)$ is the Weyl group of $R$ in $V$, then the following conditions are equivalent:
	\begin{enumerate}[(1)]
		\item $R$ is irreducible in $V$;
		\item any $\Z[W_0(R)]$-lattice $L$ fitting into the chain of inclusions:
		\begin{equation*}
			Q(R)\leq L\leq P(R)
		\end{equation*}
		\noindent
		is absolutely irreducible.
	\end{enumerate}
\end{lemma}
\begin{proof}
	Direction $[2\Rightarrow1]$ is trivial: it follows directly from instantiating implication $[7\Rightarrow1]$ of Theorem\cspace\ref{thm. - R irreducible iff modules are (abs) irreducible} to the case $L= Q(R)$.
	
	\smallskip\noindent
	For the converse direction $[1\Rightarrow2]$, observe that $Q(R)$ is absolutely irreducible, by Theorem\cspace\ref{thm. - R irreducible iff modules are (abs) irreducible}. According to Definition\cspace\ref{def - absolutely irreducible module}, this means that $\Q Q(R)$ is absolutely irreducible as a $\Q[W_0(R)]$-module. Furthermore, by Definition\cspace\ref{def - Q(R) and P(R)} and the subsequent discussion, $Q(R)$ has finite index in $P(R)$. Since
	\begin{equation*}
		[P(R):Q(R)]=[P(R):L]\,\cdot\,[L:Q(R)],
	\end{equation*}
	\noindent
	it follows that $Q(R)$ has finite index in $L$. Hence, by Definition\cspace\ref{def - centering of a Z[G]-repr module}, $Q(R)$ is a centering of $L$. Therefore, by applying Lemma\cspace\ref{lemma every centering is Q-equivalent to the ambient module}, we conclude that $\Q Q(R)\cong\Q L$, and so that $L$ is absolutely irreducible, completing the proof. 
\end{proof}

Maxwell's inequivalent lattices from \cite[Table I]{maxwell} were explicitly computed by Martinais in \cite[Tables I, II, III]{martinais}, employing Bourbaki's classical framework \cite[Planche~I to IX]{bourbaki}. These lattices fall into a few families, parametrized by the rank $l$ of the root system $R$. For completeness, these lattices are presented in \mbox{Table\cspace\ref{table - martinais nomenclature} below.}

\medskip
By Bieberbach's First Theorem, every crystallographic group $W$ has an affine realization as a subgroup of the group of isometries of a finite-dimensional real Euclidean vector space $V$. In this context, the point group of $W$ is a finite subgroup $W_0$ of the orthogonal group $O(V)$ and its translation lattice $L$ is a discrete cocompact subgroup of $V$ stable under the action of $W_0$ (cf. Section\cspace\ref{prel_sec}). It follows from a general theory (cf. \cite{hiller, martinais, maxwell}) that the group extensions of $W_0$ by $L$, regarded as a $\Z[W_0]$-lattice, correspond to cohomology classes in $H^1(W_0,V/L)$. Moreover, this correspondence is such that two crystallographic groups are isomorphic if and only if their cohomology classes lie in the same orbit under the natural action of the normalizer of $W_0$ in $\mathrm{Aut}(L)$ (see \cite[Sectiom\phantom{.}1]{martinais} and \cite[Theorem\cspace5.2]{hiller}).

\smallskip
Working within this framework, Martinais \cite{martinais} also computed, for every irreducible root system $R$ and $\Z[W_0(R)]$-lattice $L$ from Table~\ref{table - martinais nomenclature}, the number $n(W_0(R), L)$ of these orbits, and hence of isomorphism classes, of crystallographic groups extending $W_0(R)$ by $L$ (cf. the relevant column of Table~\ref{table all inequivalent lattices}). Finally, he provided explicit representatives for each of these isomorphism classes (cf. \cite[Table{\cspace}V]{martinais}), that we list in Table\cspace\ref{table - explicit crystallographic groups} for convenience of the reader. The following remark explains some conventions used by Martinais.

\smallskip
\begin{remark}\label{remark_for_star2}
	Every crystallographic group $W$ in Martinais's classification from \cite{martinais} is defined by a presentation of the form:
	\begin{equation}\tag{$\star_2$}\label{eq.1 - martinais' groups presentation}
		W=\langle (x,1), (t_1,s_1),\ldots,(t_l,s_l): x\in L'\rangle_{\iso(V)},
	\end{equation}
	\noindent
	where:
	\begin{enumerate}[(1)]
		\item $V$ is the real Euclidean vector space underlying the irreducible root system $R$;
		\item $L'$ is one of the $\Z[W_0]$-lattices in Table\cspace\ref{table - martinais nomenclature}, and $L=\{(x,1):x\in L'\}$ is the translation lattice of $W$;
		\item $l$ is the rank of $R$ (equivalently, the dimension of $V$);
		\item $t_1,\ldots,t_l$ are translations in $V$ (but generally not in $L'$);
		\item $s_1,\ldots,s_l$ is the (Coxeter) system of simple reflections associated with the basis $B=\{\alpha_i:1\leq i\leq l\}$ of $R$ from Bourbaki's standard realization \cite[Planche{\cspace}I to IX]{bourbaki}.
	\end{enumerate}
\end{remark}
In particular, for any irreducible root system $R$ and $\Z[W_0]$-lattice
there exists a split crystallographic group:
\begin{equation*}
	W=\langle (x,1),(0,s_j):1\leq j\leq l,x\in L'\rangle_{\iso(V)}\cong L\rtimes W_0(R).
\end{equation*}

By Theorem~\ref{theorem}, if $W$ is a crystallographic group of the form \eqref{eq.1 - martinais' groups presentation} from \ref{remark_for_star2}, then every finitely generated $H$ that is elementarily equivalent to $W$ necessarily extends $W_0(R)$ by $L$, and hence falls into one of the isomorphism classes of Martinais's analysis \cite{martinais}. Lemma\cspace\ref{lemma-split extensions are definable} below shows that the isomorphism $H\cong W$ is immediate when $n(W_0(R),L)\leq 2$, since in these cases there are only two possible non-isomorphic extensions: one split, and one non-split. As a result, in the proof of Theorem\cspace\ref{main_th2} we will only need an explicit description of these crystallographic groups in the case $n(W_0(R),L)\geq 3$. So, we briefly recall them in Table\cspace\ref{table - explicit crystallographic groups} below.
\begin{lemma}\label{lemma-split extensions are definable}
	Let $W$ be a crystallographic group with point group $W_0$ and translation subgroup $T$. Then there exists a first-order sentence $\psi$ (depending only on $W_0$, and not on $W$ or $T$) such that $W\models\psi$ if and only if $W$ is split.
\end{lemma}
\begin{proof}
	Let $p:W\rightarrow W_0$ denote the canonical projection of $W$ onto $W_0$, and let $\{w_i:i<k\}$ be an enumeration without repetition of $W_0$. By definition (see \cite[Proposition\phantom{.}2.1, Ch.\phantom{.}IV]{brown}), $W$ is split if and only if there exists an embedding $j:W_0\rightarrow W$ such that $p\circ j=\mathrm{id}_{W_0}$. Since $W_0\cong W/T$, this amounts to saying that there is an embedding of $W_0$ in $W$ whose elements belong to pairwise \mbox{distinct cosets.}
	
	\smallskip\noindent
	By Lemma\cspace\ref{lemma_the_prop_def_transla_crysp}, the translation subgroup $T$ is $\emptyset$-definable in $W$ by a formula depending only on the order $k$ of $W_0$. Thus, since $W_0$ is finite, it suffices to consider the formula $\psi$ stating the existence of $k$ elements $x_0,\ldots,x_{k-1}$ such that $\{x_0,\ldots,x_{k-1}\}$ is a group isomorphic to $W_0$, and $x_ix_j^{-1}\in T$ if and only if $i=j$.
\end{proof}
\begin{lemma}\label{lemma - abstract reflections and affine reflections cohincide}
	Let $R$ be a reduced root system in a real Euclidean vector space $V$, and let $B$ be a basis of $R$. If $S$ is the system of simple reflections in the Weyl group $W_0(R)$ of $R$ associated to $B$, then, for every $s\in W_0(R)$, the following conditions are equivalent:
	\begin{enumerate}[(1)]
		\item $s$ is an affine reflection of $V$ (cf. Definition\cspace\ref{def - reflection});
		\item $s$ belongs to the set $S^{W_0(R)}\vcentcolon=\{ws_\alpha w^{-1}:\alpha\in B, w\in W_0(R)\}$.
	\end{enumerate}
\end{lemma}
\begin{proof}
	For implication $[1\Rightarrow2]$, suppose that $s\in W_0(R)$ is an affine reflection in $V$ (in the sense of Definition~\ref{def - reflection})). Then, by Fact\cspace\ref{fact - every reflection in a Weyl group is induced by a root}, there exists a root $\alpha'\in R$ such that $s=s_{\alpha'}$. Since $R$ is reduced, Fact\cspace\ref{fact - every root in a reduced root system is the image of a root of the basis through an element of the Weyl group} ensures that any such $\alpha'$ is of the form $\alpha'=w(\alpha)$ for some $\alpha\in B$ and $w\in W_0(R)$. By Fact\cspace\ref{fact - w s_alpha w^-1 = s_w(alpha)}(2), we have:
	\begin{equation*}
		s=s_{w(\alpha)}=ws_\alpha w^{-1},
	\end{equation*}
	\noindent
	and hence $s\in S^{W_0(R)}$.
	
	\smallskip\noindent
	Similarly, implication $[2\Rightarrow1]$ is an straightforward consequence of Fact\cspace\ref{fact - w s_alpha w^-1 = s_w(alpha)}(2): if $s=ws_\alpha w^{-1}$ for some $\alpha\in B$ and $w\in W_0(R)$, then $s=s_{w(\alpha)}$ (recall that, by Fact\cspace\ref{fact - w s_alpha w^-1 = s_w(alpha)}(1), $w(\alpha) \in R$, and thus  $s_{w(\alpha)}$ is trivially an affine reflection~of~$V$).
\end{proof}
\begin{remark}\label{remark - all lattices with n>2 is of type Q(R) or P(R)}
	Let $R$ be an irreducible root system in a real vector space $V$ with Weyl group $W_0(R)$. If $L'$ is a $\Z[W_0]$-lattice in $V$ from Table\cspace\ref{table all inequivalent lattices} such that $n(W_0,L')\geq3$, then $L'$ is of type $Q(R)$ or $P(R)$.
\end{remark}
\begin{lemma}\label{lemma - reflections can be detected on the translation subgroup}\label{lemma - logic characterization of reflections}
	Let $V$ be a real Euclidean vector space, and $R$ be a root system of rank $l$ in $V$ with Weyl group $W_0(R) \leq O(V)$. Suppose that $L'$ is a $\Z[W_0(R)]$-lattice from \emph{Table\cspace\ref{table all inequivalent lattices}} such that $n(W_0(R),L')\geq3$, and $T'$ is the abelian group structure on $L'$. Then, for every involution $s\in W_0(R)$ the following conditions are equivalent:
	\begin{enumerate}[(1)]
		\item $s$ is an affine reflection in $V$ (cf. Definition\cspace\ref{def - reflection});
		\item $T'\cap\mathrm{ker}(1-s)$ and $T'\cap\mathrm{ker}(1+s)$ are free abelian of rank $l-1$ and $1$, respectively.
	\end{enumerate}
	Furthermore, if $W$ is a crystallographic group of the form \eqref{eq.1 - martinais' groups presentation} from \ref{remark_for_star2}, with translation subgroup $T=\{(t,1):t\in T'\}$,
	then, for each $u\in W$ such that $u=(t,s)$, with $t\in V$, (1) and (2) are equivalent to:
	\begin{enumerate}[(3)]
		\item the subgroups $T^u=\{(x,1)\in T:(x,1)^u=(x,1)\}$ and $T_u=\{(x,1)\in T: (x,1)^u=(-x,1)\}$ have rank $l-1$ and $1$, respectively.
	\end{enumerate}
\end{lemma}
\begin{proof}
	For implication $[1\Rightarrow2]$, suppose that $s\in W_0(R)$ is an affine reflection. Then, by Fact\cspace\ref{fact - every reflection in a Weyl group is induced by a root}, there exists some root $\alpha\in R$ such that $s=s_\alpha$. Since $V$ is Euclidean, by Fact\cspace\ref{fact - every involution induces a decomposition of V in ker(1-s)+ker(1+s)} the subspaces $\mathrm{ker}(1-s)$ and $\mathrm{ker}(1+s)$ of $V$ correspond to $H_\alpha$ and $\langle \alpha\rangle_\R$, with $H_\alpha=\{x\in V:(\alpha,x)=0\}$ denoting the hyperplane orthogonal to $\alpha$. We show that $\langle\alpha\rangle_\R\cap T'$ and $H_\alpha\cap T'$ have rank $1$ and $l-1$, respectively.
	
	\smallskip\noindent
	Since $R$ is reduced, Fact\cspace\ref{fact - every root in a reduced root system is the image of a root of the basis through an element of the Weyl group} ensures that each $\alpha$ as above is actually the image of an element $\alpha'$ in some basis $B'$ of $R$ through a map $w\in W_0(R)$. Any such $w$ is an orthogonal map, i.e., an automorphism of $V$ preserving inner products, and hence it maps $H_{\alpha'}$ into $H_\alpha$ and $\langle\alpha'\rangle_\R$ into $\langle\alpha\rangle_\R$. In particular, $w$ transforms the decomposition $V=H_{\alpha'}\oplus\langle \alpha'\rangle_\R$ into $V=H_{\alpha}\oplus\langle \alpha\rangle_\R$. Therefore, since $T'$ is stable under the action of $W_0(R)$, this means that $w$ restricts to an isomorphism between $H_{\alpha'}\cap T'$ and $H_{\alpha}\cap T'$, and an isomorphism between $\langle \alpha'\rangle_\R\cap T'$ and $\langle \alpha\rangle_\R\cap T'$, respectively.
	
	\smallskip\noindent
	By Fact\cspace\ref{fact - the Weyl group acts simply transitively on the set of bases of a root system}, $W_0(R)$ acts simply transitively on the set of bases of $R$. Hence, without loss of generality, we can assume that $\alpha$ belongs to one of the bases $B$ of $R$ listed in Table\cspace\ref{table - root systems, bases and reflections} in terms of the canonical basis $\{\epsilon_i:1\leq i\leq l\}$ of $V$. According to the root system and lattice types realizing $n(W_0(R),L')\geq 3$, we distinguish the following cases.
	
	\smallskip\noindent
	\underline{{\bf Case 1}}. $R=B_l$ and $T'=Q(B_l)=CL_l$, with $l\geq3$.
	
	\smallskip\noindent
	In this case, $T'$ is the $\Z$-linear span of the basis $B=\{\alpha_i:1\leq i\leq l\}$ of $R$ such that:
	\begin{equation}\label{eq.1 - case B_l lemma logical characterization of affine reflections}\tag{$\star_3$}
		\alpha_i=\epsilon_i-\epsilon_{i+1},\quad\text{and}\quad\alpha_l=\epsilon_l,\quad\text{for all}\quad i\in[1,l-1].
	\end{equation} 
	\noindent
	By \ref{fact - basis of a root system}, $B$ is also a basis of $V$, and each element of $T'$ has a unique expression as a $\Z$-linear combination of the $\alpha_i$'s. {Therefore, $\langle \alpha\rangle_\R\cap T'$ is clearly a subgroup of $T'$ of} rank $1$, with free abelian basis $\{\alpha\}$ (alternatively, one can use  {Fact\cspace\ref{fact - every vector in Q(R) l.d. from an indivisible root alpha, is Z-proposrtional to alpha}). To establish} that $H_\alpha\cap T'$ has rank $l-1$, we use the explicit description of $CL_l$ from Table\cspace\ref{table all inequivalent lattices}, {together with the fact that $\{\epsilon_i:1\leq i\leq l\}\subseteq B_l\subseteq T'$. By \eqref{eq.1 - case B_l lemma logical characterization of affine reflections}, we discuss two cases.}
	
	\smallskip\noindent
	\underline{{\bf Case 1.1}}. $\alpha=\alpha_i=\epsilon_i-\epsilon_{i+1}$, for some $i\in[1,l-1]$.
	
	\smallskip\noindent
	In this case, the subset $\{\epsilon_j:j\neq i,i+1\}$ of the canonical basis clearly consists of $\R$-linearly independent vectors lying in the intersection $H_\alpha\cap T'$. Since $H_\alpha$ is an $(l-1)$-dimensional subspace, it suffices to complete this set to an $\R$-basis of $H_\alpha$ by adding a suitable $\beta\in H_\alpha\cap T'$ that is $\R$-linear independent from all the $\epsilon_j$'s such that $j\neq i, i+1$. The root $\beta\vcentcolon=\epsilon_i+\epsilon_{i+1}\in B_l$ satisfies these requirements: it lies in $H_\alpha$, since it is orthogonal to $\alpha=\epsilon_i-\epsilon_{i+1}$, and it is clearly $\R$-linearly independent from $\{\epsilon_j:j\neq i, i+1\}$. Consequently, $\{\beta,\epsilon_j:j\neq i, i+1\}$ is a free abelian basis of $H_\alpha\cap T'$ of size $l-1$.
	
	
	\smallskip\noindent
	\underline{{\bf Case 1.2}}. $\alpha=\alpha_l=\epsilon_l$.
	
	\smallskip\noindent
	This case is trivial: $\{\epsilon_j:1\leq j<l\}$ is a basis of $H_\alpha$ lying entirely in $T'$. Hence, it is a free abelian basis of $H_\alpha\cap T'$.

	\smallskip\noindent
	\underline{{\bf Case 2}}. $R=B_4$, and $T'=P(B_4)=CCL_4$.
	
	\smallskip\noindent
	In this case, by the explicit description of $CCL_4$ in Table\cspace\ref{table - martinais nomenclature}, each element $x\in T'$ has a unique expression as a $\Z$-linear combination of the form:
	\begin{equation}\tag{$\star_4$}\label{eq.1 - reflections of type B_4 and CCL_4}
		x=\sum_{j=1}^3a_i\epsilon_j+b\left(\frac{1}{2}(\epsilon_1+\ldots+\epsilon_4)\right)=\sum_{j=1}^3\left(a_j+\frac{b}{2}\right)\epsilon_j+\frac{b}{2}\epsilon_4,
	\end{equation}
	\noindent
	for some $a_1,a_2,a_3,b\in\Z$. 
	
	\smallskip\noindent
	By Table\cspace\ref{table - root systems, bases and reflections}, we can consider the basis $B=\{\alpha_i:1\leq l\leq4\}$ for $B_4$ such that:
	\begin{equation}\label{eq.2 - case B_4 lemma logical characterization of affine reflections}\tag{$\star_5$}
		\alpha_i=\epsilon_i-\epsilon_{i+1},\quad\text{and}\quad\alpha_4=\epsilon_4,\quad\text{for all}\quad i\in[1,3].
	\end{equation}
	\noindent
	Since $T'=P(B_4)$ strictly contains the group $Q(B_4)$ of radical weights of $B_4$, the set $B$ does not form a free abelian basis of $T'$. As a consequence, for each $\alpha\in B$, we cannot directly conclude that $\langle \alpha\rangle_\R\cap T'$ coincides with the $\Z$-linear span of $\alpha$, and we must verify, case by case, both that $\langle \alpha\rangle_\R\cap T'$ has rank $1$ and that $H_\alpha\cap T'$ has rank $l-1$. In particular, from \eqref{eq.2 - case B_4 lemma logical characterization of affine reflections} we distinguish three possibilities.
	
	\smallskip\noindent
	\underline{{\bf Case 2.1}} $\alpha=\alpha_i=\epsilon_i-\epsilon_{i+1}$, with $i\in[1,2]$.
	
	\smallskip\noindent
	To prove that $H_\alpha\cap T'$ has rank $3$, we study the set of solutions $(a_1,a_2,a_3,b)\in\Z^4$ of the equation $(\alpha,x)=0$, for $x\in T'$ as in \eqref{eq.1 - reflections of type B_4 and CCL_4}. This is the set of tuples satisfying the identity $a_i=a_{i+1}$. Hence, each $x\in H_\alpha\cap T'$ is expressible a as linear a combination:
	\begin{equation*}
		x=a_k\epsilon_k+a_i(\epsilon_i+\epsilon_{i+1})+b\left(\frac{1}{2}(\epsilon_1+\ldots+\epsilon_4)\right),
	\end{equation*}
	\noindent
	such that $a_i,a_k,b\in\Z$ and $k\in\{1,2,3\}\setminus\{i,i+1\}$. This shows that \mbox{$\{\epsilon_k,\epsilon_i+\epsilon_{i+1},$}\\{$\frac{1}{2}(\epsilon_1+\ldots+\epsilon_4)\}$} is a set of generators for $H_\alpha\cap T'$. The $\Z$-linear independence follows\\directly from that of the free abelian basis \mbox{$\{\epsilon_1,\epsilon_2,\epsilon_3,\frac{1}{2}(\epsilon_1+\ldots+\epsilon_4)\}$ of $T'=CCL_4$.}
	
	\smallskip\noindent
	Similarly, it is straightforward to verify that any $x\in \langle\alpha\rangle_\R\cap T'$ as in \eqref{eq.1 - reflections of type B_4 and CCL_4} is uniquely determined by a tuple $(a_1,a_2,a_3,b)\in\Z^4$ such that:
	\begin{equation*}
		b=0,\quad a_i=-a_{i+1},\quad \text{and}\quad a_k=0,\quad \text{for} \quad k\in\{1,2,3\}\setminus\{i,i+1\}.
	\end{equation*}
	\noindent
	Therefore, $\alpha=\epsilon_i-\epsilon_{i+1}$ itself is a free generator of $\langle \alpha\rangle_\R\cap T'$, as \mbox{$\alpha\in B_4\subseteq P(B_4)=T'$.}
	
	\smallskip\noindent
	\underline{{\bf Case 2.2}} $\alpha=\alpha_3=\epsilon_3-\epsilon_4$.
	
	\smallskip\noindent
	By imposing the condition $(\alpha,x)=0$ on the elements $x\in T'$ as in \eqref{eq.1 - reflections of type B_4 and CCL_4}, we deduce that each $x\in H_\alpha\cap T'$ corresponds uniquely to a tuple $(a_1,a_2,a_3,b)\in \Z^4$ satisfying $a_3=0$. It follows that $\{\epsilon_1,\epsilon_2,\frac{1}{2}(\epsilon_1+\ldots+\epsilon_4)\}$ is a free abelian basis of \mbox{$H_\alpha \cap T'$.}
	
	\smallskip\noindent
	Similarly, each tuple $(a_1,a_2,a_3,b)\in \Z^4$ identifying some $x\in \langle \alpha\rangle_\R\cap T'$ must satisfy the identities:	
	\begin{equation*}
		a_1=a_2=-\frac{b}{2},\quad\text{and}\quad a_3=-b.
	\end{equation*}
	\noindent
	Thus, any such $x$ has an expression of the form:
	\begin{equation*}
		x=a \epsilon_1+a\epsilon_2+2a\epsilon_3-2a\left(\frac{1}{2}(\epsilon_1+\ldots+\epsilon_4)\right)=a\left(\epsilon_3-\epsilon_4\right),
	\end{equation*}
	\noindent
	for some $a\in\Z$. Hence, $\alpha=\epsilon_3-\epsilon_4$ itself is a free generator of $\langle\alpha\rangle_\R\cap T'$.
	
	\smallskip\noindent
	\underline{{\bf Case 2.3}} $\alpha=\alpha_4=\epsilon_4$.
	
	\smallskip\noindent
	In this case, \eqref{eq.2 - case B_4 lemma logical characterization of affine reflections} directly witnesses that $\{\epsilon_1,\epsilon_2,\epsilon_3\}\subseteq CCL_4=T'$ is a free abelian basis of $H_\alpha\cap T'$. Moreover, each $x\in\langle\alpha\rangle_\R\cap T'$ as in \eqref{eq.1 - reflections of type B_4 and CCL_4} is uniquely determined by a tuple $(a_1,a_2,a_3,b)\in\Z^4$ such that:
	\begin{equation*}
		a_1=a_2=a_3=-\frac{b}{2}
	\end{equation*}
	\noindent
	Therefore, any such $x$ is of the form:
	\begin{equation*}
		x=a\epsilon_1+a\epsilon_2+a\epsilon_3-2a\left(\frac{1}{2}(\epsilon_1+\ldots+\epsilon_4)\right)=a\epsilon_4,
	\end{equation*}
	\noindent
	for some $a\in\Z$. Therefore, also in this case $\alpha$ is a free generator of $\langle\alpha\rangle_\R\cap T'$.
	
	\smallskip\noindent
	\underline{{\bf Case 3}}. $C_l$, and $T'=Q(C_l)=FL_l$, with $l\geq3$ odd.
	
	\smallskip\noindent
	In this case, $T'$ is the free $\Z$-linear span of the basis $B=\{\alpha_i:1\leq i\leq l\,\}$ of $C_l$ such that:
	\begin{equation}\label{eq.3 - case C_l lemma logical characterization of affine reflections}\tag{$\star_6$}
		\alpha_i=\epsilon_i-\epsilon_{i+1},\quad\text{and}\quad\alpha_l=2\epsilon_l,\quad\text{for all}\quad i\in[1,l-1]
	\end{equation} 
	\noindent
	(see Table\cspace\ref{table - root systems, bases and reflections}). Since we assumed $\alpha\in B$, it follows that $\langle \alpha\rangle_\R\cap T'$ is necessarily a rank-$1$ subgroup of $T'$ freely generated by $\alpha$ (cf. Fact\cspace\ref{fact - every vector in Q(R) l.d. from an indivisible root alpha, is Z-proposrtional to alpha}). Moreover, each $x\in T'$ admits a unique expression as a linear combination:
	\begin{equation}\label{eq.4 - case C_l lemma logical characterization of affine reflections}\tag{$\star_7$}
		x=\sum_{i=1}^la_i\alpha_i
		=a_1\epsilon_1+\sum_{j=2}^{l-1}(a_j-a_{j-1})\epsilon_j+(2a_l-a_{l-1})\epsilon_l,
	\end{equation}
	\noindent
	for some $a_1,\ldots,a_l\in\Z$. To compute the rank of $H_\alpha\cap T'$, we refer to \eqref{eq.3 - case C_l lemma logical characterization of affine reflections} and distinguish four cases.
	
	\smallskip\noindent
	\underline{{\bf Case 3.1}}. $\alpha=\alpha_1=\epsilon_1-\epsilon_2$.
	
	\smallskip\noindent
	By imposing the condition $(\alpha,x)=0$ on the elements $x\in T'$ as in \eqref{eq.4 - case C_l lemma logical characterization of affine reflections}, we obtain that each $x\in H_\alpha\cap T'$ is uniquely determined by a tuple $(a_1,\ldots,a_l)\in\Z^l$ such that $2a_1-a_2=0$. Hence, any such $x$ is of the form:
	\begin{equation*}
		x=\sum_{j\neq 1,2}a_j\alpha_j+a_1\alpha_1+2a_1\alpha_2= \sum_{j\neq 1,2}a_j\alpha_j+a_1(\alpha_1+2\alpha_2).
	\end{equation*}
	\noindent
	It follows that $\{\alpha_j,\alpha_1+2\alpha_2:j\neq 1,2\}\subseteq Q(C_l)=T'$ is a generating set of $H_\alpha\cap T'$. We claim that this is actually a free abelian basis of $H_\alpha\cap T'$.
	
	\smallskip\noindent
	Adding $\alpha_2$ to $\{\alpha_j,\alpha_1+2\alpha_2:j\neq 1,2\}$ yields a generating set for $T'$ of cardinality $l$. Since free abelian groups of finite rank are Hopfian, it follows that $\{\alpha_2,\alpha_j,\alpha_1+2\alpha_2:j\neq 1,2\}$ is a free abelian basis of $T'$. In particular, the set $\{\alpha_j,\alpha_1+2\alpha_2:j\neq 1,2\}$ consists of $\Z$-linearly independent vectors, and thus forms a free abelian basis of $H_\alpha\cap T'$.	
	
	\smallskip\noindent
	\underline{{\bf Case 3.2}}. $\alpha=\alpha_i=\epsilon_i-\epsilon_{i+1}$, for some $i\in[2,l-2]$.
	
	\smallskip\noindent
	As above, by imposing the condition $(\alpha,x)=0$ on the elements $x\in T'$ as in \eqref{eq.4 - case C_l lemma logical characterization of affine reflections}, one derives that each $x\in H_\alpha\cap T'$ uniquely corresponds to a tuple $(a_1,\ldots,a_l)\in\Z^l$ satisfying $2a_i-a_{i-1}-a_{i+1}=0$. Hence, any such $x$ is of the form:
	\begin{equation*}
		\begin{split}
			x&=\sum_{j\neq i,i+1}a_j\alpha_j+a_i\alpha_i+(2a_i-a_{i-1})\alpha_{i+1}\\
			&= \sum_{j\neq i-1,i,i+1}a_j\alpha_j+a_{i-1}(\alpha_{i-1}-\alpha_{i+1})+a_i(\alpha_i+2\alpha_{i+1}).
		\end{split}
	\end{equation*}
	\noindent
	Therefore, the subset $\{\alpha_j,\alpha_{i-1}-\alpha_{i+1}, \alpha_i+2\alpha_{i+1}:j\neq i-1,i,i+1\}\subseteq Q(C_l)=T'$ generates $H_\alpha\cap T'$. Extending this set by including $\alpha_{i-1}$ yields a generating set for $T'$ of cardinality $l$. Consequently, by Hopfianity,
    the set $\{\alpha_j,\alpha_{i-1}-\alpha_{i+1}, \alpha_i+2\alpha_{i+1}:j\neq i-1,i,i+1\}$ consists of $\Z$-linearly independent vectors, and hence forms a free abelian basis of $H_\alpha\cap T'$.
	
	\smallskip\noindent
	\underline{{\bf Case 3.3}}. $\alpha=\alpha_{l-1}=\epsilon_{l-1}-\epsilon_{l}$.
	
	\smallskip\noindent
	By studying the set of solutions of the equation $(\alpha,x)=0$, for $x\in T'$ as in \eqref{eq.4 - case C_l lemma logical characterization of affine reflections}, one derives that each $x\in H_\alpha\cap T'$ uniquely corresponds to a tuple $(a_1,\ldots,a_l)\in\Z^l$ satisfying \mbox{$-a_{l-2}+2a_{l-1}-2a_{l}=0$.} Thus, any such $x$ is of the form:
	\begin{equation*}
		\begin{split}
			x&=\sum_{j\neq l-2}a_j\alpha_j+(2a_{l-1}-2a_{l})\alpha_{l-2}\\
			&= \sum_{j\neq l-2,l-1,l}a_j\alpha_j+a_{l-1}(\alpha_{l-1}+2\alpha_{l-2})+a_l(\alpha_l-2\alpha_{l-2}).
		\end{split}
	\end{equation*}
	\noindent
	This proves that the subset $\{\alpha_j,\alpha_{l-1}+2\alpha_{l-2}, \alpha_l-2\alpha_{l-2}:j\neq l-2,l-1,l\}\subseteq Q(C_l)=T'$ generates $H_\alpha\cap T'$. Extending it by including $\alpha_{l-2}$ gives a generating set for $T'$ of cardinality $l$. By Hopfianity, it follows that the vectors in $\{\alpha_j,\alpha_{l-1}+2\alpha_{l-2}, \alpha_l-2\alpha_{l-2}:j\neq l-2,l-1,l\}$ are also $\Z$-linearly independent, and hence form a free abelian basis of $H_\alpha\cap T'$ of size $l-1$.
	
	\smallskip\noindent
	\underline{{\bf Case 3.4}}. $\alpha=\alpha_l=2\epsilon_l$.
	
	\smallskip\noindent
	As in the previous cases, the analysis of the equation $(\alpha,x)=0$, for $x\in T'$ as in \eqref{eq.4 - case C_l lemma logical characterization of affine reflections}, shows that each $x\in H_\alpha\cap T'$ is uniquely determined by a tuple $(a_1,\ldots,a_l)\in\Z^l$ satisfying the relation $2a_{l}-a_{l-1}=0$. Hence, any such $x$ is of the form:
	\begin{equation*}
			x=\sum_{j\neq l-1,l}a_j\alpha_j+2a_{l}\alpha_{l-1}+a_l\alpha_l= \sum_{j\neq l-1,l}a_j\alpha_j+a_{l}(2\alpha_{l-1}+\alpha_l).
	\end{equation*}
	\noindent
	It follows that the subset $\{\alpha_j,2\alpha_{l-1}+\alpha_l:j\neq l-1,l\}\subseteq Q(C_l)=T'$ generates $H_\alpha\cap T'$. Extending it by including $\alpha_{l-1}$ yields a generating set for $T'$ of cardinality $l$. Since free abelian groups of finite rank are Hopfian, the original set $\{\alpha_j,2\alpha_{l-1}+\alpha_l:j\neq l-1,l\}$ must consist of $\Z$-linearly independent vectors. Therefore, it is a free abelian basis of $H_\alpha\cap T'$ of size $l-1$.
	
	\smallskip\noindent
	\underline{{\bf Case 4}}. $R=D_l$, and $T'=Q(D_l)=FL_l$, with $l\geq6$ even.
	
	\smallskip\noindent
	\if{In this case, $T'$ is the free $\Z$-linear span of the basis $B=\{\alpha_i:1\leq i\leq l\,\}$ of $C_l$ such that:
	\begin{equation}\label{eq.4 - case D_l lemma logical characterization of affine reflections}\tag{$\star_8$}
		\alpha_i=\epsilon_i-\epsilon_{i+1},\quad\text{and}\quad\alpha_l=\epsilon_{l-1}+\epsilon_l,\quad\text{for all}\quad i\in\{1,\ldots,l-1\}.
	\end{equation} 
	\noindent
	(see Fact\cspace\ref{fact - basis of a root system}). Since by assumption  $\alpha\in B$, $\langle \alpha\rangle_\R\cap T'$ is necessarily a rank-$1$ subgroup of $T'$ freely generated by $\alpha$ (cf. Fact\cspace\ref{fact - every vector in Q(R) l.d. from an indivisible root alpha, is Z-proposrtional to alpha}). To compute the rank of $H_\alpha\cap T'$, observe that each $x\in T'$ has a unique expression as a linear combination of the form:
	\begin{equation}\label{eq.5 - case D_l lemma logical characterization of affine reflections}\tag{$\star_9$}
				x=\sum_{i=1}^la_i\alpha_i
			=a_1\epsilon_1+\sum_{j=2}^{l-2}(a_j-a_{j-1})\epsilon_j+(a_{l-1}-a_{l-2}+a_l)\epsilon_{l-1}+(a_l-a_{l-1})\epsilon_l,
	\end{equation}
	\noindent
	for some $a_1,\ldots,a_l\in\Z$. From \eqref{eq.4 - case D_l lemma logical characterization of affine reflections}, we distinguish five cases.
	
	\smallskip\noindent
	\underline{{\bf Case 4.1}}. $\alpha=\alpha_1=\epsilon_1-\epsilon_2$.
	
	\smallskip\noindent
	First, consider the equation $(\alpha,x)=0$, whose solution set in $T'$ is $H_\alpha\cap T'$. By \eqref{eq.5 - case D_l lemma logical characterization of affine reflections}, each $x\in H_\alpha\cap T'$ is uniquely determined by a tuple $(a_1,\ldots,a_l)\in \Z^l$ such that $2a_1-a_2=0$. It follows that any such an $x$ is of the form:
	\begin{equation*}
		x=\sum_{j\neq1,2}a_j\alpha_j+a_1(\alpha_1+2\alpha_2),
	\end{equation*}
	\noindent
	for some $a_1,a_3,a_l\in\Z$. Therefore, $\{\alpha_j,\alpha_1+2\alpha_2:j\neq 1,2\}$ is a generating set for $H_\alpha\cap T'$. By Hopfianity,
    in order to show the $\Z$-linear independence it suffices to exhibit an element completing $\{\alpha_j,\alpha_1+2\alpha_2:j\neq 1,2\}$ to a generating set of $T'$. Observe that $\alpha_2$ is as required. Consequently, $\{\alpha_2,\alpha_j,\alpha_1+2\alpha_2:j\neq 1,2\}$ is a basis of $T'$, and hence $\{\alpha_j,\alpha_1+2\alpha_2:j\neq 1,2\}$ is a basis of $H_\alpha\cap T'$ of cardinality $l-1$.
	
	\smallskip\noindent
	\underline{{\bf Case 4.2}}. $\alpha=\alpha_i=\epsilon_i-\epsilon_{i+1}$, for some $i\in\{2,\ldots,l-3\}$.
	
	\smallskip\noindent
	Similarly, by imposing  the condition $(\alpha,x)=0$ to \eqref{eq.5 - case D_l lemma logical characterization of affine reflections}, one proves that each $x\in H_\alpha\cap T'$ uniquely corresponds to a tuple $(a_1,\ldots,a_l)\in \Z^l$ such that $-a_{i-1}+2a_i-a_{i+1}=0$. Therefore every $x\in H_\alpha\cap T'$ is of the form:
	\begin{equation*}
		\begin{split}
				x&=\sum_{j\neq i+1}a_j\alpha_j+(2a_i-a_{i-1})\alpha_{i+1}\\
				&=\sum_{j\neq i-1,i,i+1}a_j\alpha_j+ a_{i-1}(\alpha_{i-1}-\alpha_{i+1})+a_i(\alpha_i+2\alpha_{i+1}),
		\end{split}
	\end{equation*}
	\noindent
	for some $a_1,\ldots,a_i,a_{i+2},\ldots,a_l\in\Z$. This proves that $\{\alpha_j,\alpha_{i-1}-\alpha_{i+1},\alpha_i+2\alpha_{i+1}:j\neq i-1,i,i+1\}$ is a generating set for $H_\alpha\cap T'$. Since by adding $\alpha_{i+1}$ to $\{\alpha_j,\alpha_{i-1}-\alpha_{i+1},\alpha_i+2\alpha_{i+1}:j\neq i-1,i,i+1\}$ we obtain a generating set for $T'$ of cardinality $l$,
    Hopfianity ensures that $\{\alpha_{i+1},\alpha_j,\alpha_{i-1}-\alpha_{i+1},\alpha_i+2\alpha_{i+1}:j\neq 1,2\}$ is a free abelian basis of $T'$, and hence $\{\alpha_j,\alpha_{i-1}-\alpha_{i+1},\alpha_i+2\alpha_{i+1}:j\neq 1,2\}$ is a free abelian basis of $H_\alpha\cap T'$ of cardinality $l-1$.
	
	\smallskip\noindent
	\underline{{\bf Case 4.3}}. $\alpha=\alpha_{l-2}=\epsilon_{l-2}-\epsilon_{l-1}$.
	
	\smallskip\noindent
	From \eqref{eq.5 - case D_l lemma logical characterization of affine reflections} and the equation $(\alpha,x)=0$, we obtain that each $x\in H_\alpha\cap T'$ uniquely determined by a tuple $(a_1,\ldots,a_l)\in \Z^l$ such that $-a_{l-3}+2a_{l-2}-a_{l-1}-a_l=0$. Therefore, every $x\in H_\alpha\cap T'$ is of the form:
	\begin{equation*}
		\begin{split}
			x&=\sum_{j\neq l}a_j\alpha_j+(-a_{l-3}+2a_{l-2}-a_{l-1})\alpha_{l}\\
			&=\sum_{j\neq l-3, l-2,l-1,l}a_j\alpha_j+ a_{l-3}(\alpha_{l-3}-\alpha_{l})+a_{l-2}(\alpha_{l-2}+2\alpha_{l})+a_{l-1}(\alpha_{l-1}-\alpha_l),
		\end{split}
	\end{equation*}
	\noindent
	for some $a_1,\ldots,a_{l-1}\in\Z$. It follows that
	\begin{equation}\label{eq.6 - case D_l lemma logical characterization of affine reflections}\tag{$\star_{10}$}
		\{\alpha_j,\alpha_{l-3}-\alpha_{l},\alpha_{l-2}+2\alpha_{l},\alpha_{l-1}-\alpha_l:j\neq l-3, l-2,l-1,l\}
	\end{equation}
	\noindent
	is a generating set for $H_\alpha\cap T'$. Moreover,
    we can complete \eqref{eq.6 - case D_l lemma logical characterization of affine reflections} to a free abelian basis of $T'$ by extending it to $\alpha_{l}$. In particular, $\{\alpha_j,\alpha_{i-1}-\alpha_{i+1},\alpha_i+2\alpha_{i+1}:j\neq 1,2\}$ is a basis of $H_\alpha\cap T'$ of cardinality $l-1$.
	
	\smallskip\noindent
	\underline{{\bf Case 4.4}}. $\alpha=\alpha_{l-1}=\epsilon_{l-1}-\epsilon_l$.
	
	\smallskip\noindent
	As in the previous cases, by applying $(\alpha,x)=0$ to \eqref{eq.6 - case D_l lemma logical characterization of affine reflections}, one derives that each $x\in H_\alpha\cap T'$ is completely determined by a tuple $(a_1,\ldots,a_l)\in\Z^l$ such that $a_{l-2}+2a_{l-1}=0$. Hence, any $x\in H_\alpha\cap T'$ has expression of the form:
	\begin{equation*}
		\begin{split}
			x&=\sum_{j\neq l-2}a_j\alpha_j-2a_{l-1}\alpha_{l-2}\\
			&= \sum_{j\neq l-2,l-1}a_j\alpha_j+a_{l-1}(\alpha_{l-1}-2\alpha_{l-2}),
		\end{split}
	\end{equation*}
	\noindent
	for some $a_1,\ldots,a_{l-3},a_{l-1},a_l\in\Z$. Consequently, $\{\alpha_j,\alpha_{l-1}-2\alpha_{l-2}:j\neq l-2,l-1\}$ is a generating set of $H_\alpha\cap T'$. By extending it with $\alpha_{l-2}$, we obtain a generating set for $T'$ of cardinality $l$. Therefore, Hopfianity
    ensures that $\{\alpha_{l-2},\alpha_j,\alpha_{l-1}-2\alpha_{l-2}:j\neq l-1,l\}$ is a free abelian basis of $T'$. In particular, the vectors in $\{\alpha_j,\alpha_{l-1}-2\alpha_{l-2}:j\neq l-1,l\}$ are $\Z$-linear independent, and hence $H_\alpha\cap T'$ has rank $l-1$.
	
	\smallskip\noindent
	\underline{{\bf Case 4.5}}. $\alpha=\alpha_l=\epsilon_{l-1}+\epsilon_l$.
	
	\smallskip\noindent
	Arguing as above, from the equation $(\alpha,x)=0$, with $x$ as in \eqref{eq.4 - case D_l lemma logical characterization of affine reflections}, one derives that each $x\in H_\alpha\cap T'$ is determined by a tuple $(a_1,\ldots,a_l)\in\Z^l$ such that $2a_l-a_{l-2}=0$. Hence, any such a $x$ is of the form:
	\begin{equation*}
			x=\sum_{j\neq l-2}a_j\alpha_j+2a_{l}\alpha_{l-2}= \sum_{j\neq l-2,l}a_j\alpha_j+a_{l}(2\alpha_{l-2}+\alpha_l).
	\end{equation*}
	\noindent
	It follows that $\{\alpha_j,2\alpha_{l-2}+\alpha_l:j\neq l-2,l\}$ is a generating set of $H_\alpha\cap T'$. By extending it with $\alpha_{l-2}$, we obtain a generating set for $T'$ of cardinality $l$. Consequently, Fact\cspace\ref{fact - free abelian groups of finite rank are Hopfian}(3) ensures that $\{\alpha_{l-2},\alpha_j,2\alpha_{l-2}+\alpha_l:j\neq l-2,l\}$ is a free abelian basis of $T$, and hence all the vectors in $\{\alpha_j,2\alpha_{l-2}+\alpha_l:j\neq l-2,l\}$ are $\Z$-linear independent. This proves that $H_\alpha\cap T'$ has rank $l-1$, as required.
	}\fi
    Analogous to the previous. This concludes the proof of $[1\Rightarrow2]$.

	\medskip \noindent
	For the implication $[2\Rightarrow1]$, assume that $T'\cap\mathrm{ker}(1-s)$ and $T'\cap\mathrm{ker}(1+s)$ are free abelian groups of rank $l-1$ and $1$, respectively. Then, $\mathrm{ker}(1+s)$ is a subspace of $V$ of dimension at least $1$. We show that $\mathrm{ker}(1-s)$ has dimension at least $l-1$. By Remark\cspace\ref{remark - all lattices with n>2 is of type Q(R) or P(R)}, $T'$ is either the group $P(R)$ of weights, or the group $Q(R)$ of radical weights of $R$. In both cases, it follows from Fact\cspace\ref{fact - Q(R) is always a Q-structure in V} and \ref{fact - P(R) is a Q-structure on V} that $T'$ is a $\Q$-structure on $V$.
	
	\smallskip\noindent
	Let $\{x_1,\ldots,x_{l-1}\}$ be a free abelian basis of $T'$. In particular, these elements are linearly independent over $\Z$. We claim that they are also linearly independent over $\R$. Assume otherwise. Then, by definition of $\Q$-structure (cf. \ref{prop. - characterization Q-structures}(2)), $x_1,\ldots,x_{l-1}$ must be linearly dependent over $\Q$. That is, there exist $a_1,\ldots,a_{l-1}\in\Z$, not all zero, and $b_1,\ldots,b_{l-1}\in\N^+$ such that:
    \begin{equation}\label{eq.1 - direction 2->1 lemma characterization of reflections}\tag{$\star_{8}$}
        \frac{a_1}{b_1}x_1+\frac{a_2}{b_2}x_2+\ldots+\frac{a_{l-1}}{b_{l-1}}x_{l-1}=0.
    \end{equation}
	\noindent
    Let $c\in\N^+$ be the least common multiple of $b_1,\ldots,b_{l-1}$. Since each $c\cdot\frac{a_i}{b_i}\in\Z$, \eqref{eq.1 - direction 2->1 lemma characterization of reflections} yields the trivial $\Z$-linear combination:
	\begin{equation*}
    \begin{split}
        0&=c\left(\frac{a_1}{b_1}x_1+\frac{a_2}{b_2}x_2+\ldots+\frac{a_{l-1}}{b_{l-1}}x_{l-1}\right)\\
        &=a_1\frac{c}{b_1}x_1+a_2\frac{c}{b_2}x_2+\ldots+a_{l-1}\frac{c}{b_{l-1}}x_{l-1}.
    \end{split}
	\end{equation*}
    By hypothesis, $x_1,\ldots,x_{l-1}$ are linearly independent over $\Z$. Therefore, the above identity implies that $a_i\frac{c}{b_i}=0$, for all $i\in[1,l-1]$. Since each $\frac{c}{b_i}\neq0$, this means that $a_1=\ldots=a_{l-1}=0$, contradicting the assumption that the $a_i$'s are not all zero. This proves that $x_1,\ldots,x_{l-1}$ are linearly independent over $\R$.
    
    \smallskip\noindent
    Consequently, $\mathrm{ker}(1-s)$ and $\mathrm{ker}(1+s)$ are real vector spaces of dimension at least $l-1$ and $1$, respectively. By Fact\cspace\ref{fact - every involution induces a decomposition of V in ker(1-s)+ker(1+s)}, $s$ induces a decomposition of $V$ as a direct sum $V=\mathrm{ker}(1-s)\oplus\mathrm{ker}(1+s)$.
	Therefore, $\mathrm{ker}(1-s)$ must have dimension exactly $l-1$, and, by Definition\cspace\ref{def - reflection}, this confirms that $s$ is an affine reflection in $V$.
	
	\smallskip\noindent
	Finally, the equivalence $[2\Leftrightarrow3]$ follows directly from the fact that the conjugation $(\,\cdot\,)^u$ on $T$ corresponds exactly to the action of $s$ on $T'$ via the natural $\Z[W_0(R)]$-lattice isomorphism between $L=\{(x,1):x\in L'\}$ and $L'$. More precisely, for each $x\in T'$, we have:
	\begin{equation*}
		\begin{split}
			(x,1)^u=(t,s)(x,1)(t,s)^{-1}=(t+s(x)-s(s^{-1}(t)),1)=(s(x),1).
		\end{split}
	\end{equation*}
	This completes the proof.
\end{proof}
\begin{lemma}\label{lemma - affine projection and quotient projection coincide}
	Let $R$ be a root system in a real Euclidean vector space $V$, and $L'$ be a lattice in $V$ stable under the action of the Weyl group $W_0(R)$ of $R$. Suppose that $W$ is an affine crystallographic group given by the presentation \eqref{eq.1 - martinais' groups presentation}, where $s_1,\ldots,s_{l}$ is a system of simple reflections for $R$, and $T=\{(x,1):x\in L'\}$ denotes the translation subgroup of $W$. Then, the function:
	\begin{equation}\label{eq.1 - isomorphism theta:W/T->W_0(R)}\tag{$\star_{9}$}
		\theta: W/T\rightarrow W_0(R),\quad (v,g)T\mapsto g
	\end{equation}
	\noindent
	is a group isomorphism.
\end{lemma}
\begin{proof}
	The map $\theta$ is well-defined. If $(v,g),(v',g')\in W$ lie in the same coset of $T$, then $(v',g')(v,g)^{-1}\in T$. Since
	\begin{equation*}
		(v',g')(v,g)^{-1}=(v',g')(-g^{-1}(v),g^{-1})=(v'-g'g^{-1}(v),g'g^{-1}),
	\end{equation*}
	\noindent
	by the definition of $T$, the identity $g=g'$ must hold.
	Similarly, a straightforward computation shows that $\theta:W/T\rightarrow W_0(R)$ is a group homomorphism. Finally, by Fact\cspace\ref{fact - simple reflections form a Coxeter system of the Weyl gp}, $\{s_1,\ldots,s_l\}$ is a set of generators of $W_0(R)$. Thus, the cosets $(t_1,s_1)T,\ldots,(t_l,s_l)T$ witness that $\theta$ is surjective. Since $W/T$ and $W_0(R)$ are finite of the same cardinality, if follows that $\theta$ is also injective. 
\end{proof}
\begin{remark}\label{remark- canonical expression of elements in cosets}
    In the context of Remark\cspace\ref{remark_for_star2}, let $W$ be a crystallographic group given by a presentation of the form \eqref{eq.1 - martinais' groups presentation}, arising from an irreducible root system $R$ in a real Euclidean vector space $V$. Then, its translation subgroup $T$ consists of all elements $(x,1)$ with $x\in T'$, where $T'$ is a discrete subgroup of $V$ listed in Table\cspace\ref{table - martinais nomenclature} (in particular, $T'$ corresponds to the abelian group structure underlying the lattice $L'$ in \ref{remark_for_star2}). Since $T$ is normal in $W$, every element $z\in (t,s)T=T(t,s)$ admits an expression:
    \begin{equation*}
        z=(x,1)(t,s)=(t+x,s),
    \end{equation*}
    for some $x\in T'$.
\end{remark}
%
We now introduce a notation that will simplify the proof of Theorem\cspace\ref{lemma - all crystallographic groups arising from irreducible root systems are f.o. rigid and profinitely rigid}.
\begin{notation}\label{notation - lemma all crystallographic groups arising from irreducible root systems are f.o. rigid and profinitely rigid}
	Let $V,R,S,\Delta_S,L'$ be as follows:
	\begin{enumerate}[$\bullet$]
		\item $V$ is a real Euclidean vector space of finite dimension $l$;
		\item $R$ is an irreducible root system in $V$, with Weyl group $W_0(R)\leq O(V)$;
		\item $S=\{s_1,\ldots,s_l\}$ is a system of simple reflections in $W_0(R)$ induced by a basis $B=\{\alpha_1,\ldots,\alpha_l\}$ of $R$;
		\item $\Delta_S$ is the Coxeter diagram of the system $(W_0(R),S)$ (cf.\cspace\ref{fact - simple reflections form a Coxeter system of the Weyl gp});
		\item $L'$ is a ($\Z[W_0(R)]$-)lattice from Table\cspace\ref{table - martinais nomenclature} such that $n(W_0(R),L')\geq 3$ (cf. Table\cspace\ref{table all inequivalent lattices}).
	\end{enumerate}
	\noindent
	Then, by Lemma\cspace\ref{lemma_the_prop_def_transla_crysp}, the translation subgroup of any crystallographic group extending $W_0(R)$ by $L'$ is $\emptyset$-definable by a first-order formula (in the language of groups) 
	depending only on the order of $W_0(R)$. For the ease of reading, we denote by $T$ the set of realizations of such formula.
	
	\smallskip\noindent
	For each $g\in W_0(R)$ we fix a group word $w_g(x_1,\ldots,x_l)$ such that $g=w_g(s_1,\ldots,s_l)$.
	Then, we let $\eta_{\Delta_S}(x_1,\ldots,x_l)$ be the first-order formula (in the language of groups) asserting that:
	\begin{enumerate}[(1)]
		\item $\langle x_1T,\ldots,x_lT\rangle\cong W_0(R)$ via the map $s_i\mapsto x_iT$, for $i\in[1,l\,]$;
		\item for each $g\in W_0(R)$, $g$ belongs to $S^{W_0(R)}=\{ws_iw^{-1}:1\leq i\leq l, w\in W_0(R)\}$ if and only if the following conditions are simultaneously satisfied by some (or, equivalently, any) $x\in w_g(x_1,\ldots,x_l)T$:
		\begin{enumerate}[(a)]
			\item $T^{x}\vcentcolon=\{t\in T:(t)^{x}=t\}$ has quotient $T^{x}/2T^{x}$ of cardinality $2^{\,l-1}$;
			\item $T_{x}\vcentcolon=\{t\in T:(t)^{x}=-t\}$ has quotient $T_{x}/2T_{x}$ of cardinality $2$.
		\end{enumerate}
	\end{enumerate}
	\noindent
	Since $W_0(R)$ is finite and the set of words $w_g(x_1,\ldots,x_l)$ is fixed, these conditions are clearly first-order expressible.	
\end{notation}
\begin{remark}\label{notation_2 - lemma all crystallographic groups arising from irreducible root systems are f.o. rigid and profinitely rigid}
	In the setting of Notation\cspace\ref{notation - lemma all crystallographic groups arising from irreducible root systems are f.o. rigid and profinitely rigid}, let $W$ be a subgroup of $\iso(V)$ with presentation \eqref{eq.1 - martinais' groups presentation}, and $(u'_1,\ldots,u'_l)$ be a tuple realizing $\eta_{\Delta_S}(x_1,\ldots,x_l)$ in $W$ such that $u_i'=(t'_i,s_i')$ for all $i\in[1,l\,]$. Then, by Lemma\cspace\ref{lemma - affine projection and quotient projection coincide}, item $(1)$ in the definition of $\eta_{\Delta_S}(x_1,\ldots,x_l)$ just states that the assignment $s_i\mapsto s'_i$ extends to an automorphism $f$ of $W_0(R)$. This automorphism maps each involution $g=w_g(s_1,\ldots,s_l)\in S^{W_0(R)}$ into a term $g'=w_g(s'_1,\ldots,s_l')$. Then, denoting by $\theta:W/T\rightarrow W_0(R)$ the isomorphism in \eqref{eq.1 - isomorphism theta:W/T->W_0(R)}, items $(2)(a)(b)$ in the definition of $\eta_{\Delta_S}(x_1,\ldots,x_l)$ are equivalent to saying that the conjugation by each $u\in \theta^{-1}(g')$ stabilizes two free abelian subgroups $T^u=\{x\in T:x^u=x\}$ and $T_u=\{x\in T:x^u=-x\}$ of rank $l-1$ and $1$, respectively. By Lemma\cspace\ref{lemma - logic characterization of reflections}, this means that $g'\in W_0(R)$ is an affine reflection, and hence $g'\in S^{W_0(R)}$, by Lemma\cspace\ref{lemma - abstract reflections and affine reflections cohincide}. Consequently, by Fact\cspace\ref{fact - every automorphism of a Weyl group of an irreducible root system is inner-by-grapph}, $f\in \mathrm{Aut}(W_0(R))$ is inner-by-graph.
	
	\smallskip\noindent
	In particular, the corresponding automorphism $f_T\in \mathrm{Aut}(W/T)$ induced by $f$ via the identification $\theta^{-1}:W_0(R)\rightarrow W/T$ of Lemma\cspace\ref{lemma - affine projection and quotient projection coincide} is a coset permutation such that:
	\begin{enumerate}[(1)]
		\item $f_T((t_i,s_i)T)=u'_iT$ for each $i\in[1,l\,]$;
		\item $f=(\,\cdot\,)^{u'T}\circ f_\sigma$, for some $u'\in W$ and some graph automorphism $f_\sigma$ of the system $(W/T,\{(t_1,s_1)T,\ldots,(t_l,s_l)T\})$ (where $(t_1,s_1),\ldots,(t_l,s_l)$ are the generators of $W$ from the presentation \eqref{eq.1 - martinais' groups presentation}).
	\end{enumerate}
\end{remark}

\medskip
The last ingredient that we need towards the proof of Theorem\cspace\ref{main_th2} is a straightforward observation on some of the Coxeter diagrams in Table\cspace\ref{table - Coxeter diagrams of irreducible root systems}.
\begin{fact}\label{fact - form of inner-by-graph automorphisms in the case B_l,C_l,D_l}
	Let $R$ be an irreducible reduced root system of rank $l\geq3$ with Coxeter diagram $\Delta=\{s_1,\ldots,s_l\}$ as in Table\cspace\ref{table - Coxeter diagrams of irreducible root systems}. Then,
	\begin{enumerate}[(1)]
		\item if $R=B_l$, the only automorphism of $\Delta$ is the identity;
		\item if $R=D_l$, the only automorphisms of $\Delta$ are the identity and the transposition of the nodes $s_{l-1}$ and $s_l$. 
	\end{enumerate}
\end{fact}
\begin{theorem}\label{lemma - all crystallographic groups arising from irreducible root systems are f.o. rigid and profinitely rigid}
	Every crystallographic group arising from an irreducible root system is profinitely rigid (equivalently, first-order rigid).
\end{theorem}
\begin{proof}
	In light of Fact\cspace\ref{Oger}, it suffices to show that every crystallographic group $W$ arising from an irreducible root system is first-order rigid. 
	
	\smallskip\noindent
	By Fact\cspace\ref{fact - existence reduced root system and sandwich condition}, $W$ admits an affine realization as a group of isometries of a real Euclidean vector space $V$, with point group the Weyl group $W_0(R)$ of an irreducible root system $R$ in $V$, and translation lattice of the form $L=\{(t,1):t\in L'\}$, where $L'$ is a lattice in $V$ satisfying:
	\begin{equation*}
		Q(R)\leq L'\leq P(R).
	\end{equation*}
	\noindent
	By Lemma\cspace\ref{lemma-every_lattice_L_s.t._Q(R)<L<P(R)_is_absolutely_irreducible}, $L'$ and $L$ are absolutely irreducible. Therefore, if $H$ is a finitely generated group elementarily equivalent to $W$, it follows from Lemma\cspace\ref{thm. - absolutely irreducible -> there are only finitely many maximal centerings} and Theorem\cspace\ref{theorem} that $H$ is crystallographic with point group $P(H)\cong W_0(R)$ and translation lattice $L(H)\cong L$. In other words, both $W$ and $H$ are extensions of $W_0(R)$ by $L$. According to the number $n(W_0(R),L)$ of isomorphism classes of such extensions, we distinguish the following cases.	

	\smallskip\noindent
	\underline{{\bf Case 1}}. $n(W_0(R),L)=1$.
	\newline In this case both $W$ and $H$ are split and the result follows immediately.
	
	\smallskip\noindent
	\underline{{\bf Case 2}}. $n(W_0(R),L)=2$.
	\newline In this case there are just two non-isomorphic extensions of $W_0(R)$ by $L$: one split and the other non-split. By Lemma\cspace\ref{lemma-split extensions are definable}, there is a sentence $\psi_R$ (depending only on $R$) that is satisfied only by those extensions of $W_0(R)$ by $L$ that are split. So, since $W$ and $H$ are elementarily equivalent, either they both satisfy $\psi_R$ (and they are split), or they satisfy $\neg\psi_R$ (and they are non-split). In both cases they are isomorphic.
	
	\smallskip\noindent
	From the standard classification of wallpaper groups (cf. \cite[Table\cspace2]{hiller}), it is known that for any Weyl group $W_0(R)$ associated to an irreducible root system $R$ of rank $2$ there are at most two isomorphism classes of $2$-dimensional crystallographic groups having $W_0(R)$ as point group. Hence, in the remaining cases we may assume that the rank $l$ of $R$ is at least $3$. In this setting, both $W$ and $H$ are isomorphic to certain representatives of Martinais's classification (see Table\cspace\ref{table - explicit crystallographic groups}). Consequently, for fixed $R$ and $L$, it suffices to prove that no two such representatives are elementarily equivalent. In particular, in each case we will distinguish the non-split extensions $W_n$ listed in Table\cspace\ref{table - explicit crystallographic groups} by means of sentences of the form:
	\begin{equation}\label{eq.0 - lemma - all crystallographic groups arising from irreducible root systems are f.o. rigid and profinitely rigid}\tag{$\star_{10}$}
		\zeta_m\equiv \forall x_1,\ldots,x_l(\eta_{\Delta_S}(x_1,\ldots,x_l)\rightarrow \chi_m(x_1,\ldots,x_l)),
	\end{equation}
	\noindent
	such that:
	\begin{enumerate}[$\bullet$]
		\item $W_n\models\zeta_m$ if and only if $n=m$;
		\item $\eta_{\Delta_S}(x_1,\ldots,x_l)$ is a formula as in Notation\cspace\ref{notation - lemma all crystallographic groups arising from irreducible root systems are f.o. rigid and profinitely rigid}, depending only on the Coxeter system $(W_0(R),S)$ given by the simple reflections $S=\{s_1,\ldots,s_l\}$ of Table\cspace\ref{table - root systems, bases and reflections}, and a fixed set of group words $w_g(x_1,\ldots,x_l)$;
		\item $\chi_m(x_1,\ldots,x_l)$ is a formula consisting of a quantifier free formula $\tau_m(z_1,\ldots,z_l)$ 
		preceded by quantifications bounded over the cosets $x_1T,\ldots,x_lT$, that is, quantifications of type $\forall z_i\in x_iT$ and $\exists z_i\in x_iT$. In particular, the truth value of $\chi_m(x_1,\ldots,x_l)$ depends only on the cosets $x_1T,\ldots,x_lT$, and not strictly on $x_1,\ldots,x_l$ themselves.
	\end{enumerate}
	\noindent 
	By Remark\cspace\ref{notation_2 - lemma all crystallographic groups arising from irreducible root systems are f.o. rigid and profinitely rigid}, if $(u'_1,\ldots,u'_l)$ is a tuple realizing $\eta_{\Delta_S}(x_1,\ldots,x_l)$ in $W_n$ and $u_1=(t_1,s_1),\ldots,u_l=(t_l,s_l)$ are the generators in the presentation of $W_n$ in Table\cspace\ref{table - explicit crystallographic groups} corresponding to the simple reflections $s_1,\ldots,s_l$, then there exist some $u\in W_n$ and a graph automorphism $f_\sigma$ of $(W_n/T,\{u_1T,\ldots,u_lT\})$ such that, for every $i\in[1,l\,]$:
	\begin{equation}\label{eq.1 - lemma - all crystallographic groups arising from irreducible root systems are f.o. rigid and profinitely rigid}\tag{$\star_{11}$}
		u'_iT=(f_\sigma(u_iT))^{uT}.
	\end{equation}
	\noindent
	Since $f_\sigma$ acts as a permutation of $\{u_1T,\ldots,u_lT\}$, it can be identified with the corresponding index permutation $\sigma$ such that $f_\sigma(u_iT)=u_{\sigma(i)}T$ for each $i\in[1,l\,]$. Consequently, \eqref{eq.1 - lemma - all crystallographic groups arising from irreducible root systems are f.o. rigid and profinitely rigid} reduces to the following coset identity:
	\begin{equation}\label{eq.2 - lemma - all crystallographic groups arising from irreducible root systems are f.o. rigid and profinitely rigid}\tag{$\star_{12}$}
		u'_iT=(u_{\sigma(i)}T)^{uT}=uu_{\sigma(i)}u^{-1}T=(u_{\sigma(i)})^uT.
	\end{equation}
	\noindent
	By construction, the truth value of $\chi_m(x_1,\ldots,x_l)$ only depends on the cosets $x_1T,\ldots,x_lT$ and not on $x_1,\ldots,x_l$ themselves. Hence, \eqref{eq.2 - lemma - all crystallographic groups arising from irreducible root systems are f.o. rigid and profinitely rigid} and the fact that  $(\,\cdot\,)^u\in\mathrm{Aut}(W_n)$ witness that:
	\begin{equation*}
		\begin{split}
			W_n\models\chi_m(u_1',\ldots,u_l') \quad&\iff\quad W_n\models\chi_m((u_{\sigma(1)})^u,\ldots,(u_{\sigma(l)})^u)\\
			&\iff\quad W_n\models\chi_m(u_{\sigma(1)},\ldots,u_{\sigma(l)}).
		\end{split}
	\end{equation*}
	\noindent
	In all but one of the cases considered below, the root system $R$ is of type $B_l$ or $C_l$ (cf. Table\cspace\ref{table all inequivalent lattices}), and by Fact\cspace\ref{fact - form of inner-by-graph automorphisms in the case B_l,C_l,D_l} the permutation $\sigma$ is trivial. In the remaining case, $R$ is of type $D_l$, and the only non-trivial graph automorphism corresponds to the transposition of the indices $l-1$ and $l$. In this instance, we select the formulas $\chi_m(x_1,\ldots,x_l)$ to be symmetric in the variables $x_{l-1}$ and $x_l$, and hence invariant under the action of $\sigma$. In both cases, we have:
	\begin{equation*}
		W_n\models\chi_m(u_{\sigma(1)},\ldots,u_{\sigma(l)})\quad\iff\quad W_n\models\chi_m(u_1,\ldots,u_l). 
	\end{equation*}
	\noindent	
	This procedure allows to verify the truth value of a universal statement by a direct analysis of the cosets of the generators $u_1=(t_1,s_1),\ldots,u_l=(t_l,s_l)$. We will use this fact freely in what follows.
	
	\noindent\smallskip
	\underline{{\bf Case 3}}. $n(W_0(R),L)=3$.
	\newline In this case, according to Tables\cspace\ref{table all inequivalent lattices} and \ref{table - explicit crystallographic groups}, we must have $R=D_l$ and $L=FL_l$, for some even $l\geq 6$. A complete system of representatives of the isomorphism classes of crystallographic groups extending $W_0(R)$ by $L$ is the following:
	\begin{enumerate}[$\bullet$]
		\item $W_1=\langle (\sum_{i=1}^la_i\epsilon_i,1),(0,s_j):1\leq j\leq l, a_i\in\Z, \sum_{i=1}^la_i\in2\Z\rangle_{\iso(V)}$;
		\item $W_2=\langle (\sum_{i=1}^la_i\epsilon_i,1),(\epsilon_1,s_j):1\leq j\leq l, a_i\in\Z, \sum_{i=1}^la_i\in2\Z\rangle_{\iso(V)}$;
		\item \mbox{$W_3=\langle (\sum_{i=1}^la_i\epsilon_i,1),(\frac{1}{2}\sum_{i=1}^l\epsilon_i,s_j):1\leq j\leq l, a_i\in\Z, \sum_{i=1}^la_i\in2\Z\rangle_{\iso(V)}$.}
	\end{enumerate}
	\noindent
	Note that $W_1$ is split. Hence, by Lemma\cspace\ref{lemma-split extensions are definable}, $W_1\models \psi_{D_l}$ and $W_2,W_3\models\neg\psi_{D_l}$.
	
	\smallskip\noindent
	Let $\zeta_2,\zeta_3$ be sentences of the form \eqref{eq.0 - lemma - all crystallographic groups arising from irreducible root systems are f.o. rigid and profinitely rigid}, and $\chi_2(x_1,\ldots,x_l),\chi_3(x_1,\ldots,x_l)$ be the first-order formulas defined as follows: 
	\begin{enumerate}[$\bullet$]
		\item $\chi_2(x_1,\ldots,x_l)$ states that there exist $z_{l-1}\in x_{l-1}T$ and $z_l\in x_lT$ such that $z_{l-1}^2= z_l^2$;
		\item $\chi_3(x_1,\ldots,x_l)$ states that $z_{l-1}^2\neq z_l^2$ for all $z_{l-1}\in x_{l-1}T$ and $z_l\in x_lT$.
	\end{enumerate}

	\begin{claim}
		For each $n,m\in\{2,3\}$, $W_n\models\zeta_m$ if and only if $n=m$.
	\end{claim}
	
	\smallskip\noindent
	\begin{claimproof}
	Note that both $\chi_2(x_1,\ldots,x_l)$ and $\chi_3(x_1,\ldots,x_l)$ are symmetric in the variables $x_{l-1}$ and $x_l$. Hence, in light of the discussion above, it suffices to show that $((\epsilon_1,s_j):1\leq j\leq l)$ realizes $\chi_2(x_1,\ldots,x_l)$ in $W_2$, while $((\frac{1}{2}\sum_{i=1}^l\epsilon_i,s_j):1\leq j\leq l)$ realizes its negation in $W_3$.
	
	\smallskip\noindent
	The first claim is immediate. It follows from the fact that, in this case, the reflection $s_{l-1}$ acts on the canonical basis $\{\epsilon_1,\ldots,\epsilon_l\}$ by transposing $\epsilon_{l-1}$ and $\epsilon_l$, and $s_l$ maps $\epsilon_{l-1}$ to $-\epsilon_l$ and $\epsilon_l$ to $-\epsilon_{l-1}$, while leaving the other $\epsilon_i$'s unchanged (cf. Table\cspace\ref{table - root systems, bases and reflections}). From this, one readily derives the following identities:
	\begin{equation*}
		(\epsilon_1,s_{l-1})^2=(2\epsilon_1,1)=(\epsilon_1,s_l)^2.
	\end{equation*}
	\noindent
	This proves that $W_2\models\zeta_2$, and hence $W_2\not\models\zeta_3$. It only remains to show that in $W_3$ there are no $z_{l-1}\in(\frac{1}{2}\sum_{i=1}^l\epsilon_i,s_{l-1})T$ and $z_l\in(\frac{1}{2}\sum_{i=1}^l\epsilon_i,s_l)T$ with the same square. 

	\smallskip\noindent
	By Remark\cspace\ref{remark- canonical expression of elements in cosets} and the explicit description of $FL_l$ in Table\cspace\ref{table - martinais nomenclature}, each $z_{l-1}\in(\frac{1}{2}\sum_{i=1}^l\epsilon_i,s_{l-1})T$ and $z_{l}\in(\frac{1}{2}\sum_{i=1}^l\epsilon_i,s_{l})T$ is of the form:
		\begin{equation*}
			z_{l-1}=\left(\sum_{i=1}^l\left(a^{l-1}_i+\frac{1}{2}\right)\epsilon_i,s_j\right)\quad\text{and}\quad z_l=\left(\sum_{i=1}^l\left(a^l_i+\frac{1}{2}\right)\epsilon_i,s_j\right),
		\end{equation*}
		\noindent
		for some $a_1^{l-1},\ldots,a_l^{l-1},a_1^l,\ldots,a_l^l\in\Z$ such that $\sum_{i=1}^la_i^{l-1},\sum_{i=1}^la_i^l\in 2\Z$. Consequently, one derives the following expressions for the squares:
		\begin{equation*}
			\begin{split}
				z_{l-1}^2&=\left(\sum_{i\neq l-1,l}(2a^{l-1}_i + 1)\epsilon_i+(a^{l-1}_{l-1}+a_l^{l-1} + 1)\epsilon_{l-1}+(a^{l-1}_{l-1}+a_l^{l-1} + 1)\epsilon_{l},1\right),\\
				z_{l}^2&=\left(\sum_{i\neq l-1,l}(2a^{l}_i + 1)\epsilon_i+(a^{l}_{l-1}-a_l^{l})\epsilon_{l-1}+(a^{l}_{l} -a_{l-1}^{l})\epsilon_{l},1\right).
			\end{split}
		\end{equation*}
		\noindent
		Assume, for a contradiction, that $z_{l-1}^2=z_l^2$. Then, by equating the coefficients of the expressions above with respect to the canonical basis $\{\epsilon_1,\ldots,\epsilon_l\}$, we obtain a compatible system over $\Z$:
		\begin{equation*}
			\begin{cases}
				&a_i^l=a_i^{l-1}\phantom{a_l^{l-1}+a_{l-1}^{l-1}+1}\quad\quad\forall i\neq l-1,l\\
				&a^l_{l-1}-a_l^l=a_l^{l-1}+a_{l-1}^{l-1}+1\\
				&a_l^l-a_{l-1}^l=a_l^{l-1}+a_{l-1}^{l-1}+1.
			\end{cases}
		\end{equation*}
		\noindent
		From this, by summing the last two equations, one readily derives the equivalent system:
		\begin{equation}\tag{$\star_{13}$}\label{eq.1 - case D_l}
			\begin{cases}
				&a_i^l=a^{l-1}_i\phantom{a^l_l+a^{l-1}_{l-1}+a^l_{l-1}+1}\forall i\neq l-1,l\\
				&a^{l-1}_{l-1}=-a^{l-1}_l-1\\
				&a^l_l=a^{l-1}_l+a^{l-1}_{l-1}+a^l_{l-1}+1.
			\end{cases}
		\end{equation}
		\noindent
		It follows from the second equation in \eqref{eq.1 - case D_l} that $a^{l-1}_{l-1}+a^{l-1}_l\in 2\Z+1$. Since by assumption, $\sum_{i=1}^{l}a^{l-1}_i\in 2\Z$, the first equation in \eqref{eq.1 - case D_l} then forces:
		\begin{equation}\tag{$\star_{14}$}\label{eq.2 - case D_l}
			\sum_{i\neq l-1,l}a_i^l=\sum_{i\neq l-1,l}a_i^{l-1}\in2\Z+1.
		\end{equation}
		\noindent
		Similarly, substituting $a^{l-1}_{l-1}=-a^{l-1}_l-1$ into the third equation of \eqref{eq.1 - case D_l} yields $a^l_l=a^l_{l-1}$, and thus:
		\begin{equation*}
			a^l_{l-1}+a^l_l=2a^l_{l-1}\in2\Z.
		\end{equation*} 
		\noindent
		By \eqref{eq.2 - case D_l}, this entails that $\sum_{i=1}^la_i^l\in2\Z+1$, contradicting the definition of $z_l$. This concludes the proof of the claim.
	\end{claimproof}
	
	\noindent\smallskip
	\underline{{\bf Case 4}}. $n(W_0(R),L)=4$.
	\newline In this case, according to Table\cspace\ref{table - explicit crystallographic groups}, exactly one of the following occurs:
	\begin{enumerate}[(1)]
		\item $R=B_l$, with $l\geq3$, and $L=CL_l$;
		\item $R=B_4$ and $L=CCL_4$;
		\item $R=C_l$, with $l\geq3$ odd, and $L=FL_l$.
	\end{enumerate}

	\noindent
	We show that in each of these cases it is possible to find first-order sentences distinguishing the crystallographic groups extending $W_0(R)$ by $L$.
	
	\noindent\smallskip
    \underline{{\bf Case 4.1}}. $R=B_l$, with $l\geq3$, and $L=CL_l$;
    \newline In this case, by Table\cspace\ref{table - explicit crystallographic groups}, a complete system of representatives of the isomorphism classes of crystallographic groups extending $W_0(R)$ by $L$ is the following:
	\begin{enumerate}[$\bullet$]
		\item $W_1=\langle(\sum_{i=1}^lx_i\epsilon_i,1),(0,s_j):1\leq j\leq l, x_i\in \Z \rangle_{\iso(V)}$;
		\item $W_2=\langle(\sum_{i=1}^lx_i\epsilon_i,1),(0,s_j),(\frac{1}{2}\sum_{i=1}^l\epsilon_i,s_l):1\leq j< l, x_i\in \Z \rangle_{\iso(V)}$;
		\item $W_3=\langle(\sum_{i=1}^lx_i\epsilon_i,1),(\frac{1}{2}\sum_{i=1}^l\epsilon_i,s_j),(0,s_l):1\leq j< l, x_i\in \Z \rangle_{\iso(V)}$;
		\item $W_4=\langle(\sum_{i=1}^lx_i\epsilon_i,1),(\frac{1}{2}\sum_{i=1}^l\epsilon_i,s_j):1\leq j\leq l, x_i\in \Z \rangle_{\iso(V)}$.
	\end{enumerate}
	\noindent
	Since $W_1$ is split, by Lemma\cspace\ref{lemma-split extensions are definable} there exists a first-order sentence $\psi_{B_l}$ such that $W_1\models\psi_{B_l}$, while $W_n\models\neg\psi_{B_l}$ for all $n\in\{2,3,4\}$.
	
	\smallskip\noindent
	Let $\zeta_2,\zeta_3,\zeta_4$ be sentences of the form \eqref{eq.0 - lemma - all crystallographic groups arising from irreducible root systems are f.o. rigid and profinitely rigid}, and $\chi_2(x_1,\ldots,x_l),\chi_3(x_1,\ldots,x_l),\chi_4\\(x_1,\ldots,x_l)$ be the first-order formulas defined as follows: 
	\begin{enumerate}[$\bullet$]
		\item $\chi_2(x_1,\ldots,x_l)$ states that there exists an involution in each $x_iT$, for $i\in[1,l-1]$, but there are no involutions in $x_lT$;
		\item $\chi_3(x_1,\ldots,x_l)$ states that there exists an involution in $x_lT$, but there are no involutions in $x_iT$, for all $i\in[1,l-1]$;
		\item $\chi_4(x_1,\ldots,x_l)$ states that there are no involutions in any of the cosets $x_1T,\ldots,x_lT$.
	\end{enumerate}
	
	\begin{claim}
		For every $n,m\in\{2,3,4\}$, $W_n\models\zeta_m$ if and only if $n=m$.
	\end{claim}

    \smallskip\noindent
	\begin{claimproof}
	    For ease of reading, for each $n\in\{2,3,4\}$, we denote by $(t_1^n,s_1),\ldots,(t_l^n,s_l)$ the generators in the presentation of $W_n$ above, so that, for example, $(t_1^2,s_1)=(0,s_1)$, while $(t^3_1,s_1)=(t^4_1,s_1)=(\frac{1}{2}\sum_{i=1}^l\epsilon_i,s_1)$.

        \smallskip\noindent
        Since the $\chi_m(x_1,\ldots,x_l)$'s are mutually inconsistent, it suffices to show that for each $n\in\{2,3,4\}$, $W_n\models\chi_n((t^n_1,s_1),\ldots,(t^n_l,s_l))$. This is immediate for $W_2$, where we have:
        \begin{equation*}
        	(t^2_j,s_j)^2=(0,s_j)^2=(0,1)\quad \forall j\in[1,l-1].
        \end{equation*}
        
        \smallskip\noindent
        For the remaining cases, note that, by Remark\cspace\ref{remark- canonical expression of elements in cosets} and the explicit description of $CL_l$ in Table\cspace\ref{table - martinais nomenclature}, each element $z_{j}\in(\frac{1}{2}\sum_{i=1}^l\epsilon_i,s_{j})T$ for $j\in[1,l]$ has the form:
        \begin{equation}\label{eq.1 - main theorem - case B_l, CL_l}\tag{$\star_{15}$}
            z_j=\left(\,\sum_{i=1}^la_i^j\epsilon_j+\frac{1}{2}\sum_{i=1}^l\epsilon_i,s_j\right)=\left(\,\sum_{i=1}^l\left(a_i^j+\frac{1}{2}\right)\epsilon_i,s_j\right),
        \end{equation}
        for some $a_1^j,\ldots,a_l^j\in\Z$. As indicated in Table\cspace\ref{table - root systems, bases and reflections}, in this context, for each $j\in[1,l-1]$, the reflection $s_j$ acts on the canonical basis $\{\epsilon_1,\ldots,\epsilon_l\}$ as the transposition of $\epsilon_j$ and $\epsilon_{j+1}$; while $s_l$ maps $\epsilon_l$ to $-\epsilon_l$, leaving the other $\epsilon_i$'s unchanged. Consequently, \eqref{eq.1 - main theorem - case B_l, CL_l} yields the following expressions for the squares:
        \begin{equation*}
            \begin{split}
                z_l^2&=\left(\,\sum_{i=1}^{l-1}(2a_i^l+1)\epsilon_i,1\right),\quad\text{and}\\
                z_j^2&=\left(\sum_{i\neq j, j+1}(2a_i^j+1)\epsilon_i+(a_j^j+a_{j+1}^j+1)\epsilon_j+(a_j^j+a_{j+1}^j+1)\epsilon_{j+1},1\right),               
            \end{split}
        \end{equation*}
        \noindent
        for all $j\in[1,l-1]$. In both cases, none of the terms $2a_i^l+1$ or $2a_i^j+1$, for $i\in[1,l-2]$, can be zero, since the $a^l_j$'s and $a_i^j$'s are integers. Because $l\geq3$, for each $j\in[1,l\,]$ no element $z_{j}\in(\frac{1}{2}\sum_{i=1}^l\epsilon_i,s_{j})T$ can be an involution. Since $(t^4_j,s_j)=(\frac{1}{2}\sum_{i=1}^l\epsilon_i,s_{j})$ for all $j\in[1,l]$, it follows that $W_4\models\chi_4((t^4_1,s_1),\ldots,(t^4_l,s_l))$. Likewise, from $(t^3_l,s_l)=(0,s_l)$ and $(t^3_j,s_j)=(\frac{1}{2}\sum_{i=1}^l\epsilon_i,s_{j})$ for all $j\in[1,l-1]$, we obtain that $W_3\models\chi_3((t^3_1,s_1),\ldots,(t^3_l,s_l))$, as required. 
	\end{claimproof}
	
	\smallskip\noindent
	\underline{{\bf Case 4.2}}. $R=B_4$, and $L=CCL_4$;
    \newline In this case, by Table\cspace\ref{table - explicit crystallographic groups}, a complete system of representatives of the isomorphism classes of crystallographic groups extending $W_0(R)$ by $L$ is the following:
	\begin{enumerate}[$\bullet$]
		\item $W_1=\langle(\sum_{i=1}^3(x_i+\frac{y}{2})\epsilon_i+\frac{y}{2}\epsilon_4,1), (0,s_j):1\leq j\leq 4, x_i,y\in\Z\rangle_{\iso(V)}$;
		\item $W_2=\langle(\sum_{i=1}^3(x_i+\frac{y}{2})\epsilon_i+\frac{y}{2}\epsilon_4,1), (\frac{1}{4}\sum_{i=1}^4\epsilon_i,s_j),(0,s_4):1\leq j\leq 3,\\
		\phantom{W_2=\langle (} x_i,y\in\Z\rangle_{\iso(V)}$;
		\item $W_3=\langle(\sum_{i=1}^3(x_i+\frac{y}{2})\epsilon_i+\frac{y}{2}\epsilon_4,1), (\frac{1}{2}\epsilon_3,s_1),(\frac{1}{2}\epsilon_1,s_2),(\frac{1}{2}\epsilon_2,s_3),(0,s_4):\\
		\phantom{W_3=\langle (}x_i,y\in\Z\rangle_{\iso(V)}$;
		\item $W_4=\langle(\sum_{i=1}^3(x_i+\frac{y}{2})\epsilon_i+\frac{y}{2}\epsilon_4,1), (\frac{1}{2}\epsilon_3+\frac{1}{4}\sum_{i=1}^4\epsilon_i,s_1),(\frac{1}{2}\epsilon_1+\frac{1}{4}\sum_{i=1}^4\epsilon_i,s_2),\\\phantom{W_4=\langle\,}(\frac{1}{2}\epsilon_2+\frac{1}{4}\sum_{i=1}^4\epsilon_i,s_3),(0,s_4):x_i,y\in\Z\rangle_{\iso(V)}$.
	\end{enumerate}
	\noindent
	Since $W_1$ is split, by Lemma\cspace\ref{lemma-split extensions are definable} there exists a first-order sentence $\psi_{B_4}$ such that $W_1\models\psi_{B_4}$, while $W_n\models\neg\psi_{B_4}$ for all $n\in\{2,3,4\}$.
	
	\smallskip\noindent
	Let $\zeta_2,\zeta_3,\zeta_4$ be sentences of the form \eqref{eq.0 - lemma - all crystallographic groups arising from irreducible root systems are f.o. rigid and profinitely rigid}, with $\chi_2(x_1,\ldots,x_4),\chi_3(x_1,\ldots,x_4),\\\chi_4(x_1,\ldots,x_4)$ being the first-order formulas defined as follows:
	\begin{enumerate}[$\bullet$]
		\item $\chi_2(x_1,\ldots,x_4)$ states that there exist $z_1\in x_1T$, $z_2\in x_2T$ and $z_3\in x_3 T$ such that $z_1^2=z_2^2=z_3^2$;
		\item $\chi_3(x_1,\ldots,x_4)$ states that no elements in the cosets $x_1T,x_2T$ and $x_3T$ share the same square, and that there exists an element $z_1\in x_1 T$ commuting with an involution in $x_4T$;
        \item $\chi_4(x_1,\ldots,x_4)$ states that no elements in the cosets $x_1T,x_2T$ and $x_3T$ share the same square, and there is no $z_1\in x_1 T$ commuting with some involution in $x_4T$.
	\end{enumerate}    
	
	\begin{claim}
		For every $n,m\in\{2,3,4\}$, $W_n\models\zeta_m$ if and only if $n=m$.
	\end{claim}

    \smallskip\noindent
    \begin{claimproof}
    	As above, for each $n\in\{2,3,4\}$ we denote by $(t_1^n,s_1),\ldots,(t_4^n,s_4)$ the generators of $W_n$ in the presentation above. By construction, $\chi_2(x_1,\ldots,x_l),\chi_3(x_1,\ldots,x_l)$ and $\chi_4(x_1,\ldots,x_l)$ are mutually inconsistent. Thus, it suffices to verify that $W_n\models((t^n_1,s_1),\ldots,(t^n_4,s_4))$ for all $n\in\{2,3,4\}$.
        
        \smallskip\noindent
        First, we show that $W_2\models\chi_2((t^2_1,s_1),\ldots,(t^2_4,s_4))$. Recall that for each $j\in[1,3]$, $s_j$ acts on the canonical basis $\{\epsilon_1,\ldots,\epsilon_4\}$ as the transposition of $\epsilon_j$ and $\epsilon_{j+1}$, while $s_4$ maps $\epsilon_4$ to $-\epsilon_4$ and leaves $\epsilon_1,\epsilon_2,\epsilon_3$ fixed (cf. Table\cspace\ref{table - root systems, bases and reflections}). Since $(t^2_1,s_1),(t^2_2,s_2)$ and $(t_3^2,s_3)$ have translation component $t^2_1=t^2_2=t^2_3=\frac{1}{4}\sum_{i=1}^{4}\epsilon_i$, it follows that their squares coincide:
        \begin{equation*}
        	(t_1^2,s_1)^2=(t_2^2,s_2)^2=(t_3^2,s_3)^2=\left(\frac{1}{2}\sum_{i=1}^4\epsilon_i,1\right).
        \end{equation*}
    	\noindent
    	This proves that $W_2\models\chi_2((t^2_1,s_1),\ldots,(t^2_4,s_4))$, and hence $W_2\models\zeta_2$.
    	
    	\smallskip\noindent
    	We now show that $W_3\not\models\chi_2((t^3_1,s_1),\ldots,(t^3_4,s_4))$. By Remark\cspace\ref{remark- canonical expression of elements in cosets} and the explicit description of $CCL_4$ in Table\cspace\ref{table - martinais nomenclature}, every $z_1\in (t^3_1,s_1)T$, $z_2\in(t^3_2,s_2)T$ and $z_3\in(t^3_3,s_3)T$ is of the form:
    	\begin{equation*}
    		\begin{split}
    			z_1&=\left(\left(a^1_1+\frac{b^1}{2}\right)\epsilon_1+\left(a^1_2+\frac{b^1}{2}\right)\epsilon_2+\left(a^1_3+\frac{b^1}{2}+\frac{1}{2}\right)\epsilon_3+\frac{b^1}{2}\epsilon_4,s_1\right),\\
    			z_2&=\left(\left(a^2_1+\frac{b^2}{2}+\frac{1}{2}\right)\epsilon_1+\left(a^2_2+\frac{b^2}{2}\right)\epsilon_2+\left(a^2_3+\frac{b^2}{2}\right)\epsilon_3+\frac{b^2}{2}\epsilon_4,s_2\right),\\
    			z_3&=\left(\left(a^3_1+\frac{b^3}{2}\right)\epsilon_1+\left(a^3_2+\frac{b^3}{2}+\frac{1}{2}\right)\epsilon_2+\left(a^3_3+\frac{b^3}{2}\right)\epsilon_3+\frac{b^3}{2}\epsilon_4,s_3\right).
    		\end{split}
    	\end{equation*}
        \noindent
        for some $a^i_j,b^i\in\Z$, with $i,j\in[1,3]$. Consequently, we derive the following expression for the squares:
        \begin{equation*}
        	\begin{split}
        		z_1^2&=\left((a^1_1+a^1_2+b^1)(\epsilon_1+\epsilon_2)+(2a^1_3+b^1+1)\epsilon_3+b^1\epsilon_4,1\right),\\
        		z_2^2&=\left((2a^2_1+b^2+1)\epsilon_1+(a^2_2+a_3^2+b^2)(\epsilon_2+\epsilon_3)+b^2\epsilon_4,1\right),\\
        		z_3^2&=\left((2a^3_1+b^3)\epsilon_1+(2a^3_2+b^3+1)\epsilon_2+(a^3_3+b^3)(\epsilon_3+\epsilon_4),1\right).
        	\end{split}
        \end{equation*}
    	\noindent
    	If the identity $z_1^2=z_2^2=z_3^2$ holds, then equating the coefficients of these expressions with respect to the canonical basis yields the following compatible system of equations over $\Z$:
    	\begin{equation}\label{eq.2 - case B_4 - main_th_2}\tag{$\star_{16}$}
    			\begin{tabular}{cccccc}
    				$\epsilon_1:$&$a^1_1+a^1_2+b^1$&$=$&$2a^2_1+b^2+1$&$=$&$2a^3_1+b^3$\\ \addlinespace
    				& \verteq &&&&\\
    				$\epsilon_2:$& $a^1_1+a^1_2+b^1$ & $=$ & $a^2_2+a_3^2+b^2$ & $=$ & $2a^3_2+b^3+1.$
    			\end{tabular}
    	\end{equation}
    	\noindent
    	In particular the identity $2a^3_1+b^3=2a^3_2+b^3+1$ holds, which simplifies to $2a^3_1=2a^3_2+1$. This equation admits no integer solutions, as it equates an even and an odd value, leading to a contradiction. It follows that $W_3\not\models\chi_2((t^3_1,s_1),\ldots,(t^3_4,s_4))$. 
    	
    	\smallskip\noindent
    	Furthermore, since the nodes with labels $1$ and $4$ in the Coxeter diagram of $(W_0(B_4),\\\{s_1,\ldots,s_4\})$ are not linked (cf. Table\cspace\ref{table - Coxeter diagrams of irreducible root systems}), $s_1$ and $s_4$ commute. Hence, using that $s_4$ fixes $\epsilon_3$ (cf. Table\cspace\ref{table - root systems, bases and reflections}), one readily verifies that:
    	\begin{equation*}
    	 	\begin{tabular}{ccccc}
    	 		$(t^3_1,s_1)\cdot(t^3_4,s_4)$ & $=$ & $\left(\frac{1}{2}\epsilon_3,s_1\right)\cdot(0,s_4)$ & $=$ & $\left(\frac{1}{2}\epsilon_3,s_1s_4\right)$\\ \addlinespace
    	 		&&&& \verteq \\
    	 		$(t^3_4,s_4)\cdot(t^3_1,s_1)$ & $=$ & $(0,s_4)\cdot\left(\frac{1}{2}\epsilon_3,s_1\right)$ & $=$ & $\left(\frac{1}{2}\epsilon_3,s_4s_1\right)$.
    	 	\end{tabular}
     	\end{equation*}
     	\noindent
     	Therefore, we have $W_3\models \chi_3((t^3_1,s_1),\ldots,(t^3_4,s_4))$. It only remains to prove that $W_4\models \chi_4((t^4_1,s_1),\ldots,(t^4_4,s_4))$.
     	
     	\smallskip\noindent
     	As above, by Remark\cspace\ref{remark- canonical expression of elements in cosets} and the explicit description of $CCL_4$ in Table\cspace\ref{table - martinais nomenclature}, each $z_1\in (t^4_1,s_1)T$, $z_2\in(t^4_2,s_2)T$, $z_3\in(t^4_3,s_3)T$ and $z_4\in(t^4_4,s_4)T$ is of the form:
    	\begin{equation*}
    		\begin{split}
    			z_1&=\left(\left(a^1_1+\frac{b^1}{2}+\frac{1}{4}\right)\epsilon_1+\left(a^1_2+\frac{b^1}{2}+\frac{1}{4}\right)\epsilon_2+\left(a^1_3+\frac{b^1}{2}+\frac{3}{4}\right)\epsilon_3+\left(\frac{b^1}{2}+\frac{1}{4}\right)\epsilon_4,s_1\right),\\
    			z_2&=\left(\left(a^2_1+\frac{b^2}{2}+\frac{3}{4}\right)\epsilon_1+\left(a^2_2+\frac{b^2}{2}+\frac{1}{4}\right)\epsilon_2+\left(a^2_3+\frac{b^2}{2}+\frac{1}{4}\right)\epsilon_3+\left(\frac{b^2}{2}+\frac{1}{4}\right)\epsilon_4,s_2\right),\\
    			z_3&=\left(\left(a^3_1+\frac{b^3}{2}+\frac{1}{4}\right)\epsilon_1+\left(a^3_2+\frac{b^3}{2}+\frac{3}{4}\right)\epsilon_2+\left(a^3_3+\frac{b^3}{2}+\frac{1}{4}\right)\epsilon_3+\left(\frac{b^3}{2}+\frac{1}{4}\right)\epsilon_4,s_3\right),\\
    			z_4&=\left(\sum_{i=1}^3\left(a^4_i+\frac{b^4}{2}\right)\epsilon_i+\frac{b^4}{2}\epsilon_4,s_4\right).
    		\end{split}
    	\end{equation*}
    	\noindent
    	for some $a^i_j,b^i\in\Z$. The corresponding squares
    	have expressions:
    	\begin{equation*}
    		\begin{split}
    			z_1^2&=\left(\left(a^1_1+a_2^1+b^1+\frac{1}{2}\right)(\epsilon_1+\epsilon_2)+\left(2a^1_3+b^1+\frac{3}{2}\right)\epsilon_3+\left(b^1+\frac{1}{2}\right)\epsilon_4,1\right),\\
    			z_2^2&=\left(\left(2a^2_1+b^2+\frac{3}{2}\right)\epsilon_1+\left(a^2_2+a^2_3+b^2+\frac{1}{2}\right)(\epsilon_2+\epsilon_3)+\left(b^2+\frac{1}{2}\right)\epsilon_4,1\right),\\
    			z_3^2&=\left(\left(2a^3_1+b^3+\frac{1}{2}\right)\epsilon_1+\left(2a^3_2+b^3+\frac{3}{2}\right)\epsilon_2+\left(a^3_3+b^3+\frac{1}{2}\right)(\epsilon_3+\epsilon_4),1\right).
    		\end{split}
    	\end{equation*}
    	\noindent
    	Assume that the identity $z_1^2=z_2^2=z_3^2$ holds for some $z_1,z_2,z_3$. Then, equating the coefficients of the expressions above with respect to the canonical basis returns the same system as \eqref{eq.2 - case B_4 - main_th_2}, thus yielding a contradiction. It follows that $W_4\not\models\chi_2((t^4_1,s_1),\ldots,(t^4_4,s_4))$.
    	
    	\smallskip\noindent
    	Finally, we prove that no $z_1$ and $z_4$ as above commute. This clearly implies the weaker statement that no element in the coset $(t^4_1,s_1)T$ commutes with an involution in $(t^4_4,s_4)T$. Recall that $s_1$ acts on the basis $\{\epsilon_1,\ldots,\epsilon_4\}$ by transposing $\epsilon_1$ and $\epsilon_2$, while $s_4$ fixes $\epsilon_1,\epsilon_2,\epsilon_3$ and maps $\epsilon_4$ to $-\epsilon_4$. Consequently, one derives the following identities:
    	\begin{equation*}
    		\begin{split}
    			z_1\cdot z_4 & =\bigg(\left(a^1_1+a_2^4+\frac{b^1+b^4}{2}+\frac{1}{4}\right)\epsilon_1+\left(a^1_2+a^4_1+\frac{b^1+b^4}{2}+\frac{1}{4}\right)\epsilon_2\\
    			&+\left(a^1_3+a_3^4+\frac{b^1+b^4}{2}+\frac{3}{4}\right)\epsilon_3+\left(\frac{b^1+b^4}{2}+\frac{1}{4}\right)\epsilon_4,s_1s_4\bigg),\\
    			z_4\cdot z_1 & = \bigg(\left(a^1_1+a_1^4+\frac{b^1+b^4}{2}+\frac{1}{4}\right)\epsilon_1+\left(a^1_2+a_2^4+\frac{b^1+b^4}{2}+\frac{1}{4}\right)\epsilon_2\\
    			&+\left(a^1_3+a_3^4+\frac{b^1+b^4}{2}+\frac{3}{4}\right)\epsilon_3+\left(\frac{b^4}{2}-\frac{b^1}{2}-\frac{1}{4}\right)\epsilon_4,s_4s_1\bigg).
    		\end{split}
    	\end{equation*}
    	\noindent
    	If $z_1\cdot z_4=z_4\cdot z_1$ were to hold, then equating the coefficients of $z_1\cdot z_4$ and $z_4\cdot z_1$ with respect to the element $\epsilon_4$ of the canonical basis would determine the compatible integral equation:
    	\begin{equation*}
    		\frac{b^1+b^4}{2}+\frac{1}{4}=\frac{b^4}{2}-\frac{b^1}{2}-\frac{1}{4}
    		\quad\Rightarrow\quad
    		b^1=-\frac{1}{2}.
    	\end{equation*}
    	\noindent
    	Since $b^1$ is supposed to be an integer, this leads to a contradiction. It follows that $W_4\models \chi_4((t^4_1,s_1),\ldots,(t^4_4,s_4))$, completing the proof.
    	\end{claimproof}
    	
    	\smallskip\noindent
    	\underline{{\bf Case 4.3}}. $R=C_l$, with $l\geq3$ odd, and $L=FL_l$;
    	\newline In this case, by Table\cspace\ref{table - explicit crystallographic groups} a complete system of representatives of the isomorphism classes of crystallographic groups extending $W_0(R)$ by $L$ is the following:
    	\begin{enumerate}[$\bullet$] 
    		\item $W_1=\langle(\sum_{i=1}^lx_i\epsilon_i,1),(0,s_j):1\leq j\leq l,x_i\in \Z, \sum_{i=1}^lx_i\in2\Z \rangle_{\iso(V)}$;
    		\item $W_2=\langle(\sum_{i=1}^lx_i\epsilon_i,1),(0,s_j),(\frac{1}{2}\sum_{i=1}^l\epsilon_i,s_l):1\leq j< l,x_i\in \Z,\\
    		\phantom{W_e=\langle\,}\sum_{i=1}^lx_i\in2\Z \rangle_{\iso(V)}$;
    		\item $W_3=\langle(\sum_{i=1}^lx_i\epsilon_i,1),(\epsilon_1,s_j):1\leq j\leq l,x_i\in \Z, \sum_{i=1}^lx_i\in2\Z \rangle_{\iso(V)}$;
    		\item $W_4=\langle(\sum_{i=1}^lx_i\epsilon_i,1),(\epsilon_1,s_j),(\epsilon_1+\frac{1}{2}\sum_{i=1}^l\epsilon_i,s_l):1\leq j< l,x_i\in \Z,\\
    		\phantom{W_4=\langle\,} \sum_{i=1}^lx_i\in2\Z \rangle_{\iso(V)}$.
    	\end{enumerate}
    	\noindent
    	Since $W_1$ is split, by Lemma\cspace\ref{lemma-split extensions are definable} there exists a first-order sentence $\psi_{C_l}$ such that $W_1\models\psi_{C_l}$, while $W_n\models\neg\psi_{C_l}$ for all $n\in\{2,3,4\}$.
    	
    	\smallskip\noindent
    	Let $\zeta_2,\zeta_3,\zeta_4$ be sentences of the form \eqref{eq.0 - lemma - all crystallographic groups arising from irreducible root systems are f.o. rigid and profinitely rigid}, with $\chi_2(x_1,\ldots,x_4),\chi_3(x_1,\ldots,x_4),\\\chi_4(x_1,\ldots,x_4)$ being the first-order formulas defined as follows:  
    	\begin{enumerate}[$\bullet$]
    		\item $\chi_2(x_1,\ldots,x_l)$ states that there exists an involution in $x_iT$ for all $i\in[1,l-1]$;
    		\item $\chi_3(x_1,\ldots,x_l)$ states that there are no involutions in $x_2T$, and there exist $z_2\in x_2T$ and $z_l\in x_l T$ such that $z_2^2=z_l^2$;
    		\item $\chi_4(x_1,\ldots,x_l)$ states that there are no involutions in $x_2T$,
    		and no element in $x_2T$ has the same square of an element in $x_lT$.
    	\end{enumerate}    
    	\smallskip\noindent
    	\begin{claim}
    		For every $n,m\in\{2,3,4\}$, $W_n\models\zeta_m$ if and only if $n=m$.
    	\end{claim}
    	
    	\smallskip\noindent
    	\begin{claimproof}
    		For each $n\in\{2,3,4\}$, denote by $(t_1^n,s_1),\ldots,(t_l^n,s_l)$ the generators of $W_n$ in the presentation above. As in the preceding cases, $\chi_2(x_1,\ldots,x_l),\chi_3(x_1,\ldots,x_l)$ and $\chi_4(x_1,\ldots,x_l)$ are mutually inconsistent, thus it suffices to establish that $W_n\models\chi_n((t^n_1,s_1),\ldots,(t^n_l,s_l))$ for all $n\in\{2,3,4\}$.
    		
    		\smallskip\noindent
    		Clearly, $W_2\models\chi_2((t^2_1,s_1),\ldots,(t^2_l,s_l))$, since for each $j\in[1,l-1]$ we have $(t^2_j,s_j)=(0,s_j)$. We prove that $W_3\models\chi_3((t^3_1,s_1),\ldots,(t^3_l,s_l))$.
    		
    		\smallskip\noindent
    		The second claim of the formula follows directly from the fact that $s_2$ acts on $\{\epsilon_1,\ldots,\epsilon_l\}$ as the transposition of $\epsilon_2$ and $\epsilon_3$, while $s_l$ maps $\epsilon_l$ to $-\epsilon_l$ and leaves the other elements of the canonical basis fixed (cf.\cspace Table\cspace\ref{table - root systems, bases and reflections}). Therefore, since $(t^3_2,s_2)=(\epsilon_1,s_2)$ and $(t^3_l,s_l)=(\epsilon_1,s_l)$, we have:
    		\begin{equation*}
    			(t^3_2,s_2)^2=(\epsilon_1,s_2)^2=(2\epsilon_1,1)=(\epsilon_1,s_l)^2=(t^3_l,s_l)^2.
       		\end{equation*}
    		\noindent
    		We now prove that there are no involutions in the coset $(t^3_2,s_2)T$. By Remark\cspace\ref{remark_for_star2} and the explicit description of $FL_l$ in Table\cspace\ref{table all inequivalent lattices}, each $z_2\in (t^3_2,s_2)T$ has the form:
    		\begin{equation}\label{eq.2 - case C_l main thm}\tag{$\star_{17}$}
    				z_2=\left((a^2_1+1)\epsilon_1+\sum_{j=2}^la_j^2\epsilon_j,s_2\right),
    		\end{equation}
    		\noindent
    		for some $a_1^2,\ldots,a_l^2\in\Z$ such that $\sum_{i=1}^la_i^2\in2\Z$.
    		Consequently, using again that $s_2$ acts on $\{\epsilon_1,\ldots,\epsilon_l\}$ by transposing $\epsilon_2$ and $\epsilon_3$ (cf. Table\cspace\ref{table - root systems, bases and reflections}), one readily computes the following expression for the square:
    		\begin{equation}\label{eq.3 - case C_l main thm}\tag{$\star_{18}$}
    			\begin{split}
    				z_2^2&=\left((2a^2_1+2)\epsilon_1+(a^2_2+a^2_3)\epsilon_2+(a^2_2+a^2_3)\epsilon_3+\sum_{i\neq1,2,3}2a_i^2\epsilon_i,1\right).
    			\end{split}
    		\end{equation}
    		\noindent
    		If $z_2$ were an involution, the identity $z_2^2=(0,1)$ would yield the following compatible system of equations over $\Z$:
    		\begin{equation*}
    			\begin{cases}
    				&a^2_1+1=0\\
    				&a^2_2+a^2_3=0\\
    				&a^2_i=0\quad\text{for all }i\neq1,2,3.
    			\end{cases}
    		\end{equation*}
    		\noindent
    		This leads to a contradiction, since each tuple $a^2_1,\ldots,a^2_l\in\Z$ realizing the system above satisfies the identity:
    		\begin{equation*}
    			\sum_{i=1}^la^2_i=a^2_1+(a^2_2+a^2_3)+\sum_{i\neq1,2,3}a^2_i=-1+0+0=-1\notin2\Z,
    		\end{equation*}
    		\noindent
    		against the assumption $\sum_{i=1}a^2_i\in2\Z$. This proves that there are no involutions in the coset $(t^3_2,s_2)T$, and thus that $W_3\not\models\chi_2((t^3_1,s_1),\ldots,(t^3_l,s_l))$.
    		
    		\smallskip\noindent
    		It only remains to prove that $W_4\models\chi_4((t^4_1s_1),\ldots,(t^4_l,s_l))$. In this case, we have $(t^4_j,s_j)=(t^3,s_j)$ for all $j\in[1,l-1]$. Thus, each $z_2\in(t^4_2,s_2)T$ is of the form \eqref{eq.2 - case C_l main thm} and the preceding argument shows that this is not an involution.
    		
    		\smallskip\noindent
    		By Remark\cspace\ref{remark- canonical expression of elements in cosets} and the explicit description of $FL_l$ in Table\cspace\ref{table - martinais nomenclature}, each $z_l\in(t^4_l,s_l)T$ is admits expression:    		
    		\begin{equation*}
    			z_l=\left(\left(a^4_1+\frac{3}{2}\right)\epsilon_1+\sum_{i=2}^l\left(a^4_i+\frac{1}{2}\right)\epsilon_i,s_l\right),
    		\end{equation*}
    		\noindent
    		for some $a^4_1,\ldots,a^4_l\in\Z$ such that $\sum_{i=1}^la^4_i\in2\Z$. Hence, it has square:
    		\begin{equation}\label{eq.4 - case C_l main thm}\tag{$\star_{19}$}
    			z_l^2=\left((2a^4_1+3)\epsilon_1+\sum_{i=2}^{l-1}(2a^4_i+1)\epsilon_i,1\right).
    		\end{equation}
    		\noindent
    		If the identity $z_2^2=z_l^2$ were to hold, then comparing the terms of \eqref{eq.3 - case C_l main thm} and \eqref{eq.4 - case C_l main thm} with respect to the canonical basis would yield the following system of equations:
    		\begin{equation*}
    			\begin{cases}
    				&2a^2_1+2=2a^4_1+3\\
    				&a^2_2+a^2_3=2a^4_2+1\\
    				&a^2_2+a^2_3=a^4_3+1\\
    				&2a_i^2=2a^4_i+1\quad\text{for all}\quad i\neq1,2,3.			
    			\end{cases}
    		\end{equation*}
    		\noindent
    		However, this system is clearly inconsistent, since it contains the identity $2a_i^2=2a^4_i+1$ of an even and an odd number. It follows that no $z_2\in(t^4_2,s_2)T$ and $z_l\in(t^4_l,s_l)T$ have the same square, and hence that $W_4\models\chi_4((t^4_1,s_1),\ldots,(t^4_l,s_l))$.
    		
    	\end{claimproof}
    
    	\noindent
    	Since in each case we have found a first-order sentence separating the representatives from Martinais's classification, this concludes the proof.
\end{proof}
\begin{theorem1.2}
	\emph{Finite direct products of crystallographic groups arising from an irreducible root system are profinitely rigid (equiv. first-order rigid).}
\end{theorem1.2}
\begin{proof}
	This is immediate by Lemma\cspace\ref{lemma_direct_products} and Theorem\cspace\ref{lemma - all crystallographic groups arising from irreducible root systems are f.o. rigid and profinitely rigid}.
\end{proof}


\newpage \section{Appendix}\label{app_sec}
\begin{table}[h!]
	\caption{Coxeter diagrams of all irreducible root systems $R$ of rank $l\geq3$. Each node corresponds to a simple reflection $s_i$ associated to the simple root $\alpha_i$ in the standard Bourbaki numbering (cf. \cite[Planche{\cspace}I to IX]{bourbaki}). The index $i$ labels the node corresponding to $s_i$, for all $i\in[1,l\,]$.}
	\begin{tabular}{ll}
		\toprule
		Type of $R$ & Coxeter diagram of $(W_0(R),S)$ \\
		\bottomrule
		
		&\\
		
		$A_l$, $l\geq3$ &
		\begin{tikzpicture}[scale=0.8, baseline={(current bounding box.center)}]
			
			\foreach \i in {1,...,3} {
				\draw (\i*1.2,0) -- ({(\i+1)*1.2},0);
			}
			\foreach \i in {1,...,3} {
				\nodecircle{\i}{(\i*1.2,0)}
			}
			\node at (5.4,0) {$\cdots$};
			\draw (6,0) -- (7,0);
			\nodecircle{$l$}{(7,0)};
		\end{tikzpicture} \\[1em]
		
		$B_l$, $l\geq3$ &
		\begin{tikzpicture}[scale=0.8, baseline={(current bounding box.center)}]
			
			\foreach \i in {1,...,3} {
				\draw (\i*1.2,0) -- ({(\i+1)*1.2},0);
			}
			\foreach \i in {1,...,3} {
				\nodecircle{\i}{(\i*1.2,0)}
			}
			\node at (5.4,0) {$\cdots$};
			\draw (6,0) -- (7,0);
			\draw (7,0) -- (8,0) node[midway, above] {\small 4};
			\nodecircle{$l-1$}{(7,0)};
			\nodecircle{$l$}{(8,0)};
			
		\end{tikzpicture} \\[1em]
		
		$C_l$, $l\geq3$ &
		\begin{tikzpicture}[scale=0.8, baseline={(current bounding box.center)}]
			\foreach \i in {1,...,3} {
				\draw (\i*1.2,0) -- ({(\i+1)*1.2},0);
			}
			\foreach \i in {1,...,3} {
				\nodecircle{\i}{(\i*1.2,0)}
			}
			\node at (5.4,0) {$\cdots$};
			\draw (6,0) -- (7,0);
			\draw (7,0) -- (8,0) node[midway, above] {\small 4};
			\nodecircle{$l-1$}{(7,0)};
			\nodecircle{$l$}{(8,0)};
		\end{tikzpicture} \\[1em]
		
		$D_l$, $l\geq3$ &
		\begin{tikzpicture}[scale=0.8, baseline={(current bounding box.center)}]
			\node at (4.2,0) {$\cdots$};
			\foreach \i in {1,...,2} { \draw (\i*1.2,0) -- ({(\i+1)*1.2},0); }
			\foreach \i in {1,...,2} { \nodecircle{\i}{(\i*1.2,0)} }
			\draw (4.8,0) -- (5.8,0);
			\draw (5.8,0) -- (7,0.6);
			\draw (5.8,0) -- (7,-0.6);
			\nodecircle{$l-2$}{(5.8,0)}
			\nodercircle{$l-1$}{(7,0.6)};
			\nodercircle{$l$}{(7,-0.6)};
		\end{tikzpicture} \\[1em]
		
		$E_6$ &
		\begin{tikzpicture}[scale=0.8, baseline={(current bounding box.center)}]
			\foreach \i in {1,...,4} { \draw (\i*1.2,0) -- ({(\i+1)*1.2},0); }
			\draw (3.6,0) -- (3.6,1.2);
			\nodecircle{1}{(1.2,0)};
			\nodecircle{3}{(2.4,0)};
			\nodecircle{4}{(3.6,0)};
			\nodecircle{5}{(4.8,0)};
			\nodecircle{6}{(6,0)};
			\nodercircle{2}{(3.6,1.2)};
		\end{tikzpicture} \\[1em]
		
		$E_7$ &
		\begin{tikzpicture}[scale=0.8, baseline={(current bounding box.center)}]
			\foreach \i in {1,...,5} { \draw (\i*1.2,0) -- ({(\i+1)*1.2},0); }
			\draw (3.6,0) -- (3.6,1.2);
			\nodecircle{1}{(1.2,0)};
			\nodecircle{3}{(2.4,0)};
			\nodecircle{4}{(3.6,0)};
			\nodecircle{5}{(4.8,0)};
			\nodecircle{6}{(6,0)};
			\nodecircle{7}{(7.2,0)};
			\nodercircle{2}{(3.6,1.2)};
		\end{tikzpicture} \\[1em]
		
		$E_8$ &
		\begin{tikzpicture}[scale=0.8, baseline={(current bounding box.center)}]
			\foreach \i in {1,...,6} { \draw (\i*1.2,0) -- ({(\i+1)*1.2},0); }
			\draw (3.6,0) -- (3.6,1.2);
			\nodecircle{1}{(1.2,0)};
			\nodecircle{3}{(2.4,0)};
			\nodecircle{4}{(3.6,0)};
			\nodecircle{5}{(4.8,0)};
			\nodecircle{6}{(6,0)};
			\nodecircle{7}{(7.2,0)};
			\nodecircle{8}{(8.4,0)};
			\nodercircle{2}{(3.6,1.2)};
		\end{tikzpicture} \\[1em]
		
		$F_4$ &
		\begin{tikzpicture}[scale=0.8, baseline={(current bounding box.center)}]
			\draw (0,0) -- (1.2,0);
			\draw (1.2,0) -- (2.4,0) node[midway, above] {\small 4};
			\draw (2.4,0) -- (3.6,0);
			\nodecircle{1}{(0,0)};
			\nodecircle{2}{(1.2,0)};
			\nodecircle{3}{(2.4,0)};
			\nodecircle{4}{(3.6,0)};
		\end{tikzpicture} \\[1em]
		
		$G_2$ &
		\begin{tikzpicture}[scale=0.8, baseline={(current bounding box.center)}]
			\draw (0,0) -- (1.2,0) node[midway, above] {\small 6};
			\nodecircle{1}{(0,0)};
			\nodecircle{2}{(1.2,0)};
		\end{tikzpicture} \\
		
		&\\
		
		\bottomrule
	\end{tabular}
	\label{table - Coxeter diagrams of irreducible root systems}
\end{table}
\newpage 
\begin{table}[h!]
	\caption{Representatives $L$ of the isomorphism classes of lattices in a real vector space $V$ invariant under the Weyl group $W_0(R)$ of an irreducible root system $R$ of rank $l\geq3$ in $V$. Each lattice $L$ is described both implicitly, in terms of $Q(R)$ and $P(R)$ as in \cite{maxwell}, and explicitly, following Martinais's notation from \cite{martinais} (see also Table\cspace\ref{table - martinais nomenclature}). The elements $\overline{\omega}_1$ and $\overline{\omega}_l$ denote the fundamental weights associated to the simple roots $\alpha_1$ and $\alpha_l$ in Bourbaki's standard realization \cite[Planche I to IX]{bourbaki}.}
	\makebox[\linewidth]{
		\begin{tabular}{lllc}
			\toprule
			Type of $R$ & Inequivalent lattices $L(R)$ & Martinais's & $n(W_0(R),L(R))$ \\
			& & notation & \\ \bottomrule
			
			&&&\\
			
			$A_l$, $l \geq 4$ & $L_k(A_l)=Q(A_l)+\langle a_k\,\overline{\omega}_1\rangle_\Z$ & $\Lambda_{l,a_k}$ & $1$ if $a_k$ is odd \\
			&{\small with $a_k\in\N^+$ the $k$th-divisor of $l+1$}&&$2$ if $a_k$ is even\\
			
			&&&\\
			
			$B_l$, $l\geq3$ & $L_1(B_l)=Q(B_l)$ & $CL_l$ & $4$ \\
			
			&&&\\
			
			& $L_2(B_l)=P(B_l)$ & $CCL_l$ & $2$ if $l=3$ \\
			&&& $4$ if $l=4$ \\
			&&& $1$ if $l\geq 5$ \\
			
			&&&\\
			
			$C_l$, $l\geq3$ & $L_1(C_l)=Q(C_l)$ & $FL_l$ & $4$ if $l$ odd\\
			&&& $2$ if $l$ is even \\
			
			&&&\\
			
			$D_l$, $l \geq 3$ odd, & $L_1(D_l)=Q(D_l)$ & $FL_l$ & $2$ \\
			$\phantom{D_l,}$ or $l=4$ & $L_2(D_l)=Q(D_l)+\langle \overline{\omega}_1\rangle_\Z$ & $CL_l$ & $2$ \\
			& $L_3(D_l)=P(D_l)$ & $CCL_l$ & $2$ if $l=3$ or $l=4$\\
			&&& $1$ if $l\geq5$ \\
			
			&&&\\
			
			$D_l$, $l\geq6$ even & $L_1(D_l)=Q(D_l)$ & $FL_l$ & $3$ \\
			& $L_2(D_l)=Q(D_l)+\langle\overline{\omega}_1\rangle_\Z$ & $CL_l$ & $2$ \\
			& $L_3(D_l)=Q(D_l)+\langle\overline{\omega}_l\rangle_\Z$ & $\Omega_l$ & $2$ \\			
			& $L_4(D_l)=P(D_l)$ & $CCL_l$ & $1$ \\
			
			&&&\\
			
			$E_6$ & $L_1(E_6)=Q(E_6)$ & $Q_6$ & $1$ \\
			& $L_2(E_6)=P(E_6)$ & $P_6$ & $1$ \\
			
			&&&\\
			
			$E_7$ & $L_1(E_7)=Q(E_7)$ & $Q_7$ & $2$ \\
			& $L_2(E_7)=P(E_7)$ & $P_7$ & $1$ \\
			
			&&&\\
			
			$E_8$ & $L_1(E_8)=Q(E_8)$ & $\Omega_8$ & $1$ \\
			
			&&&\\
			
			$F_4$ & $L_1(F_4)=Q(F_4)$ & $CCL_4$ & $1$ \\
			
			&&&\\
			\bottomrule
		\end{tabular}
	}
	\label{table all inequivalent lattices}
\end{table}
\newpage
\begin{table}[h!]
	\caption{Families of lattices associated with irreducible root systems $R$, described according to Bourbaki's standard notation \cite[Planche I to VIII]{bourbaki}. In particular, $l$ represents the rank of $R$, and each $\epsilon_i$ denotes the $i$th-element of the canonical basis of the real vector space underlying $R$.}
	\begin{tabular}{ll}
		\toprule
		&\\
		
		& $\Lambda_{l,a_k}=\bigoplus_{i=1}^{l-1}\langle \epsilon_i-\epsilon_{i+1}\rangle_\Z\oplus\langle a_k\epsilon_1-\frac{a_k}{l+1}\sum_{i=1}^{l+1}\epsilon_i\rangle_\Z,$ \\
		& \phantom{$\Lambda_{l,a_k}=$}with $a_k\in\N^+$ being the $k$th-divisor of $l+1$ \\
		
		&\\
		
		& $CL_l=\bigoplus_{i=1}^{l}\langle\epsilon_i\rangle_\Z$ \\
		
		&\\
		
		& $CCL_l=\bigoplus_{i=1}^{l-1}\langle\epsilon_i\rangle_\Z\oplus\langle\frac{1}{2}\sum_{i=1}^l\epsilon_i\rangle_\Z$ \\
		
		&\\
		
		& $FL_l=\{\sum_{i=1}^lx_i\epsilon_i\,:\,x_i\in\Z\text{ and }\sum_{i=1}^lx_i\text{ even}\}$ \\
		
		&\\
		
		& $\Omega_l=\{\sum_{i=1}^lx_i\epsilon_i\,:\,x_i\in\Z\text{ and }\sum_{i=1}^lx_i\text{ even}\}+\langle \frac{1}{2}\sum_{i=1}^{l}\epsilon_i\rangle_\Z$ \\
		
		&\\
		
		& $Q_6=\bigoplus_{i=1}^5\langle\epsilon_1+\epsilon_i\rangle_\Z\oplus\langle \frac{1}{2}(\epsilon_1+\epsilon_8-\sum_{i=2}^7\epsilon_i)\rangle_\Z$ \\
		
		&\\
		
		& $P_6=\bigoplus_{i=1}^4\langle\epsilon_1+\epsilon_i\rangle_\Z\oplus\langle\frac{1}{2}(\epsilon_1+\epsilon_8-\sum_{i=2}^7\epsilon_i)\rangle_\Z\oplus\langle \epsilon_1+\epsilon_5+\frac{2}{3}(\epsilon_6+\epsilon_7-\epsilon_8)\rangle_\Z$ \\
		
		&\\
		
		& $Q_7=\bigoplus_{i=1}^6\langle \epsilon_1+\epsilon_i\rangle_\Z \oplus\langle \frac{1}{2}(\epsilon_1+\epsilon_8-\sum_{i=2}^7\epsilon_i)\rangle_\Z$ \\
		
		&\\
		
		& $P_7=\bigoplus_{i=1}^5\langle \epsilon_1+\epsilon_i\rangle_\Z\oplus\langle \frac{1}{2}(\epsilon_1+\epsilon_8-\sum_{i=2}^7\epsilon_i)\rangle_\Z\oplus \langle \frac{1}{2}\sum_{i=1}^6\epsilon_i\rangle_\Z$ \\
		
		&\\
		
		\bottomrule
	\end{tabular}
	\label{table - martinais nomenclature}
\end{table}

\newpage
\begin{table}[h!]
	\caption{Representatives of the isomorphism classes of crystallographic groups arising from irreducible root systems from \cite[Table{\cspace}V]{martinais}. The representatives are described according to Bourbaki's standard notation \cite[Planche I to VIII]{bourbaki}. In particular, $\epsilon_i$ denotes the $i$th-element of the canonical basis of the ambient real vector space and $l$ the rank of~$R$.}
	\makebox[\linewidth]{
		\renewcommand{\arraystretch}{1.5}
		\begin{tabular}{ll}
			\toprule
			Type of $R$ & Crystallographic groups with point group $W_0(R)$ and translation lattice $L$ \\
			Lattice $L$ & \\
			\bottomrule
			&\\
			
			$B_l$, $l\geq3$ & $W_1=\langle(\sum_{i=1}^lx_i\epsilon_i,1),(0,s_j):1\leq j\leq l, x_i\in \Z \rangle_{\iso(V)}$ \\
			$CL_l$ & $\phantom{W_1}=CL_l\rtimes W_0(R)$ \\
			& $W_2=\langle(\sum_{i=1}^lx_i\epsilon_i,1),(0,s_j),(\frac{1}{2}\sum_{i=1}^l\epsilon_i,s_l):1\leq j< l, x_i\in \Z \rangle_{\iso(V)}$ \\
			& $W_3=\langle(\sum_{i=1}^lx_i\epsilon_i,1),(\frac{1}{2}\sum_{i=1}^l\epsilon_i,s_j),(0,s_l):1\leq j< l, x_i\in \Z \rangle_{\iso(V)}$ \\
			& $W_4=\langle(\sum_{i=1}^lx_i\epsilon_i,1),(\frac{1}{2}\sum_{i=1}^l\epsilon_i,s_j):1\leq j\leq l, x_i\in \Z \rangle_{\iso(V)}$ \\
			
			&\\
			
			$B_4$ & $W_1=\langle(\sum_{i=1}^4(x_i+\frac{y}{2})\epsilon_i,1), (0,s_j):1\leq j\leq 4, x_i,y\in\Z\rangle_{\iso(V)}$ \\
			$CCL_4$ & $\phantom{W_1}=CCL_4\rtimes W_0(B_4)$ \\
			& $W_2=\langle(\sum_{i=1}^4(x_i+\frac{y}{2})\epsilon_i,1), (\frac{1}{4}\sum_{i=1}^4\epsilon_i,s_j),(0,s_4):1\leq j\leq 3, x_i,y\in\Z\rangle_{\iso(V)}$ \\
			& $W_3=\langle(\sum_{i=1}^4(x_i+\frac{y}{2})\epsilon_i,1), (\frac{1}{2}\epsilon_3,s_1),(\frac{1}{2}\epsilon_1,s_2),(\frac{1}{2}\epsilon_2,s_3),(0,s_4):x_i,y\in\Z\rangle_{\iso(V)}$ \\
			& $W_4=\langle(\sum_{i=1}^4(x_i+\frac{y}{2})\epsilon_i,1), (\frac{1}{2}\epsilon_3+\frac{1}{4}\sum_{i=1}^4\epsilon_i,s_1),$\\
			& $\phantom{W_4=\langle}(\frac{1}{2}\epsilon_1+\frac{1}{4}\sum_{i=1}^4\epsilon_i,s_2),(\frac{1}{2}\epsilon_2+\frac{1}{4}\sum_{i=1}^4\epsilon_i,s_3),(0,s_4):x_i,y\in\Z\rangle_{\iso(V)}$ \\
			
			&\\
			
			$C_l$, $l\geq 3$ odd & $W_1=\langle(\sum_{i=1}^lx_i\epsilon_i,1),(0,s_j):1\leq j\leq l,x_i\in \Z, \sum_{i=1}^lx_i\in2\Z \rangle_{\iso(V)}$ \\
			$FL_l$ & $\phantom{W_1}=FL_L\rtimes W_0(R)$ \\
			& $W_2=\langle(\sum_{i=1}^lx_i\epsilon_i,1),(0,s_j),(\frac{1}{2}\sum_{i=1}^l\epsilon_i,s_l):1\leq j< l,x_i\in \Z, \sum_{i=1}^lx_i\in2\Z \rangle_{\iso(V)}$ \\
			& $W_3=\langle(\sum_{i=1}^lx_i\epsilon_i,1),(\epsilon_1,s_j):1\leq j\leq l,x_i\in \Z, \sum_{i=1}^lx_i\in2\Z \rangle_{\iso(V)}$ \\
			& $W_4=\langle(\sum_{i=1}^lx_i\epsilon_i,1),(\epsilon_1,s_j),(\epsilon_1+\frac{1}{2}\sum_{i=1}^l\epsilon_i):1\leq j< l,x_i\in \Z, \sum_{i=1}^lx_i\in2\Z \rangle_{\iso(V)}$ \\
			
			&\\
			
			$D_l$, $l\geq 6$ even & $W_1=\langle(\sum_{i=1}^lx_i\epsilon_i,1),(0,s_j):1\leq j\leq l, x_i\in\Z, \sum_{i=1}^l x_i\in 2\Z\rangle_{\iso(V)}$ \\
			$FL_l$ & $\phantom{W_1}=FL_l\rtimes W_0(D_l)$ \\
			& $W_2=\langle(\sum_{i=1}^lx_i\epsilon_i,1),(\epsilon_1,s_j):1\leq j\leq l, x_i\in\Z, \sum_{i=1}^l x_i\in 2\Z\rangle_{\iso(V)}$ \\
			& $W_3=\langle(\sum_{i=1}^lx_i\epsilon_i,1),(\frac{1}{2}\sum_{i=1}^l \epsilon_i,s_j):1\leq j\leq l, x_i\in\Z, \sum_{i=1}^l x_i\in 2\Z\rangle_{\iso(V)}$ \\
			
			&\\
			\bottomrule
		\end{tabular}
		\renewcommand{\arraystretch}{1}
	}
	\label{table - explicit crystallographic groups}
\end{table}

\newpage
\begin{table}[h!]
	\caption{Irreducible root systems $R$ of type $B_l,C_l$ and $D_l$, for $l\geq3$, together with selected bases and their simple reflections, given in Bourbaki's standard realization \cite[Planche I to IX]{bourbaki}. Each $\epsilon_i$ denotes the $i$th-element of the canonical basis of the real Euclidean vector space underlying $R$, and $s_i$ denotes the simple reflection induced by the root $\alpha_i$.}
	\label{table - root systems, bases and reflections}
	\renewcommand{\arraystretch}{1.5}
	\begin{tabular}{ll}
		\toprule
		
		& Root system type: $B_l$, $l\geq3$\\
		& Roots: $\pm\epsilon_i$, $\pm(\epsilon_i+\epsilon_j)$, $\pm(\epsilon_i-\epsilon_j)$, {\small with} $1\leq i<j\leq l$ \\
		& Basis: $\alpha_j=\epsilon_j-\epsilon_{j+1}$, $\alpha_l=\epsilon_l$, {\small with} $1\leq j\leq l-1$\\
		& Simple reflections:\\
		&$s_j(\epsilon_i)=\begin{cases}
			\epsilon_{j+1}&\text{if }i=j\\
			\epsilon_{j}&\text{if }i=j+1,\\
			\epsilon_i&\text{if }i \neq j, j+1
		\end{cases}$ for all $1\leq j\leq l-1$\\ \addlinespace
		& $s_l(\epsilon_i)=\begin{cases}
			\epsilon_i&\text{if }i\neq l\\
			-\epsilon_l&\text{if }i=l
		\end{cases}$\\ 
		[3ex]
		
		\midrule
		
		& Root system type: $C_l$, $l\geq3$\\
		& Roots: $\pm2\epsilon_i$, $\pm(\epsilon_i+\epsilon_j)$, $\pm(\epsilon_i-\epsilon_j)$, {\small with} $1\leq i<j\leq l$ \\
		& Basis: $\alpha_j=\epsilon_j-\epsilon_{j+1}$, $\alpha_l=2\epsilon_l$, {\small with} $1\leq j\leq l-1$\\
		& Simple reflections:\\
		&$s_j(\epsilon_i)=\begin{cases}
			\epsilon_{j+1}&\text{if }i=j\\
			\epsilon_{j}&\text{if }i=j+1,\\
			\epsilon_i&\text{if }i \neq j, j+1
		\end{cases}$ for all $1\leq j\leq l-1$\\ \addlinespace
		& $s_l(\epsilon_i)=\begin{cases}
			\epsilon_i&\text{if }i\neq l\\
			-\epsilon_l&\text{if }i=l
		\end{cases}$\\[3ex]
		
		\toprule
		
		& Root system type: $D_l$, $l\geq3$\\
		& Roots: $\pm(\epsilon_i+\epsilon_j)$, $\pm(\epsilon_i-\epsilon_j)$, {\small with} $1\leq i<j\leq l$ \\
		& Basis: $\alpha_j=\epsilon_j-\epsilon_{j+1}$, $\alpha_l=\epsilon_{l-1}+\epsilon_l$, {\small with} $1\leq j\leq l-1$\\
		& Simple reflections:\\
		&$s_j(\epsilon_i)=\begin{cases}
			\epsilon_{j+1}&\text{if }i=j\\
			\epsilon_{j}&\text{if }i=j+1,\\
			\epsilon_i&\text{if }i \neq j, j+1
		\end{cases}$ for all $1\leq j\leq l-1$\\ \addlinespace
		
		& $s_l(\epsilon_i)=\begin{cases}
			\epsilon_i&\text{if }i\neq l-1,l\\
			-\epsilon_l&\text{if }i=l-1\\
			-\epsilon_{l-1}&\text{if }i=l
		\end{cases}$\\ \addlinespace
		
		\bottomrule
	\end{tabular}
	\renewcommand{\arraystretch}{1}
\end{table}

\newpage

\end{document}